 \documentclass[12pt]{amsart}
\usepackage{etex}
 \usepackage{amssymb}
 \usepackage{graphicx}
 \usepackage{lscape}
 \usepackage[all,cmtip]{xy}
 \usepackage{cancel}
 \usepackage{bm}
 \usepackage{multirow}
 \usepackage{framed}
 \usepackage{color}
 \usepackage{hyperref}
 \usepackage{youngtab}  \Yvcentermath1
 \usepackage{underscore}
 \usepackage{tikz}
 \usetikzlibrary{trees}
 \usepackage{ifthen}
 \usepackage{xspace}

\usepackage{longtable}
\usepackage{tabu}
\usepackage{booktabs}
\usepackage[font={normalsize}]{caption}

\usepackage{pifont}
\newcommand{\cmark}{\checkmark}
\newcommand{\xmark}{\mbox{\ding{53}}}

\def\xstringversion     {1.6}
\def\xstringdate        {2012/10/24}

\edef\CurrentAtCatcode  {\the\catcode`\@}
\catcode`\@=11
\newwrite\@xs@message% canal pour les messages
\newcount\integerpart\newcount\decimalpart% compteurs utilis\'es par xstring
\newif\if@xs@empty

\expandafter\ifx\csname @latexerr\endcsname\relax% on n'utilise pas LaTeX ?
	\immediate\write\m@ne{Package: xstring \xstringdate\space\space v\xstringversion\space\space String manipulations (C Tellechea)}%
	\long\def\@firstoftwo#1#2{#1}
	\long\def\@secondoftwo#1#2{#2}
	\long\def\@gobble#1{}
	\long\def\@ifnextchar#1#2#3{%
		\let\reserved@d=#1%
		\def\reserved@a{#2}%
		\def\reserved@b{#3}%
		\futurelet\@let@arg\@ifnch}
	\def\@ifnch{%
		\ifx\@let@arg\@sptoken
			\let\reserved@c\@xifnch
		\else
			\ifx\@let@arg\reserved@d
				\let\reserved@c\reserved@a
			\else
				\let\reserved@c\reserved@b
			\fi
		\fi
		\reserved@c}
	\def\:{\let\@sptoken= } \:
	\def\:{\@xifnch} \expandafter\def\: {\futurelet\@let@arg\@ifnch}
	\def\@ifstar#1{\@ifnextchar *{\@firstoftwo{#1}}}
	\long\def\@testopt#1#2{\@ifnextchar[{#1}{#1[{#2}]}}
	\def\@empty{}
\fi% fin des d\'efinitions LaTeX

\def\@xs@testempty#1{%
	\expandafter\ifx\expandafter\@empty\detokenize{#1}\@empty\@xs@emptytrue\else\@xs@emptyfalse\fi}

\def\@xs@MakeVerb{% lit 1 argument et le transforme en verb
	\begingroup% groupe o\`u les catcodes sont \`a 12 pour la lecture suivante
		\def\do##1{\catcode`##112\relax}%
		\dospecials% on entre dans le mode verb
		\obeyspaces% et on tient compte des espaces
		\@xs@ReadVerb% et on va lire l'argument
}

\def\setverbdelim#1{% d\'efinit quel est le d\'elimiteur de verb
	\expandafter\@xs@testempty\expandafter{\@gobble#1}%
	\if@xs@empty
	\else
		\begingroup
			\newlinechar`\^^J%
			\immediate\write\@xs@message
			{Package xstring Warning: verb delimiter is not a single token on input line \the\inputlineno^^J}%
		\endgroup
	\fi
	\def\@xs@ReadVerb##1#1##2#1{% lit ##2 qui est entre les d\'elimiteurs de verb
		\endgroup% on ferme le groupe
		\@xs@afterreadverb{##2}}% on appelle l'ex\'ecution de fin
}

\def\verbtocs#1{%
	\def\@xs@afterreadverb##1{\def#1{##1}}%
	\@xs@MakeVerb
}

\begingroup% on ouvre un groupe o\`u...
	\catcode\z@12 % ...le caract\`ere 0 a 12 pour catcode
	\gdef\tokenize#1#2{%
		\begingroup
			\@xs@def\@xs@reserved@A{#2}% on d\'eveloppe en accord avec \fullexpandarg ou \noexpandarg
			\def\@xs@AssignResult^^00##1^^00\@xs@nil{\gdef#1{##1}}% on assigne en tenant compte du \@xs@nil qui vient de la fin du fichier virtuel
			\everyeof{\@xs@nil}% met un \@xs@nil \`a la fin du fichier virtuel
			\endlinechar\m@ne
			\catcode\z@12\relax
			\expandafter\@xs@AssignResult\scantokens\expandafter{\expandafter^^00\@xs@reserved@A^^00}% on fait l'assignation
		\endgroup
	}%
\endgroup

\begingroup
	\catcode\z@3 \def\@xs@twochars{^^00}%
	\catcode\z@7 \xdef\@xs@twochars{\@xs@twochars^^00}%
\endgroup

\edef\@xs@reserved@A{\def\noexpand\@xs@AssignResult##1\@xs@twochars}
\@xs@reserved@A#2{\endgroup\expandafter\def\expandafter#2\expandafter{\@gobble#1}}
\def\tokenize#1#2{%
	\begingroup
		\@xs@def\@xs@reserved@A{#2}% on d\'eveloppe en accord avec \fullexpandarg ou \noexpandarg
		\everyeof\expandafter{\@xs@twochars#1}% met "^^@^^@#1" \`a la fin du fichier virtuel
		\endlinechar\m@ne
		\expandafter\@xs@AssignResult\scantokens\expandafter{\expandafter\relax\@xs@reserved@A}% on fait l'assignation
}%

\def\@xs@ReturnResult#1#2{%
	\def\@xs@argument@A{#1}%
	\@xs@testempty{#2}%
	\if@xs@empty
		\@xs@argument@A
	\else
		\let#2\@xs@argument@A
	\fi
}

\def\normalexpandarg{%
	\let\@xs@def\def% on d\'efinit \@xs@call avec \def
	\def\@xs@expand##1{\unexpanded\expandafter{##1}}}
\def\expandarg{%
	\let\@xs@def\def% on d\'efinit \@xs@call avec \def
	\def\@xs@expand##1{\unexpanded\expandafter\expandafter\expandafter{##1}}%
}

\def\fullexpandarg{%
	\let\@xs@def\edef% on d\'efinit\@xs@call avec \edef
	\def\@xs@expand##1{##1}% et on neutralise \@xs@expand
}

\def\saveexpandmode{\let\@xs@saved@def\@xs@defarg\let\@xs@saved@expand\@xs@expand}
\def\restoreexpandmode{\let\@xs@defarg\@xs@saved@def\let\@xs@expand\@xs@saved@expand}

\def\@xs@ifbeginwithbrace#1{%
	\csname @%
		\expandafter\@gobble\string{%
		\expandafter\@gobble\expandafter{\expandafter{\string#1}%
		\expandafter\expandafter\expandafter\expandafter\expandafter\expandafter\expandafter\expandafter\expandafter\expandafter\expandafter\expandafter\expandafter\expandafter\expandafter\@firstoftwo
		\expandafter\expandafter\expandafter\expandafter\expandafter\expandafter\expandafter\@gobble
		\expandafter\expandafter\expandafter\@gobble
		\expandafter\expandafter\expandafter{\expandafter\string\expandafter}\string}%
		\expandafter\@gobble\string}%
		\@secondoftwo{first}{second}oftwo%
	\endcsname
}

\def\@xs@returnfirstsyntaxunit#1#2{%
	\def\@xs@groupfound{\expandafter\def\expandafter#2\expandafter{\expandafter{#2}}\@xs@gobbleall}% on met #2 dans des accolades et on finit
	\def\@xs@assignfirsttok##1##2\@xs@nil{\let\@xs@toks0\def#2{##1}}%
	\def\@xs@testfirsttok{%
		\let\@xs@next\@xs@assignfirsttok
		\ifx\@xs@toks\bgroup
			\expandafter\@xs@ifbeginwithbrace\expandafter{\@xs@argument}{\def\@xs@next{\afterassignment\@xs@groupfound\def#2}}{}%
		\fi
		\@xs@next}%
	\def\@xs@argument{#1}%
	\edef\@xs@next{\expandafter\@xs@beforespace\detokenize{#1} \@xs@nil}% #1 commence par un espace ?
	\ifx\@xs@next\@empty
		\def\@xs@next{\expandafter\ifx\expandafter\@empty\detokenize\expandafter{\@xs@argument}\@empty\let#2\@empty\else\def#2{ }\let\@xs@toks0\fi}%
	\else
		\def\@xs@next{\expandafter\futurelet\expandafter\@xs@toks\expandafter\@xs@testfirsttok\@xs@argument\@xs@nil}%
	\fi
	\@xs@next
}

\def\@xs@testsecondtoken#1\@xs@nil{\@xs@ifbeginwithbrace{#1}}
\def\@xs@gobblespacebeforebrace#1#{}% supprime tout ce qui est avant la 1ere accolade ouvrante
\def\@xs@removefirstsyntaxunit#1#2{%
	\def\@xs@argument{#1}%
	\expandafter\expandafter\expandafter\ifx\expandafter\expandafter\expandafter\@empty\expandafter\@xs@beforespace\detokenize\expandafter{\@xs@argument} \@xs@nil\@empty% #1 commence par un espace ?
		\expandafter\@xs@testempty\expandafter{\@xs@argument}%
		\if@xs@empty
			\let#2\@empty
		\else
			\afterassignment\@xs@testsecondtoken% après avoir mangé le 1er token, on va tester si la suite commence par «{»
			\expandafter\let\expandafter\@xs@secontoken\expandafter=\expandafter\@sptoken\@xs@argument\@xs@@nil\@xs@nil% on mange le 1er token et on rajoute \@xs@@nil à la fin pour éviter de perdre les accolades du groupe
				{\expandafter\expandafter\expandafter\def\expandafter\expandafter\expandafter#2%
				 \expandafter\expandafter\expandafter{\expandafter\@xs@gobblespacebeforebrace\@xs@argument}}%
				{\expandafter\expandafter\expandafter\def\expandafter\expandafter\expandafter#2%
				 \expandafter\expandafter\expandafter{\expandafter\@xs@behindspace\@xs@argument\@xs@nil}}%
		\fi
	\else
		\expandafter\expandafter\expandafter\def\expandafter\expandafter\expandafter#2%
		\expandafter\expandafter\expandafter{\expandafter\@gobble\@xs@argument}%
	\fi
}

\def\@xs@beforespace#1 #2\@xs@nil{#1}
\def\@xs@behindspace#1 #2\@xs@nil{#2}
\def\@xs@returnfirstsyntaxunit@ii#1#2\@xs@nil{#1}
\def\@xs@gobbleall#1\@xs@nil{}

\def\@xs@expand@and@detokenize#1#2{%
	\def#1{#2}%
	\expandafter\edef\expandafter#1\expandafter{\@xs@expand#1}% on d\'eveloppe #2 selon le mode de d\'eveloppement
	\edef#1{\detokenize\expandafter{#1}}% puis on d\'etokenize et on assigne \`a #1
}

\def\@xs@expand@and@assign#1#2{\@xs@def#1{#2}}% on d\'eveloppe #2 selon \fullexpandarg ou \normalexpandarg

\def\@xs@edefaddtomacro#1#2{\edef#1{\unexpanded\expandafter{#1}#2}}
\def\@xs@addtomacro#1#2{\expandafter\def\expandafter#1\expandafter{#1#2}}

\def\@xs@argstring{0########1########2########3########4########5########6########7########8########9}
\def\@xs@DefArg#1{\def\@xs@defarg0##1#1##2\@xs@nil{\def\@xs@myarg{##1#1}}\expandafter\@xs@defarg\@xs@argstring\@xs@nil}
\def\@xs@DefArg@#1{\expandafter\@xs@defarg@\expandafter{\number\numexpr#1+1}}
\def\@xs@defarg@#1{\def\@xs@defarg0##11##2#1##3\@xs@nil{\def\@xs@myarg{[##11]##2#1}}\expandafter\@xs@defarg\@xs@argstring\@xs@nil}
\def\@xs@OneArg#1{\expandafter\@xs@onearg\expandafter{\number\numexpr#1-1}{#1}}
\def\@xs@onearg#1#2{\def\@xs@defarg##1#1##2#2##3\@xs@nil{\def\@xs@myarg{##2#2}}\expandafter\@xs@defarg\@xs@argstring\@xs@nil}

\def\@xs@BuildLines#1#2#3#4{%
	\let\@xs@newlines\@empty
	\let\@xs@newargs\@empty
	\def\@xs@buildlines##1{%
		\expandafter\@xs@OneArg\expandafter{\number\numexpr##1+#1-1}%
		\edef\@xs@reserved@B{\noexpand\@xs@expand\csname @xs@arg@\romannumeral\numexpr##1\endcsname}%
		\ifnum##1=\@ne% si c'est le premier argument
			\@xs@testempty{#3}%
			\if@xs@empty
				\expandafter\@xs@addtomacro\expandafter\@xs@newargs\expandafter{\expandafter{\@xs@reserved@B}}%
				\edef\@xs@reserved@B{\ifnum##1>#4 @xs@def\else @xs@assign\fi}%
			\else% et s'il y a un argument optionnel alors, on met des crochets
				\expandafter\@xs@addtomacro\expandafter\@xs@newargs\expandafter{\expandafter[\@xs@reserved@B]}%
				\def\@xs@reserved@B{@xs@def}% ne pas d\'etok\'eniser l'argument optionnel grace au \@xs@def
			\fi
		\else
			\expandafter\@xs@addtomacro\expandafter\@xs@newargs\expandafter{\expandafter{\@xs@reserved@B}}%
			\edef\@xs@reserved@B{\ifnum##1>#4 @xs@def\else @xs@assign\fi}%
		\fi
		\edef\@xs@newlines{\unexpanded\expandafter{\@xs@newlines}\expandafter\noexpand\csname\@xs@reserved@B\endcsname\expandafter\noexpand\csname @xs@arg@\romannumeral\numexpr##1\endcsname{\@xs@myarg}}%
		\ifnum##1<#2\relax
			\def\@xs@next{\expandafter\@xs@buildlines\expandafter{\number\numexpr##1+1}}%
			\expandafter\@xs@next
		\fi}%
	\@xs@buildlines\@ne
}

\def\@xs@newmacro{%
	\@ifstar
		{\let\@xs@reserved@D\@empty\@xs@newmacro@}
		{\let\@xs@reserved@D\relax\@xs@newmacro@0}%
}
\def\@xs@newmacro@#1#2#3#4#5{%
	\edef\@xs@reserved@A{@xs@\expandafter\@gobble\string#2}%
	\edef\@xs@reserved@C{\expandafter\noexpand\csname\@xs@reserved@A @\ifx\@empty#3\@empty @\fi\endcsname}%
	\edef\@xs@reserved@B{%
		\ifx\@empty\@xs@reserved@D
			\def\noexpand#2{\noexpand\@ifstar
				{\let\noexpand\@xs@assign\noexpand\@xs@expand@and@detokenize\expandafter\noexpand\@xs@reserved@C}%
				{\let\noexpand\@xs@assign\noexpand\@xs@expand@and@assign\expandafter\noexpand\@xs@reserved@C}%
			}%
		\else
			\def\noexpand#2{\let\noexpand\@xs@assign\noexpand\@xs@expand@and@assign\expandafter\noexpand\@xs@reserved@C}%
		\fi
		\ifx\@empty#3\@empty
		\else
			\def\expandafter\noexpand\@xs@reserved@C{%
				\noexpand\@testopt{\expandafter\noexpand\csname\@xs@reserved@A @@\endcsname}{\ifx\@xs@def\edef#3\else\unexpanded{#3}\fi}}%
		\fi
	}%
	\@xs@reserved@B
	\ifx\@empty#3\@empty
		\@xs@BuildLines1{#4}{#3}{#1}%
		\@xs@DefArg{#4}%
	\else
		\expandafter\@xs@BuildLines\expandafter1\expandafter{\number\numexpr#4+1}{#3}{#1}%
		\@xs@DefArg@{#4}%
	\fi
	\edef\@xs@reserved@B{\def\expandafter\noexpand\csname\@xs@reserved@A @@\endcsname\@xs@myarg}%
	\edef\@xs@reserved@C{\unexpanded\expandafter{\@xs@newlines}\edef\noexpand\@xs@call}%
	\edef\@xs@reserved@D{%
		\noexpand\noexpand\expandafter\noexpand\csname\@xs@reserved@A\endcsname\unexpanded\expandafter{\@xs@newargs}%
	}%
	\ifnum#5=\@ne\edef\@xs@reserved@D{\noexpand\noexpand\noexpand\@testopt{\unexpanded\expandafter{\@xs@reserved@D}}{}}\fi
	\@xs@edefaddtomacro\@xs@reserved@C{{\unexpanded\expandafter{\@xs@reserved@D}}\noexpand\@xs@call}%
	\@xs@edefaddtomacro\@xs@reserved@B{{\unexpanded\expandafter{\@xs@reserved@C}}}%
	\@xs@reserved@B
	\edef\@xs@reserved@B{%
		\def\expandafter\noexpand\csname\@xs@reserved@A\endcsname
			\@xs@myarg\ifnum#5=\@ne[\unexpanded{##}\number\numexpr\ifx\@empty#3\@empty#4+1\else#4+2\fi]\fi
	}%
	\@xs@reserved@B
}

\def\@xs@read@reserved@C{%
	\expandafter\@xs@testempty\expandafter{\@xs@reserved@C}%
	\if@xs@empty
		\ifnum\@xs@nestlevel=\z@
			\let\@xs@next\relax
		\else
			\let\@xs@next\@xs@atendofgroup
		\fi
	\else
		\expandafter\@xs@returnfirstsyntaxunit\expandafter{\@xs@reserved@C}\@xs@reserved@A
		\expandafter\@xs@removefirstsyntaxunit\expandafter{\@xs@reserved@C}\@xs@reserved@C
		\let\@xs@next\@xs@read@reserved@C
		\@xs@exploregroups
		\ifx\bgroup\@xs@toks
			\advance\integerpart\@ne
			\begingroup
				\expandafter\def\expandafter\@xs@reserved@C\@xs@reserved@A
				\@xs@manage@groupID
				\let\@xs@nestlevel\@ne
				\integerpart\z@
				\@xs@atbegingroup
		\else
			\global\advance\decimalpart\@ne
			\@xs@atnextsyntaxunit
		\fi
	\fi
	\@xs@next
}

\def\@xs@read@reserved@D{%
	\expandafter\@xs@testempty\expandafter{\@xs@reserved@D}%
	\if@xs@empty
		\ifnum\@xs@nestlevel=\z@
			\let\@xs@next\relax
		\else
			\let\@xs@next\@xs@atendofgroup
		\fi
	\else
		\expandafter\expandafter\expandafter\@xs@IfBeginWith@i\expandafter\expandafter\expandafter{\expandafter\@xs@reserved@D\expandafter}\expandafter{\@xs@reserved@E}%
			{\global\advance\decimalpart\@ne
			\let\@xs@reserved@D\@xs@reserved@A
			\@xs@atoccurfound
			}%
			{\expandafter\@xs@returnfirstsyntaxunit\expandafter{\@xs@reserved@D}\@xs@reserved@A
			\expandafter\@xs@removefirstsyntaxunit\expandafter{\@xs@reserved@D}\@xs@reserved@D
			\let\@xs@next\@xs@read@reserved@D
			\@xs@exploregroups
			\ifx\bgroup\@xs@toks
				\advance\integerpart\@ne
				\begingroup
					\expandafter\def\expandafter\@xs@reserved@D\@xs@reserved@A
					\@xs@manage@groupID
					\let\@xs@reserved@C\@empty
					\let\@xs@nestlevel\@ne
					\integerpart\z@
			\else
				\expandafter\@xs@addtomacro\expandafter\@xs@reserved@C\expandafter{\@xs@reserved@A}%
			\fi
			}%
	\fi
	\@xs@next
}

\@xs@newmacro\StrRemoveBraces{}{1}{1}{%
	\def\@xs@reserved@C{#1}%
	\let\@xs@reserved@B\@empty
	\let\@xs@nestlevel\z@
	\@xs@StrRemoveBraces@i
	\expandafter\@xs@ReturnResult\expandafter{\@xs@reserved@B}{#2}%
}

\def\@xs@StrRemoveBraces@i{%
	\expandafter\@xs@testempty\expandafter{\@xs@reserved@C}%
	\if@xs@empty
		\ifnum\@xs@nestlevel=\z@
			\let\@xs@next\relax
		\else
			\expandafter\endgroup
			\expandafter\@xs@addtomacro\expandafter\@xs@reserved@B\expandafter{\@xs@reserved@B}%
			\let\@xs@next\@xs@StrRemoveBraces@i
		\fi
	\else
		\expandafter\@xs@returnfirstsyntaxunit\expandafter{\@xs@reserved@C}\@xs@reserved@A
		\expandafter\@xs@removefirstsyntaxunit\expandafter{\@xs@reserved@C}\@xs@reserved@C
		\let\@xs@next\@xs@StrRemoveBraces@i
		\ifx\bgroup\@xs@toks
			\ifx\@xs@exploregroups\relax% on explore les groupes ?
				\begingroup
					\expandafter\def\expandafter\@xs@reserved@C\@xs@reserved@A
					\let\@xs@nestlevel\@ne
					\integerpart\z@
					\let\@xs@reserved@B\@empty
			\else
				\expandafter\@xs@addtomacro\expandafter\@xs@reserved@B\@xs@reserved@A
			\fi
		\else
			\expandafter\@xs@addtomacro\expandafter\@xs@reserved@B\expandafter{\@xs@reserved@A}%
		\fi
	\fi
	\@xs@next
}

\def\@xs@cutafteroccur#1#2#3{%
	\ifnum#3<\@ne\expandafter\@firstoftwo\else\expandafter\@secondoftwo\fi
		{\let\@xs@reserved@C\@empty\let\@xs@reserved@E\@empty\let\groupID\@empty}
		{\@xs@cutafteroccur@i{#1}{#2}{#3}}%
}
	
\def\@xs@cutafteroccur@i#1#2#3{%
	\def\@xs@reserved@D{#1}\let\@xs@reserved@C\@empty\def\@xs@reserved@E{#2}%
	\decimalpart\z@\integerpart\z@\def\groupID{0}\let\@xs@nestlevel\z@
	\def\@xs@atendofgroup{%
		\expandafter\endgroup
		\expandafter\@xs@addtomacro\expandafter\@xs@reserved@C\expandafter{\expandafter{\@xs@reserved@C}}%
		\@xs@read@reserved@D}%
	\def\@xs@atoccurfound{%
		\ifnum\decimalpart=\numexpr(#3)\relax
			\global\let\@xs@reserved@D\@xs@reserved@D
			\global\let\@xs@reserved@C\@xs@reserved@C
			\global\let\groupID\groupID
			\@xs@exitallgroups
			\let\@xs@next\relax
		\else
			\expandafter\@xs@addtomacro\expandafter\@xs@reserved@C\expandafter{\@xs@reserved@E}%
			\let\@xs@next\@xs@read@reserved@D
		\fi}%
	\@xs@read@reserved@D
	\def\@xs@argument@A{#2}%
	\ifnum\decimalpart=\numexpr(#3)\relax % occurrence trouv\'ee ?
		\let\@xs@reserved@E\@xs@reserved@D
		\expandafter\expandafter\expandafter\def\expandafter\expandafter\expandafter\@xs@reserved@D\expandafter\expandafter\expandafter{\expandafter\@xs@reserved@C\@xs@argument@A}%
	\else
		\let\@xs@reserved@C\@empty\let\@xs@reserved@E\@empty\let\groupID\@empty
	\fi
}

\@xs@newmacro*3\IfSubStr{1}{2}{0}{%
	\def\@xs@argument@A{#2}\def\@xs@argument@B{#3}%
	\expandafter\expandafter\expandafter\@xs@cutafteroccur
	\expandafter\expandafter\expandafter{\expandafter\@xs@argument@A\expandafter}\expandafter{\@xs@argument@B}{#1}%
	\expandafter\@xs@testempty\expandafter{\@xs@reserved@D}%
	\if@xs@empty
		\expandafter\@secondoftwo
	\else
		\expandafter\@firstoftwo
	\fi
}

\@xs@newmacro*2\IfBeginWith{}{2}{0}{%
	\def\@xs@argument@A{#1}\def\@xs@argument@B{#2}%
	\expandafter\@xs@testempty\expandafter{\@xs@argument@B}%
	\if@xs@empty
		\let\@xs@next\@secondoftwo
	\else
		\def\@xs@next{\expandafter\expandafter\expandafter\@xs@IfBeginWith@i
			\expandafter\expandafter\expandafter{\expandafter\@xs@argument@A\expandafter}\expandafter{\@xs@argument@B}}%
	\fi
	\@xs@next
}

\def\@xs@IfBeginWith@i#1#2{%
	\def\@xs@argument@A{#1}\def\@xs@argument@B{#2}%
	\expandafter\@xs@testempty\expandafter{\@xs@argument@B}%
	\if@xs@empty% #2 est vide, tous les tests sont pass\'es avec succ\`es : on renvoie #3
		\let\@xs@next\@firstoftwo
	\else
		\expandafter\@xs@testempty\expandafter{\@xs@argument@A}%\@xs@testempty{#1}%
		\if@xs@empty
			\let\@xs@next\@secondoftwo% #1 est vide, c'est que #2 est + long que #1 : on renvoie #4
		\else
			\expandafter\@xs@returnfirstsyntaxunit\expandafter{\@xs@argument@B}\@xs@reserved@B
			\expandafter\@xs@returnfirstsyntaxunit\expandafter{\@xs@argument@A}\@xs@reserved@A
			\ifx\@xs@reserved@A\@xs@reserved@B% il y a \'egalit\'e...
				\expandafter\@xs@removefirstsyntaxunit\expandafter{\@xs@argument@B}\@xs@reserved@B
				\expandafter\@xs@removefirstsyntaxunit\expandafter{\@xs@argument@A}\@xs@reserved@A% on enl\`eve les 1ere unit\'es syntaxiques
				\def\@xs@next{% et on recommence avec ces arguments racourcis d'1 unit\'e syntaxique
					\expandafter\expandafter\expandafter\@xs@IfBeginWith@i
					\expandafter\expandafter\expandafter{\expandafter\@xs@reserved@A\expandafter}\expandafter{\@xs@reserved@B}}%
			\else
				\let\@xs@next\@secondoftwo
			\fi
		\fi
	\fi
	\@xs@next
}

\@xs@newmacro*2\IfEndWith{}{2}{0}{%
	\def\@xs@argument@A{#1}\def\@xs@argument@B{#2}%
	\@xs@testempty{#2}%
	\if@xs@empty
		\let\@xs@reserved@A\@secondoftwo
	\else
		\expandafter\expandafter\expandafter\@xs@StrCount
			\expandafter\expandafter\expandafter{\expandafter\@xs@argument@A\expandafter}\expandafter
			{\@xs@argument@B}[\@xs@reserved@A]%
		\ifnum\@xs@reserved@A=\z@
			\let\@xs@reserved@A\@secondoftwo
		\else
			\expandafter\@xs@testempty\expandafter{\@xs@reserved@C}%
			\if@xs@empty
				\let\@xs@reserved@A\@firstoftwo
			\else
				\let\@xs@reserved@A\@secondoftwo
			\fi
		\fi
	\fi
	\@xs@reserved@A
}

\@xs@newmacro*4\IfSubStrBefore{1,1}{3}{0}{\@xs@IfSubStrBefore@i[#1]{#2}{#3}{#4}}
\def\@xs@IfSubStrBefore@i[#1,#2]#3#4#5{%
	\def\@xs@reserved@C{#3}%
	\ifx\@xs@exploregroups\relax% si on explore les groupes
		\let\@xs@reserved@B\@empty
		\let\@xs@nestlevel\z@
		\@xs@StrRemoveBraces@i% on retire les accolades
		\let\@xs@reserved@C\@xs@reserved@B
	\fi
	\def\@xs@reserved@A{#5}%
	\expandafter\expandafter\expandafter\@xs@cutafteroccur\expandafter\expandafter\expandafter{\expandafter\@xs@reserved@C\expandafter}\expandafter{\@xs@reserved@A}{#2}%
	\def\@xs@reserved@A{#4}%
	\expandafter\expandafter\expandafter\@xs@cutafteroccur\expandafter\expandafter\expandafter{\expandafter\@xs@reserved@C\expandafter}\expandafter{\@xs@reserved@A}{#1}%
	\let\groupID\@empty
	\expandafter\@xs@testempty\expandafter{\@xs@reserved@C}%
	\if@xs@empty
		\expandafter\@secondoftwo
	\else
		\expandafter\@firstoftwo
	\fi
}

\@xs@newmacro*4\IfSubStrBehind{1,1}{3}{0}{\@xs@IfSubStrBehind@i[#1]{#2}{#3}{#4}}
\def\@xs@IfSubStrBehind@i[#1,#2]#3#4#5{\@xs@IfSubStrBefore@i[#2,#1]{#3}{#5}{#4}}

\def\@xs@formatnumber#1#2{%
	\def\@xs@argument@A{#1}%
	\@xs@testempty{#1}%
	\if@xs@empty
		\def#2{0X}% si vide, renvoie 0X
	\else
		\@xs@returnfirstsyntaxunit{#1}\@xs@reserved@A
		\def\@xs@reserved@B{+}%
		\ifx\@xs@reserved@A\@xs@reserved@B
			\expandafter\@xs@removefirstsyntaxunit\expandafter{\@xs@argument@A}\@xs@reserved@C
			\expandafter\@xs@testempty\expandafter{\@xs@reserved@C}%
			\if@xs@empty
			 	\def#2{+0X}%
			 \else
			 	\expandafter\def\expandafter#2\expandafter{\expandafter+\expandafter0\@xs@reserved@C}%
			 \fi
		\else
			\def\@xs@reserved@B{-}%
			\ifx\@xs@reserved@A\@xs@reserved@B
				\expandafter\@xs@removefirstsyntaxunit\expandafter{\@xs@argument@A}\@xs@reserved@A
				\expandafter\@xs@testempty\expandafter{\@xs@reserved@A}%
				\if@xs@empty
				 	\def#2{-0X}%
				 \else
				 	\expandafter\def\expandafter#2\expandafter{\expandafter-\expandafter0\@xs@reserved@A}%
				 \fi
			\else
				\expandafter\def\expandafter#2\expandafter{\expandafter0\@xs@argument@A}%
			\fi
		\fi
	\fi
}

\@xs@newmacro\IfInteger{}{1}{0}{%
	\@xs@formatnumber{#1}\@xs@reserved@A
	\decimalpart\z@
	\afterassignment\@xs@defafterinteger\integerpart\@xs@reserved@A\relax\@xs@nil
	\let\@xs@after@intpart\@xs@afterinteger
	\expandafter\@xs@testdot\@xs@afterinteger\@xs@nil
	\ifx\@empty\@xs@afterdecimal
		\ifnum\decimalpart=\z@
			\let\@xs@next\@firstoftwo% partie décimale constituée de 0 --> seul cas où on renvoie vrai
		\else
			\let\@xs@afterinteger\@xs@after@intpart
			\let\@xs@next\@secondoftwo
		\fi
	\else
		\let\@xs@afterinteger\@xs@after@intpart
		\let\@xs@next\@secondoftwo
	\fi
	\@xs@next
}

\@xs@newmacro\IfDecimal{}{1}{0}{%
	\@xs@formatnumber{#1}\@xs@reserved@A
	\decimalpart\z@
	\afterassignment\@xs@defafterinteger\integerpart\@xs@reserved@A\relax\@xs@nil
	\expandafter\@xs@testdot\@xs@afterinteger\@xs@nil
	\ifx\@empty\@xs@afterdecimal
		\expandafter\@firstoftwo
	\else
		\expandafter\@secondoftwo
	\fi
}

\def\@xs@defafterinteger#1\relax\@xs@nil{\def\@xs@afterinteger{#1}}
\def\@xs@testdot{%
	\let\xs@decsep\@empty
	\@ifnextchar.%
		{\def\xs@decsep{.}\@xs@readdecimalpart}%
		{\@xs@testcomma}%
}
\def\@xs@testcomma{%
	\@ifnextchar,%
		{\def\xs@dessep{,}\@xs@readdecimalpart}%
		{\@xs@endnumber}%
}
\def\@xs@readdecimalpart#1#2\@xs@nil{%
	\ifx\@empty#2\@empty
		\def\@xs@reserved@A{0X}%
	\else
		\def\@xs@reserved@A{0#2}%
	\fi
	\afterassignment\@xs@defafterinteger\decimalpart\@xs@reserved@A\relax\@xs@nil
	\expandafter\@xs@endnumber\@xs@afterinteger\@xs@nil
}
\def\@xs@endnumber#1\@xs@nil{\def\@xs@afterdecimal{#1}}

\def\@xs@IfStrEqFalse@i#1#2{\let\@xs@reserved@A\@secondoftwo}
\def\@xs@IfStrEqFalse@ii#1#2{% renvoie vrai si les 2 arg sont d\'ecimaux et s'ils sont \'egaux, faux sinon
	\@xs@IfDecimal{#1}%
		{\@xs@IfDecimal{#2}%
			{\ifdim#1pt=#2pt
				\let\@xs@reserved@A\@firstoftwo
			\else
				\let\@xs@reserved@A\@secondoftwo
			\fi% les 2 sont d\'ecimaux : on fait le test
			}%
			{\let\@xs@reserved@A\@secondoftwo}% un seul est d\'ecimal
		}%
		{\let\@xs@reserved@A\@secondoftwo}% #1 n'est pas d\'ecimal
}

\def\@xs@TestEqual#1#2{% teste si les 2 arguments sont \'egaux
	\def\@xs@reserved@A{#1}\def\@xs@reserved@B{#2}%
	\ifx\@xs@reserved@A\@xs@reserved@B
		\let\@xs@reserved@A\@firstoftwo% \'egalit\'e parfaite des 2 chaines
	\else
		\expandafter\expandafter\expandafter\@xs@reserved@D\expandafter\expandafter\expandafter{\expandafter\@xs@reserved@A\expandafter}\expandafter{\@xs@reserved@B}%
	\fi
	\@xs@reserved@A
}

\@xs@newmacro*2\IfStrEq{}{2}{0}{% teste si les deux chaines sont \'egales
	\let\@xs@reserved@D\@xs@IfStrEqFalse@i
	\@xs@TestEqual{#1}{#2}%
}

\@xs@newmacro*2\IfEq{}{2}{0}{% teste si les 2 arguments (chaine ou nombre) sont \'egaux
	\let\@xs@reserved@D\@xs@IfStrEqFalse@ii
	\@xs@TestEqual{#1}{#2}%
}

\def\IfStrEqCase{%
	\@ifstar
		{\def\@xs@reserved@E{\IfStrEq*}\@xs@IfStrCase}%
		{\def\@xs@reserved@E{\IfStrEq}\@xs@IfStrCase}%
}
\def\@xs@IfStrCase#1#2{\@testopt{\@xs@IfStringCase{#1}{#2}}{}}

\def\IfEqCase{%
	\@ifstar
		{\def\@xs@reserved@E{\IfEq*}\@xs@IfEqCase}%
		{\def\@xs@reserved@E{\IfEq}\@xs@IfEqCase}%
}
\def\@xs@IfEqCase#1#2{\@testopt{\@xs@IfStringCase{#1}{#2}}{}}

\def\@xs@IfStringCase#1#2[#3]{%
	\def\@xs@testcase##1##2##3\@xs@nil{% lit les 2 premieres unit\'es syntaxiques dans ##1 et ##2. Les autres dans ##3
		\@xs@reserved@E{#1}{##1}%
			{##2}% le test est positif, on ex\'ecute le code correspondant
			{\@xs@testempty{##3}%
			 \if@xs@empty% le test est n\'egatif
			 	\def\@xs@next{#3}% s'il n'y a plus de cas, on ex\'ecute le code <autres cas>
			 \else
			 	\def\@xs@next{\@xs@testcase##3\@xs@nil}% sinon, on recommence avec ce qui reste
			 \fi
			 \@xs@next
			 }%
	}%
	\@xs@testcase#2\@xs@nil
}

\@xs@newmacro*3\StrBefore{1}{2}{1}{%
	\@xs@cutafteroccur{#2}{#3}{#1}%
	\expandafter\@xs@ReturnResult\expandafter{\@xs@reserved@C}{#4}%
}

\@xs@newmacro*3\StrBehind{1}{2}{1}{%
	\@xs@cutafteroccur{#2}{#3}{#1}%
	\expandafter\@xs@ReturnResult\expandafter{\@xs@reserved@E}{#4}%
}

\@xs@newmacro*4\StrBetween{1,1}{3}{1}{\@xs@StrBetween@i[#1]{#2}{#3}{#4}[#5]}
\def\@xs@StrBetween@i[#1,#2]#3#4#5[#6]{%
	\begingroup
		\noexploregroups
		\@xs@cutafteroccur{#3}{#5}{#2}%
		\expandafter\@xs@cutafteroccur\expandafter{\@xs@reserved@C}{#4}{#1}%
		\expandafter
	\endgroup
	\expandafter\@xs@ReturnResult\expandafter{\@xs@reserved@E}{#6}%
	\let\groupID\@empty
}

\def\exploregroups{\let\@xs@exploregroups\relax}
\def\noexploregroups{\def\@xs@exploregroups{\let\@xs@toks0\relax}}
\def\saveexploremode{\let\@xs@saveexploremode\@xs@exploregroups}
\def\restoreexploremode{\let\@xs@exploregroups\@xs@saveexploremode}

\@xs@newmacro\StrSubstitute{0}{3}{1}{%
	\def\@xs@reserved@D{#2}\let\@xs@reserved@C\@empty\def\@xs@reserved@E{#3}%
	\def\@xs@argument@C{#3}\def\@xs@argument@D{#4}%
	\decimalpart\z@\integerpart\z@\def\groupID{0}\let\@xs@nestlevel\z@
	\def\@xs@atendofgroup{%
		\expandafter\endgroup
		\expandafter\@xs@addtomacro\expandafter\@xs@reserved@C\expandafter{\expandafter{\@xs@reserved@C}}%
		\@xs@read@reserved@D
	}%
	\def\@xs@atoccurfound{%
		\ifnum#1<\@ne
			\let\@xs@reserved@A\@xs@argument@D
		\else
			\ifnum\decimalpart>#1
				\let\@xs@reserved@A\@xs@argument@C
			\else
				\let\@xs@reserved@A\@xs@argument@D
			\fi
		\fi
		\expandafter\@xs@addtomacro\expandafter\@xs@reserved@C\expandafter{\@xs@reserved@A}%
		\@xs@read@reserved@D
	}%
	\@xs@testempty{#3}%
	\if@xs@empty
		\expandafter\@xs@ReturnResult\expandafter{\@xs@reserved@D}{#5}%
	\else
		\@xs@read@reserved@D
		\expandafter\@xs@ReturnResult\expandafter{\@xs@reserved@C}{#5}%
	\fi
}

\@xs@newmacro\StrDel{0}{2}{1}{\@xs@StrSubstitute[#1]{#2}{#3}{}[#4]}

\def\@xs@exitallgroups{\ifnum\@xs@nestlevel>\z@\endgroup\expandafter\@xs@exitallgroups\fi}

\@xs@newmacro\StrLen{}{1}{1}{%
	\def\@xs@reserved@C{#1}%
	\decimalpart\z@
	\let\@xs@nestlevel\z@
	\def\groupID{0}%
	\let\@xs@atbegingroup\relax
	\def\@xs@atendofgroup{\endgroup\@xs@read@reserved@C}%
	\let\@xs@atnextsyntaxunit\relax
	\@xs@read@reserved@C
	\expandafter\@xs@ReturnResult\expandafter{\number\decimalpart}{#2}%
}

\def\@xs@continuetonext{%
	\expandafter\@xs@testempty\expandafter{\@xs@reserved@C}%
	\if@xs@empty
		\ifnum\@xs@nestlevel>\z@
			\expandafter\endgroup\expandafter\@xs@addtomacro\expandafter\@xs@reserved@B\expandafter{\expandafter{\@xs@reserved@B}}% on concat\`ene
			\expandafter\expandafter\expandafter\@xs@continuetonext% on recommence
		\fi
	\fi
}%

\def\@xs@manage@groupID{%
	\begingroup\def\@xs@reserved@A{0}%
	\ifx\@xs@reserved@A\groupID% si on arrive du groupe de niveau 0
		\endgroup\edef\groupID{\number\integerpart}% on met simplement le niveau
	\else
		\endgroup\expandafter\@xs@addtomacro\expandafter\groupID\expandafter{\expandafter,\number\integerpart}%
	\fi
}

\def\StrSplit{%
	\@ifstar
		{\let\@xs@reserved@E\@xs@continuetonext\StrSpl@t}%
		{\let\@xs@reserved@E\relax\StrSpl@t}%
}
\@xs@newmacro\StrSpl@t{}{2}{0}{\@xs@StrSplit@i{#2}{#1}\@xs@StrSplit@ii}
\def\@xs@StrSplit@i#1#2{%
	\def\@xs@reserved@D{#1}\def\@xs@reserved@C{#2}\let\@xs@reserved@B\@empty\let\groupID\@empty
	\ifnum#1>\z@
		\decimalpart\z@\integerpart\z@\def\groupID{0}\let\@xs@nestlevel\z@
		\def\@xs@atendofgroup{%
			\expandafter\endgroup
			\expandafter\@xs@addtomacro\expandafter\@xs@reserved@B\expandafter{\expandafter{\@xs@reserved@B}}%
			\@xs@read@reserved@C
		}%
		\def\@xs@atbegingroup{\let\@xs@reserved@B\@empty}%
		\def\@xs@atnextsyntaxunit{%
			\expandafter\@xs@addtomacro\expandafter\@xs@reserved@B\expandafter{\@xs@reserved@A}%
			\ifnum\decimalpart=\@xs@reserved@D\relax
				\ifx\@xs@reserved@C\@empty\@xs@reserved@E\fi
				\global\let\@xs@reserved@B\@xs@reserved@B
				\global\let\@xs@reserved@C\@xs@reserved@C
				\global\let\groupID\groupID
				\@xs@exitallgroups
				\let\@xs@next\relax
			\fi
		}%
		\@xs@read@reserved@C
	\fi
}
\def\@xs@StrSplit@ii#1#2{\let#1\@xs@reserved@B\let#2\@xs@reserved@C}

\@xs@newmacro*3\StrCut{1}{2}{0}{%
	\@xs@testempty{#3}%
	\if@xs@empty\expandafter\@firstoftwo\else\expandafter\@secondoftwo\fi
		{\let\groupID\@empty
		\let\@xs@reserved@C\@empty
		\let\@xs@reserved@E\@empty
		}
		{\ifnum#1<\@ne\expandafter\@firstoftwo\else\expandafter\@secondoftwo\fi
			{\@xs@StrCut@ii{#2}{#3}1}
			{\@xs@StrCut@ii{#2}{#3}{#1}}%
		}%
	\@xs@StrCut@iii
}

\def\@xs@StrCut@ii#1#2#3{%
	\def\@xs@reserved@D{#1}%
	\let\@xs@reserved@C\@empty
	\def\@xs@reserved@E{#2}%
	\decimalpart\z@\integerpart\z@
	\def\groupID{0}%
	\let\@xs@nestlevel\z@
	\def\@xs@atendofgroup{%
		\expandafter\endgroup
		\expandafter\@xs@addtomacro\expandafter\@xs@reserved@C\expandafter{\expandafter{\@xs@reserved@C}}%
		\@xs@read@reserved@D
	}%
	\def\@xs@atoccurfound{%
		\ifnum\decimalpart=\numexpr(#3)\relax
			\global\let\@xs@reserved@D\@xs@reserved@D
			\global\let\@xs@reserved@C\@xs@reserved@C
			\global\let\groupID\groupID
			\@xs@exitallgroups
			\let\@xs@next\relax
		\else
				\expandafter\@xs@addtomacro\expandafter\@xs@reserved@C\expandafter{\@xs@reserved@E}%
			\let\@xs@next\@xs@read@reserved@D
		\fi
	}%
	\@xs@read@reserved@D
	\def\@xs@argument@A{#2}%
	\let\@xs@reserved@E\@xs@reserved@D
		\expandafter\expandafter\expandafter
	\def
		\expandafter\expandafter\expandafter
	\@xs@reserved@D
		\expandafter\expandafter\expandafter
	{\expandafter\@xs@reserved@C\@xs@argument@A}%
}

\def\@xs@StrCut@iii#1#2{\let#1\@xs@reserved@C\let#2\@xs@reserved@E}

\@xs@newmacro\StrMid{}{3}{1}{%
	\begingroup
		\noexploregroups
		\let\@xs@reserved@E\relax
		\@xs@StrSplit@i{#3}{#1}%
		\edef\@xs@reserved@C{\number\numexpr#2-1}%
		\let\@xs@reserved@E\relax
		\expandafter\expandafter\expandafter\@xs@StrSplit@i\expandafter\expandafter\expandafter{\expandafter\@xs@reserved@C\expandafter}\expandafter{\@xs@reserved@B}%
	\expandafter\endgroup
	\expandafter\@xs@ReturnResult\expandafter{\@xs@reserved@C}{#4}%
	\let\groupID\@empty
}

\@xs@newmacro\StrGobbleLeft{}{2}{1}{%
	\let\@xs@reserved@E\relax
	\@xs@StrSplit@i{#2}{#1}%
	\expandafter\@xs@ReturnResult\expandafter{\@xs@reserved@C}{#3}%
}

\@xs@newmacro\StrLeft{}{2}{1}{%
	\let\@xs@reserved@E\relax
	\@xs@StrSplit@i{#2}{#1}%
	\expandafter\@xs@ReturnResult\expandafter{\@xs@reserved@B}{#3}%
}

\@xs@newmacro\StrGobbleRight{}{2}{1}{%
	\@xs@StrLen{#1}[\@xs@reserved@D]%
	\let\@xs@reserved@E\relax
	\expandafter\@xs@StrSplit@i\expandafter{\number\numexpr\@xs@reserved@D-#2}{#1}%
	\expandafter\@xs@ReturnResult\expandafter{\@xs@reserved@B}{#3}%
}

\@xs@newmacro\StrRight{}{2}{1}{%
	\@xs@StrLen{#1}[\@xs@reserved@D]%
	\let\@xs@reserved@E\relax
	\expandafter\@xs@StrSplit@i\expandafter{\number\numexpr\@xs@reserved@D-#2}{#1}%
	\expandafter\@xs@ReturnResult\expandafter{\@xs@reserved@C}{#3}%
}

\@xs@newmacro\StrChar{}{2}{1}{%
	\let\@xs@reserved@B\@empty
	\def\@xs@reserved@C{#1}\def\@xs@reserved@D{#2}%
	\ifnum#2>\z@
		\def\groupID{0}\let\@xs@nestlevel\z@\integerpart\z@\decimalpart\z@
		\let\@xs@atbegingroup\relax
		\def\@xs@atendofgroup{\endgroup\@xs@read@reserved@C}%
		\def\@xs@atnextsyntaxunit{%
			\ifnum\decimalpart=\@xs@reserved@D% la n i\`eme US est atteinte ?
				\global\let\@xs@reserved@B\@xs@reserved@A% on capture l'US en cours qui est celle cherch\'ee
				\global\let\groupID\groupID
				\@xs@exitallgroups
				\let\@xs@next\relax
			\fi
		}%
		\@xs@read@reserved@C
	\fi
	\expandafter\@xs@testempty\expandafter{\@xs@reserved@B}%
	\if@xs@empty\let\groupID\@empty\fi
	\expandafter\@xs@ReturnResult\expandafter{\@xs@reserved@B}{#3}%
}

\@xs@newmacro\StrCount{}{2}{1}{%
	\@xs@testempty{#2}%
	\def\@xs@reserved@D{#1}\def\@xs@reserved@E{#2}\let\@xs@reserved@C\@empty
	\if@xs@empty
		\@xs@ReturnResult{0}{#3}%
	\else
		\decimalpart\z@\integerpart\z@\def\groupID{0}\let\@xs@nestlevel\z@
		\def\@xs@atendofgroup{%
			\expandafter\endgroup
			\expandafter\@xs@addtomacro\expandafter\@xs@reserved@C\expandafter{\expandafter{\@xs@reserved@C}}%
			\@xs@read@reserved@D
		}%
		\def\@xs@atoccurfound{\let\@xs@reserved@C\@empty\@xs@read@reserved@D}%
		\@xs@read@reserved@D
		\expandafter\@xs@ReturnResult\expandafter{\number\decimalpart}{#3}%
	\fi
}

\@xs@newmacro\StrPosition{1}{2}{1}{%
	\@xs@cutafteroccur{#2}{#3}{#1}%
	\let\@xs@reserved@E\groupID
	\ifx\@xs@reserved@C\@xs@reserved@D
		\@xs@ReturnResult{0}{#4}%
		\let\@xs@reserved@E\@empty
	\else
		\expandafter\@xs@StrLen\expandafter{\@xs@reserved@C}[\@xs@reserved@C]%
		\expandafter\@xs@ReturnResult\expandafter{\number\numexpr\@xs@reserved@C+1}{#4}%
	\fi
	\let\groupID\@xs@reserved@E
}

\def\comparestrict{\let\@xs@comparecoeff\@ne}
\def\comparenormal{\let\@xs@comparecoeff\z@}
\def\savecomparemode{\let\@xs@saved@comparecoeff\@xs@comparecoeff}
\def\restorecomparemode{\let\@xs@comparecoeff\@xs@saved@comparecoeff}
\@xs@newmacro*2\StrCompare{}{2}{1}{%
	\def\@xs@reserved@A{#1}%
	\def\@xs@reserved@B{#2}%
	\ifx\@xs@reserved@B\@xs@reserved@A
		\@xs@ReturnResult{0}{#3}%
	\else
		\def\@xs@next{\@xs@StrCompare@i{#1}{#2}{#3}}%
		\expandafter\@xs@next
	\fi
}

\def\@xs@StrCompare@i#1#2#3{%
	\def\@xs@StrCompare@iii##1{%
		\let\@xs@reserved@A\@empty
		\expandafter\@xs@testempty\expandafter{\@xs@reserved@C}%
		\if@xs@empty
			\def\@xs@reserved@A{*\@xs@comparecoeff}%
		\else
			\expandafter\@xs@testempty\expandafter{\@xs@reserved@D}%
			\if@xs@empty
				\def\@xs@reserved@A{*\@xs@comparecoeff}%
			\fi
		\fi
		\def\@xs@next{%
			\expandafter\@xs@ReturnResult\expandafter
			{\number\numexpr##1\@xs@reserved@A\relax}{#3}%
		}%
	}%
	\def\@xs@StrCompare@ii##1{% ##1 est la position
		\expandafter\@xs@returnfirstsyntaxunit\expandafter{\@xs@reserved@C}\@xs@reserved@A
		\expandafter\@xs@returnfirstsyntaxunit\expandafter{\@xs@reserved@D}\@xs@reserved@B
		\ifx\@xs@reserved@B\@xs@reserved@A
			\expandafter\@xs@testempty\expandafter{\@xs@reserved@A}%
			\if@xs@empty
				\@xs@StrCompare@iii{##1}% les 2 unit\'es syntaxiques sont \'egales, on renvoie la position
			\else
				\def\@xs@next{\expandafter\@xs@StrCompare@ii\expandafter{\number\numexpr##1+1}}% les 2 unit\'es syntaxiques sont \'egales et non vides, on recommence
				\expandafter\@xs@removefirstsyntaxunit\expandafter{\@xs@reserved@C}\@xs@reserved@C
				\expandafter\@xs@removefirstsyntaxunit\expandafter{\@xs@reserved@D}\@xs@reserved@D
			\fi
		\else% les 2 unit\'es syntaxiques sont diff\'erentes : on renvoie la position
			\@xs@StrCompare@iii{##1}%
		\fi
		\@xs@next
	}%
	\def\@xs@reserved@C{#1}\def\@xs@reserved@D{#2}%
	\@xs@StrCompare@ii1%
}

\@xs@newmacro\StrFindGroup{}{2}{1}{%
	\def\@xs@reserved@A{#2}\def\@xs@reserved@B{0}%
	\ifx\@xs@reserved@A\@xs@reserved@B
		\def\@xs@next{\@xs@ReturnResult{{#1}}{#3}}%
	\else
		\def\@xs@next{\@xs@StrFindGroup@i{#1}{#2}[#3]}%
	\fi
	\@xs@next
}

\def\@xs@StrFindGroup@i#1#2[#3]{%
	\def\@xs@StrFindGroup@ii{%
		\expandafter\@xs@testempty\expandafter{\@xs@reserved@C}%
		\if@xs@empty
			\def\@xs@next{\@xs@ReturnResult{}{#3}}% s'il ne reste plus rien, on renvoie vide
		\else
			\expandafter\@xs@returnfirstsyntaxunit\expandafter{\@xs@reserved@C}\@xs@reserved@D
			\ifx\bgroup\@xs@toks% si la 1\`ere unit\'e syntaxique est un groupe explicite
				\advance\decimalpart\@ne% on augmente le compteur
				\ifnum\decimalpart=\@xs@reserved@A% on est au groupe cherch\'e lors de la profondeur courante ?
					\ifx\@empty\@xs@reserved@B% on est \`a la profondeur maximale ?
						\def\@xs@next{\expandafter\@xs@ReturnResult\expandafter{\@xs@reserved@D}{#3}}% on renvoie ce groupe
					\else% sinon
						\expandafter\def\expandafter\@xs@next\expandafter{\expandafter\@xs@StrFindGroup@i\@xs@reserved@D}% on recommence avec ce groupe
						\expandafter\@xs@addtomacro\expandafter\@xs@next\expandafter{\expandafter{\@xs@reserved@B}[#3]}% et les profondeurs de recherche restantes
					\fi
				\else
					\expandafter\@xs@removefirstsyntaxunit\expandafter{\@xs@reserved@C}\@xs@reserved@C
					\let\@xs@next\@xs@StrFindGroup@ii
				\fi
			\else
				\expandafter\@xs@removefirstsyntaxunit\expandafter{\@xs@reserved@C}\@xs@reserved@C
				\let\@xs@next\@xs@StrFindGroup@ii
			\fi
		\fi
		\@xs@next
	}%
	\@xs@extractgroupnumber{#2}\@xs@reserved@A\@xs@reserved@B
	\decimalpart\z@
	\ifnum\@xs@reserved@A>\z@\def\@xs@reserved@C{#1}\else\let\@xs@reserved@C\@empty\fi
	\@xs@StrFindGroup@ii
}

\def\@xs@extractgroupnumber#1#2#3{%
	\def\@xs@extractgroupnumber@i##1,##2\@xs@nil{\def#2{##1}\def#3{##2}}%
	\@xs@extractgroupnumber@i#1,\@xs@nil
	\ifx\@empty#3\else\@xs@extractgroupnumber@i#1\@xs@nil\fi
}

\def\expandingroups{\let\@xs@expandingroups\exploregroups}
\def\noexpandingroups{\let\@xs@expandingroups\noexploregroups}
\def\StrExpand{\@testopt{\@xs@StrExpand}{1}}
\def\@xs@StrExpand[#1]#2#3{%
	\begingroup
		\@xs@expandingroups
		\ifnum#1>\z@
			\integerpart#1\relax
			\decimalpart\z@\def\groupID{0}\let\@xs@nestlevel\z@
			\def\@xs@atendofgroup{%
				\expandafter\endgroup
				\expandafter\@xs@addtomacro\expandafter\@xs@reserved@B\expandafter{\expandafter{\@xs@reserved@B}}%
				\@xs@read@reserved@C
			}%
			\def\@xs@atbegingroup{\let\@xs@reserved@B\@empty}%
			\def\@xs@atnextsyntaxunit{%
				\expandafter\expandafter\expandafter\@xs@addtomacro
				\expandafter\expandafter\expandafter\@xs@reserved@B
				\expandafter\expandafter\expandafter{\@xs@reserved@A}%
			}%
			\def\@xs@reserved@C{#2}%
			\@xs@StrExpand@i{#1}% Appel de la macro r\'ecursive
		\else
			\def\@xs@reserved@B{#2}%
		\fi
		\global\let\@xs@reserved@B\@xs@reserved@B
	\endgroup
	\let#3\@xs@reserved@B
	\let\groupID\@empty
}

\def\@xs@StrExpand@i#1{%
	\ifnum#1>\z@
		\let\@xs@reserved@B\@empty
		\@xs@read@reserved@C
		\let\@xs@reserved@C\@xs@reserved@B
		\def\@xs@reserved@A{\expandafter\@xs@StrExpand@i\expandafter{\number\numexpr#1-1}}%
	\else
		\let\@xs@reserved@A\relax
	\fi
	\@xs@reserved@A
}

\def\scancs{\@testopt{\@xs@scancs}{1}}
\def\@xs@scancs[#1]#2#3{%
	\@xs@StrExpand[#1]{#3}{#2}%
	\edef#2{\detokenize\expandafter{#2}}%
}

\catcode`\@=\CurrentAtCatcode\relax
\setverbdelim{|}%
\fullexpandarg\saveexpandmode
\comparenormal\savecomparemode
\noexploregroups\saveexploremode
\expandingroups

\newcommand\myscale{0.8}

\newcommand\tcirc[3]{
	\ifthenelse{\equal{#1}{w}}{\filldraw[fill=white,draw=black] (#2) circle (0.075);}{}%
	\ifthenelse{\equal{#1}{b}}{\filldraw[black] (#2) circle (0.075);}{}%
	\draw (#2) ++(0,0.35) node {$#3$};
	}

\newcommand\bond[1]{\draw (#1) -- +(1,0)}

\newcommand\vbond[1]{\draw (#1) -- +(0,-1)}

\newcommand\diagbond[2]{
	\ifthenelse{\equal{#1}{u}}{
		\draw (#2) -- +(0.5,0.865);
	}{}
	\ifthenelse{\equal{#1}{d}}{
		\draw (#2) -- +(0.5,-0.865);
	}{}
	}

\newcommand\dbond[2]{
	\draw (#2) ++(0.03,0.03) -- +(0.94,0);
	\draw (#2) ++(0.03,-0.03) -- +(0.94,0);
	\ifthenelse{\equal{#1}{r}}{
		\draw[semithick] (#2) ++(0.6,0) ++(-0.15,0.2) -- ++(0.15,-0.2) -- +(-0.15,-0.2);
	}{}
	\ifthenelse{\equal{#1}{l}}{
		\draw[semithick] (#2) ++(0.45,0) ++(0.15,0.2) -- ++(-0.15,-0.2) -- +(0.15,-0.2);
	}{}
	}

\newcommand\tbond[2]{
	\draw (#2)  -- +(1,0);
	\draw (#2) ++(0.05,0.06) -- +(0.9,0);
	\draw (#2) ++(0.05,-0.06) -- +(0.9,0);
	\ifthenelse{\equal{#1}{r}}{
		\draw[semithick] (#2) ++(0.6,0) ++(-0.15,0.2) -- ++(0.15,-0.2) -- +(-0.15,-0.2);
	}{}
	\ifthenelse{\equal{#1}{l}}{
		\draw[semithick] (#2) ++(0.45,0) ++(0.15,0.2) -- ++(-0.15,-0.2) -- +(0.15,-0.2);
	}{}
	}

\newcommand\tcross[2]{
	\draw (#1) ++(0,0.35) node {$#2$};
	\draw[semithick] (#1) ++(-0.15,-0.15)-- +(0.3,0.3);
	\draw[semithick] (#1) ++(-0.15,0.15)-- +(0.3,-0.3);
	}

\newcommand\tsquare[2]{
		\draw[semithick,color=blue] (#1) ++(-0.15,-0.15) rectangle ++(0.3,0.3);
		\tcross{#1}{#2};
		}

\newcommand\tstar[2]{
	\draw[color=red] (#1) node {\Large$*$};
	\draw (#1) ++(0,0.35) node {$#2$};
	}

\newcommand\DDnode[3]{
\ifthenelse{\equal{#1}{w}}{\tcirc{w}{#2}{#3}}{}		% white - non-compact root (Satake diagram)
\ifthenelse{\equal{#1}{b}}{\tcirc{b}{#2}{#3}}{}		% black - compact root (Satake diagram)
\ifthenelse{\equal{#1}{x}}{\tcross{#2}{#3}}{}		% crossed root (corresponding to parabolic)
\ifthenelse{\equal{#1}{s}}{\tstar{#2}{#3}}{}		% starred root (my notation for sub-parabolic)
\ifthenelse{\equal{#1}{q}}{\tsquare{#2}{#3}}{}		% crossed square (Iw root)
}

 \newcommand\Athree[2]{
 \begin{tiny}
 \begin{tikzpicture}[scale=\myscale,baseline=-3pt]

 \bond{0,0};		% bond
 \bond{1,0};		% double bond

 \StrBefore{#2}{,}[\labelone]
 \StrBetween[1,2]{#2}{,}{,}[\labeltwo]
 \StrBehind[2]{#2}{,}[\labelthree]

 \StrChar{#1}{1}[\nodetype];
 \DDnode{\nodetype}{0,0}{\labelone};
 \StrChar{#1}{2}[\nodetype];
 \DDnode{\nodetype}{1,0}{\labeltwo};
 \StrChar{#1}{3}[\nodetype];
 \DDnode{\nodetype}{2,0}{\labelthree};
 \useasboundingbox (-.4,-.2) rectangle (2.4,0.55); % make bounding box bigger
 \end{tikzpicture}
 \end{tiny}
 }

\newcommand\Edd[2]{
 \begin{tiny}
 \begin{tikzpicture}[scale=\myscale,baseline=-3pt]
 \foreach \x in {0,1,2,3} {
	\bond{\x,0};
 }
 \vbond{2,0};
 
 \StrLen{#1}[\Ernk]
 
 \StrChar{#1}{1}[\nodetype];
 \DDnode{\nodetype}{0,0}{\StrBefore{#2}{,}};
 \StrChar{#1}{2}[\nodetype];
 \DDnode{\nodetype}{2,-1}{\StrBetween[1,2]{#2}{,}{,}};
 \StrChar{#1}{3}[\nodetype];
 \DDnode{\nodetype}{1,0}{\StrBetween[2,3]{#2}{,}{,}};
 \StrChar{#1}{4}[\nodetype];
 \DDnode{\nodetype}{2,0}{\StrBetween[3,4]{#2}{,}{,}};
 \StrChar{#1}{5}[\nodetype];
 \DDnode{\nodetype}{3,0}{\StrBetween[4,5]{#2}{,}{,}};
 \StrChar{#1}{6}[\nodetype];

 \ifthenelse{\equal{\Ernk}{6}}{
 		\DDnode{\nodetype}{4,0}{\StrBehind[5]{#2}{,}};
 		\useasboundingbox (-.4,-1.2) rectangle (4.4,0.55);
	}{}%
 
 \ifthenelse{\equal{\Ernk}{7}}{
 		\bond{4,0};
 		\DDnode{\nodetype}{4,0}{\StrBetween[5,6]{#2}{,}{,}};
		\StrChar{#1}{7}[\nodetype];
		\DDnode{\nodetype}{5,0}{\StrBehind[6]{#2}{,}};
 		\useasboundingbox (-.4,-1.2) rectangle (5.4,0.55);
	}{}%

 \ifthenelse{\equal{\Ernk}{8}}{
 		\bond{4,0};
 		\bond{5,0};
 		\DDnode{\nodetype}{4,0}{\StrBetween[5,6]{#2}{,}{,}};
		\StrChar{#1}{7}[\nodetype];
		\DDnode{\nodetype}{5,0}{\StrBetween[6,7]{#2}{,}{,}};
		\StrChar{#1}{8}[\nodetype];
		\DDnode{\nodetype}{6,0}{\StrBehind[7]{#2}{,}};
		\useasboundingbox (-.4,-1.2) rectangle (6.4,0.55);
	}{}%

 \end{tikzpicture}
 \end{tiny}
 }
 \newcommand\tr{\operatorname{tr}}

 \renewcommand\l{\left}
 \renewcommand\r{\right}

 \newcommand\qRa{\quad\Rightarrow\quad}

 \newcommand\op{\oplus}
 \newcommand\ot{\otimes}

 \newcommand\fann{\mathfrak{ann}}

 \newcommand\rnk{\mathrm{rank}}

 \newcommand\im{\operatorname{im}}

 \newcommand\fa{{\mathfrak a}}

 \newcommand\fg{{\mathfrak g}}
 \newcommand\fgl{\mathfrak{gl}}
 \newcommand\fh{{\mathfrak h}}

 \newcommand\fk{{\mathfrak k}}

 \newcommand\fn{{\mathfrak n}}
 
 \newcommand\fp{{\mathfrak p}}

 \newcommand\fs{{\mathfrak s}}
 \newcommand\fsl{\mathfrak{sl}}
 \newcommand\fso{\mathfrak{so}}

 \newcommand\fz{{\mathfrak z}}

 \newcommand\cD{{\mathcal D}}
 \newcommand\cE{{\mathcal E}}
 \newcommand\cF{{\mathcal F}}
 \newcommand\cG{{\mathcal G}}

 \newcommand\cO{{\mathcal O}}
 \newcommand\cP{{\mathcal P}}

 \newcommand\cS{{\mathcal S}}

 \newcommand\cZ{{\mathcal Z}}

 \newcommand\sfr{\mathsf{r}}
 \newcommand\sfs{\mathsf{s}}

 \newcommand\sfw{\mathsf{w}}

 \newcommand\sfH{\mathsf{H}}

 \newcommand\sfX{\mathsf{X}}
 \newcommand\sfY{\mathsf{Y}}

 \newcommand\bbC{{\mathbb C}}

 \newcommand\bbK{{\mathbb K}}

 \newcommand\bbP{{\mathbb P}}
 
 \newcommand\bbR{{\mathbb R}}
 
 \newcommand\bbT{{\mathbb T}}
 
 \newcommand\bbV{{\mathbb V}}
 \newcommand\bbW{{\mathbb W}}

 \newcommand\bbZ{{\mathbb Z}}

 \newcommand\tspan{\mathrm{span}}

 \newcommand\Ben{\begin{enumerate}}
 \newcommand\Een{\end{enumerate}}
 \newcommand\Bex{\begin{example}}
 \newcommand\Eex{\end{example}}
 \newcommand\GL{\operatorname{GL}}

 \newcommand\SL{\mathrm{SL}}
\newcommand\Flag{\mathrm{Flag}}

 \newcommand\Ad{{\rm Ad}}
 \def\assoc/{associative}
 \def\arb/{arbitrary}
 \def\btw/{between}
 \def\coeff/{coefficient}
 \def\cohom/{cohomology}
 \def\coord/{coordinate}
 \def\coordsys/{coordinate system}
 \def\cpt/{compact}
 \def\cred/{completely reducible}
 \def\cts/{continuous}
 \def\dga/{differential-graded algebra}
 \def\dR/{de Rham}
 \def\Euc/{Euclidean} 
 \def\grp/{group}
 \def\hom/{homomorphism}
 \def\inv/{invariant}
 \def\iso/{isomorphism}
 \def\La/{Lie algebra}
 \def\Lag/{Lagrangian Grassmannian}
 \def\LG/{Lie group}
 \def\MA/{Monge--Amp\`ere}
 \def\MC/{Maurer--Cartan}
 \def\lintr/{linear transformation} 
 \def\mfld/{manifold}
 \def\nb/{normal bundle}
 \def\nbd/{neighbourhood}
 \def\nondeg/{non-degenerate}
 \def\posdef/{positive definite}
 \def\pu/{partition of unity}
 \def\rep/{representation}
 \def\Riem/{Riemannian}
 \def\sg/{subgroup}
 \def\ss/{semi-simple}
 \def\inv/{invariant}
 \def\irr/{irreducible}
 \def\Jacid/{Jacobi identity}
 \def\li/{linearly independent}
 \def\nd/{nowhere dependent}
 \def\nz/{nowhere zero}
 \def\on/{orthonormal}
 \def\onb/{\on/ basis}
 \def\orc/{\orth/ complement}
 \def\orth/{orthogonal}
 \def\orp/{\orth/ projection}
 \def\pde/{partial differential equation}
 \def\resp/{respectively}
 \def\seq/{sequence}
 \def\std/{standard}
 \def\SW/{Stiefel-Whitney}
 \def\uc/{universal cover}
 \def\vb/{vector bundle}
 \def\vf/{vector field}
 \def\vs/{vector space}
 \def\wrt/{with respect to}
 
 \renewcommand\mod{\,{\rm mod}\ }
 \renewcommand\dim{{\rm dim}}
 
\newtheorem{theorem}{Theorem}[section]
\newtheorem{lemma}[theorem]{Lemma}

\newtheorem{prop}[theorem]{Proposition}
\theoremstyle{definition}

\newtheorem{example}[theorem]{Example}

\theoremstyle{remark}
\newtheorem{remark}[theorem]{Remark}
\numberwithin{equation}{section}
\numberwithin{table}{section}

 \newcommand\diag{\operatorname{diag}}
 \newcommand\bmat[1]{\begin{bmatrix} #1 \end{bmatrix}}

 \newcommand\finf{\mathfrak{inf}}

 \newcommand\Axox[1]{\Athree{xwx}{#1}}
 \newcommand{\p}{\partial}
 \newcommand\PSL{\operatorname{PSL}}

 \renewcommand\myscale{0.8}

\title[Homogeneous ILC structures in dimension five]{Homogeneous integrable Legendrian contact structures in dimension five}

\author{Boris Doubrov}
\address{Faculty of Mathematics and Mechanics, Belarusian State 
University, Nezavisimosti ave. 4, 220050, Minsk, Belarus}
\email{doubrov@bsu.by}
 
\author{Alexandr Medvedev}
\address{School of Science and Technology, University of New 
England, Armidale NSW 2351, Australia}
\curraddr{International School for Advanced Studies, via Bonomea 
265, Trieste	34136, Italy}
\email{amedvedev@sissa.it}
 
\author{Dennis The}
\address{Mathematical Sciences Institute, Australian National 
University, ACT 0200, Australia}
\curraddr{Fakult\"at f\"ur Mathematik, Universit\"at Wien, Oskar-Morgenstern Platz 1, 1090 Wien, Austria}
\email{dennis.the@univie.ac.at}

\begin{document}
\begin{abstract}
We consider Legendrian contact structures on odd-dimensional complex analytic manifolds.  We are particularly interested in integrable structures, which can be encoded by compatible complete systems of second order PDEs on a scalar function of many independent variables and considered up to point transformations.  Using the techniques of parabolic differential geometry, we compute the associated regular, normal Cartan connection and give explicit formulas for the harmonic part of the curvature.  The PDE system is trivializable by means of point transformations if and only if the harmonic curvature vanishes identically.

In dimension five, the harmonic curvature takes the form of a binary quartic field, so there is a Petrov classification based on its root type.  We give a complete local classification of all five-dimensional integrable Legendrian contact structures whose symmetry algebra is transitive on the manifold and has at least one-dimensional isotropy algebra at any point.
\end{abstract}

\maketitle

 \section{Introduction}
 
A \emph{Legendrian contact structure} $(M;E,F)$ is defined to be a splitting of a contact distribution $C$ (on an odd-dimensional manifold $M$) into the direct sum of two subdistributions $E,F$ that are maximally isotropic with respect to the naturally defined conformal symplectic structure on $C$.  Such structures can be treated in both the real smooth and complex analytic categories. In the current paper, we assume that all our manifolds and related objects are complex analytic, although many results are also valid in the smooth category.
  
We shall exclusively deal with \emph{integrable} Legendrian contact structures (or just ILC structures), which means that both isotropic subdistributions are completely integrable. The main sources of ILC structures are compatible complete systems of 2nd order PDEs on one unknown function of several variables (considered up to {\em point} transformations), i.e.
\[
 \frac{\partial^2 u}{\partial x^i \partial x^j} = f_{ij}(x,u,\partial u), \qquad 1 \leq i,j \leq n,
\]
 and the complexifications of (Levi-nondegenerate) CR structures of codimension 1. 

The smallest dimension of a manifold with a Legendrian contact structure is 3. In this dimension both isotropic subdistributions are 1-dimensional and are automatically completely integrable. The corresponding ILC structures can be encoded by a single 2nd order ODE and have been well-studied starting from the pioneering work of Tresse~\cite{Tresse1896} (see also~\cite{Bol1932,Cartan1924,Olver1995}). Their real counterpart, CR structures on 3-dimensional real hypersurfaces in $\mathbb{C}^2$, have also been well-studied starting from the classical works of \'Elie Cartan~\cite{Cartan1932a,Cartan1932b}.

Legendrian contact structures belong to the class of so-called parabolic geometries. In particular, they enjoy a number of important properties derived from the general theory of parabolic geometries~\cite{CS2009}: the existence of a natural Cartan connection, description of the principal invariants in terms of the representation theory of simple Lie algebras, finite-dimensional symmetry algebras, and the classification of submaximal symmetry dimensions~\cite{KT2013}. Legendrian contact structures are modeled by the flag variety $\operatorname{Flag}_{1,n+1}(\mathbb{C}^{n+2})$ of pairs of incident lines and hyperplanes in $\mathbb{C}^{n+1}$ equipped with a natural action of $\operatorname{PGL}(n+2,\mathbb{C})$.

 We note that in \cite{Tak1994}, Takeuchi studied the special class of Legendrian contact structures that are induced on the projective cotangent bundle $M = \cP(T^*N)$ from a projective structure $(N,[\nabla])$.  With the sole exception of the flat model, this induced structure on $M$ is never an ILC structure.  Thus, his study is transverse to our study here.
 
In the current paper we are mainly interested in the classification of \emph{multiply transitive} ILC structures in dimension~5. The term ``multiply transitive'' means that the symmetry algebra of the ILC structure should be transitive on the manifold and should have a non-trivial isotropy subalgebra (i.e.\ at least one-dimensional) at each point.  As our study here is local in nature, we may as well require these conditions in an open subset of the manifold.  

In dimension~3, all multiply transitive ILC structures are flat. This reflects a well-known fact that any 2nd order ODE is either equivalent to the trivial equation $u''(x)=0$ and has 8-dimensional symmetry algebra, or its symmetry algebra is at most 3-dimensional. In dimension 5 this is no longer the case, as, for example, the submaximally symmetric ILC structures have symmetry algebras of dimension~8 and are multiply transitive~\cite{KT2013}.  In fact, all ILC structures with 8 symmetries are locally equivalent. This leaves us with the classification of ILC structures with 6- and 7-dimensional symmetry. 
A similar classification of integrable CR-manifolds in dimension 5 with transitive symmetry algebras of dimension~7 was done by A.V.~Loboda~\cite{Loboda2001a,Loboda2001b}.

As in the case of the geometry of scalar 2nd order ODEs, complete systems of 2nd order PDEs also admit a notion of duality that swaps the set of dependent and independent variables with the space of constants of integration parametrizing the generic solution. This corresponds to swapping the two isotropic distributions defining the ILC structure. We classify ILC structures up to this duality and indicate which structures are self-dual, i.e. locally contact equivalent to their dual.

In his famous 1910 paper~\cite{Cartan1910}, \'Elie Cartan studied the geometry of rank two distributions on 5-manifolds having generic growth vector $(2,3,5)$.  For such structures, Cartan solved the local equivalence problem and obtained a classification of all multiply transitive models.\footnote{One inadvertent omission from Cartan's list was recently discovered in~\cite{DouGov2013}.}  While the equivalence problem was solved by means of {\em Cartan's equivalence method}~\cite{Gardner1989}, we bypass this step in our study of ILC structures by using the full power of parabolic geometry.  Indeed, representation theory is used to quickly construct the full curvature module and set up the structure equations for the (regular, normal) Cartan geometry.  Our classification of multiply transitive ILC structures implements Cartan's technique, which we refer to as \emph{Cartan's reduction method}.

There is another striking similarity between ILC structures in dimension 5 and $(2,3,5)$ distributions. In both cases the fundamental invariant is represented by a single binary quartic. Similar to the Petrov classification for the Weyl curvature tensor in Lorentzian (conformal) geometry, we classify ILC structures in dimension 5 by the number and multiplicity of roots of this quartic. We also prove that non-flat multiply transitive structures may only have type N (a single root of multiplicity $4$), type D (two roots of multiplicity $2$), or type III (one simple root and one root of multiplicity 3). This is quite similar to Cartan's result \cite{Cartan1910} that all multiply transitive $(2,3,5)$-distributions have either type N or type D.  We identify the maximal symmetry dimension for each Petrov type in Theorem \ref{T:Petrov-sym}.

The main result of our paper can be summarized as follows:
\begin{theorem}
Any multiply transitive ILC structure in dimension 5 is locally equivalent to the ILC structure defined by one of PDE models in Table \ref{F:ILC-classify} or its dual.
 \begin{table}[ht]
 \caption{Classification of all multiply transitive ILC structures in dimension 5}
 \label{F:ILC-classify}
 \[
 \begin{array}{|l|c|c|c|c|l|}\hline
 \mbox{{\rm Model}}& \mbox{{\rm SD}} & u_{11} & u_{12} & u_{22} & \mbox{{\rm Remarks}}\\ \hline\hline
 \mbox{{\rm O.15}} & \checkmark & 0 & 0 & 0 & \mbox{{\rm flat model}}\\ \hline
 \mbox{{\rm N.8}} & \checkmark & q^2 & 0 & 0 & \mbox{{\rm unique submaximal}}\\
 \mbox{{\rm N.7-1}} & \xmark & q^2\cG_\kappa(x) & 0 & 0 & \kappa \in \bbC_\infty \backslash \{ 0, -3 \}; \,\,\kappa\sim -3-\kappa\\
 \mbox{{\rm N.7-2}} & \checkmark & \frac{1}{q} & 1 & 0 & \\
 \mbox{{\rm N.6-1}} & \checkmark & \cF_\mu(q) & 1 & 0 & \mu \in \bbC \backslash \{ -1, 2 \}\\
 \mbox{{\rm N.6-2}} & * & \cF_\mu(q)\cG_\kappa(x) & 0 & 0 & \mu\in\bbC_\infty\backslash\{-1,2\}, \kappa\in\bbC_\infty\backslash\{0,-3\}; \\
&&&&& \mu\sim 1-\mu,\,\,\kappa\sim-3-\kappa\\ \hline
 \mbox{{\rm D.7}} & \checkmark & p^2 & 0 & \lambda q^2 & \lambda \in \bbC \backslash\{-1\};\,\, \lambda\sim \frac{1}{\lambda} \mbox{ {\rm for} } \lambda \neq 0\\ 
 \mbox{{\rm D.6-1}} & \checkmark & p^2-\frac{q^4}{4} & q(p-\frac{q^2}{2}) & p- \frac{q^2}{2} & \\
 \mbox{{\rm D.6-2}} & \checkmark & \cG_\mu(p) & 0 & 0 & \mu \in \bbC_\infty \backslash \{ 0,1,2 \}\\
 \mbox{{\rm D.6-3}} & \checkmark & \lambda p^2R & 1+\lambda(pq-2u)R & \lambda q^2R & R=\frac{\sqrt{u-pq}}{u^{3/2}},\, \lambda \in \bbC \backslash \{ 0, \pm \frac{1}{2} \}; \lambda\sim-\lambda\\[1mm]
 \mbox{{\rm D.6-3}}_\infty & \checkmark & p^2\sqrt{1-2pq} & (pq-1)\sqrt{1-2pq} & q^2\sqrt{1-2pq} & \\[1mm]
 \mbox{{\rm D.6-4}} & \xmark & 0 & \frac{1+pq}{u} & 0 & \\[1mm] \hline
 \mbox{{\rm III.6-1}} & \xmark & \frac{p}{x-q} & 0 & 0 & \\
 \mbox{{\rm III.6-2}} & \xmark & 2q(2p-qu) & q^2 & 0 & \\ \hline
 \end{array}
 \]
 \end{table}
\end{theorem}
\begin{remark}
We denote by $u_{11}, u_{12}, u_{22}$ the second order partial derivatives of the unknown function $u$, and use the notation $p=u_1, q=u_2$ for the first order derivatives. 

The functions $\cF_\mu$ and $\cG_\kappa$ are defined as follows:
 \[ 
 \begin{array}{ll}
 \cF_\mu(z)&=\left\{ \begin{array}{ll} 
 z^\mu, & \mu \in \bbC \backslash \{ 0, 1\}\\
 \ln(z), & \mu = 0\\
 z\ln(z), & \mu = 1\\
 \exp(z), & \mu = \infty
 \end{array} \right.
\\
\\
 \cG_\kappa(z) &=\left\{ \begin{array}{ll} 
 z^\kappa, & \kappa \in \bbC \\
 \exp(z), & \kappa = \infty
 \end{array} \right.
 \end{array}\]
In particular, the parameters $\mu,\kappa$ are both allowed to take the value $\infty$ if the contrary is not stated.
\end{remark}

\begin{remark}
 A checkmark or cross under the SD column indicates that \emph{every} element in the indicated family is self-dual or not self-dual, respectively.  The situation for N.6-2 is more complicated. The corresponding ILC structure is self-dual if and only if the parameters $\mu$ and $\kappa$ satisfy $\mu-\kappa-2=0$ or $\mu+\kappa+1=0$ (see Table~\ref{Ap:Dual}).
\end{remark}

\begin{remark}
Equations from different items in this list correspond to inequivalent ILC structures. However, there are some additional equivalence relations on parameter spaces for multi-parameter equations within the same item. They are indicated in the last column of Table~\ref{F:ILC-classify}. 

Our labelling abides by the following rules. The first letter (N, D, or III) denotes the type of the invariant binary quartic. The next digit (6, 7, or 8) refers to the dimension of the symmetry algebra. The final digit is a labelling of the equation within the given subclass. Finally, the case D.6-3${}_\infty$ is a limit of D.6-3 as the parameter $\lambda$ tends to infinity.
\end{remark}

Table~\ref{TT:algebras} describes basic algebraic properties of symmetry algebras for obtained models.
   \begin{table} 
   \caption{Symmetry algebras of multiply transitive ILC structures}\label{TT:algebras}
\[
\begin{array}{llll}
  \toprule
  \mbox{Model} & \mbox{Derived series (DS)} & \mbox{Nilradical} & \mbox{Comments} \\ \midrule
  \mbox{N.8}  & [8,6,4,0] & \mbox{6-dim, DS = [6,4,0], LCS = [6,4,3,1,0]}\\ \midrule
  \mbox{N.7-1} & [7,5,2,0] & \mbox{5-dim, DS = LCS = [5,2,0]}\\ \midrule
  \mbox{N.7-2} & [7,6,6] & \mbox{4-dim abelian} & (\fso_3 \ltimes \bbC^3) \times \bbC\\ \midrule
  \mbox{N.6-1} & \left\{ \begin{array}{ll} [6,5,2,0], & \mu\ne 0\\ {} [6,4,1,0], & \mu=0\end{array} \right. & \mbox{5-dim, DS = [5,2,0], LCS = [5,2,1,0]}\\ \midrule
  \mbox{N.6-2} & [6,4,0] & 
  \left\{ \begin{array}{ll} 
  \begin{array}{@{}l} \mbox{5-dim, DS = [5,2,0],} \\ \,\, \quad \mbox{ LCS = [5,2,1,0]} \end{array}, & \mu=\kappa=\infty\\
  \mbox{4-dim abelian}, & \mbox{otw} \\ 
   \end{array} \right. \\ \midrule
  \mbox{D.7} & \left\{ \begin{array}{ll} [7,6,6], & \lambda\ne 0 \\ {} [7,6,4,3,3], & \lambda= 0 \end{array} \right. & \left\{ \begin{array}{ll} \mbox{1-dim abelian}, & \lambda\ne0 \\ {} \mbox{3-dim Heisenberg}, & \lambda=0 \end{array} \right. 
&  \begin{array}{ll} \fsl_2\times\fsl_2\times\bbC \\ \ \end{array} \\ \midrule
  \mbox{D.6-1}  & [6,6] & \mbox{1-dim} & \fsl_2 \ltimes \fs_3, \\ 
                            &         &                          &  \fs_3 \mbox{ is Heisenberg} \\ \midrule
  \mbox{D.6-2} & [6,4,1,0] & \mbox{4-dim, DS = LCS = [4,1,0]}\\ \midrule
  \mbox{D.6-3} & [6,6] & \mbox{0-dim} & \fsl_2\times \fsl_2\\ \midrule
  \mbox{D.6-3${}_\infty$} & [6,6] & \mbox{3-dim abelian} & \fso_3\ltimes \bbC^3\\ \midrule
  \mbox{D.6-4} & [6,6] & \mbox{0-dim} & \fsl_2\times \fsl_2\\ \midrule
  \mbox{III.6-1} & [6,4,2,0] & \mbox{4-dim, DS = [4,2,0], LCS = [4,2,1,0]}\\ \midrule
  \mbox{III.6-2} & [6,5,5] & \mbox{2-dim abelian} & \fgl_2\ltimes\bbC^2\\ \bottomrule
\end{array}
\]
\end{table}

The paper is organized as follows. In Section~\ref{S:LCS} we provide generalities concerning Legendrian contact structures, establish the relationship between ILC structures and compatible complete systems of 2nd order PDEs, discuss the notion of duality, define the (regular, normal) Cartan connection associated with a given ILC structure, and provide explicit formulas for the fundamental (harmonic) part of its curvature. 

In Section~\ref{S:five} we specialize to 5-dimensional ILC structures, define the fundamental binary quartic and prove that ILC structures of types I and II cannot be multiply transitive. We also reconstruct the full curvature tensor of the Cartan geometry. 

In Section~\ref{S:Cartan} we proceed with the detailed Cartan analysis of the general regular, normal Cartan connection, which involves normalizing parts of the curvature and its derivatives, reducing the Cartan bundle and iterating the procedure. As we are interested only in multiply transitive ILC structures, we terminate this process as soon as the fibers become 0-dimensional. This leads us to the list of all possible structure equations for the reduced bundles. We integrate each of these structure equations and come up with the corresponding ILC model defined in terms of the system of 2nd order PDEs. Finally, in the Appendix we give the detailed Lie algebra isomorphisms establishing the correspondence between the Cartan equations of the reduced bundle and the model systems of 2nd order PDEs, the equivalence relations on the parameters and the duality.
 
\medskip
\textbf{Acknowledgements:} The \texttt{Cartan} and \texttt{DifferentialGeometry} packages in Maple (written by Jeanne Clelland and Ian Anderson respectively) provided an invaluable framework for implementing the Cartan reduction method and subsequently carrying out the analysis of the structures obtained. The work of the second and third authors was supported by ARC Discovery grants DP130103485 and DP110100416 respectively.  D.T. was also supported by project M1884-N35 of the Austrian Science Fund (FWF).
 
 \section{Legendrian contact structures}\label{S:LCS}
 
 \subsection{Generalities} \label{S:gen}
 On any contact manifold $(M,C)$, the contact distribution $C \subset TM$ is locally defined by the vanishing of a 1-form $\sigma$ (unique up to multiplication by a non-vanishing function), and $d\sigma|_C$ is a (conformal) symplectic form.   Given a splitting $C = E \op F \subset TM$ into transverse Legendrian subdistributions $E$ and $F$, i.e.\ $d\sigma|_E = 0$ and $d\sigma|_F = 0$, we say $(M; E,F)$ is a {\em Legendrian contact (LC) structure}.  Let $\dim(M) = 2n+1$, so $n = \rnk(E) = \rnk(F)$.  Two LC structures $(M; E,F)$ and $(\tilde{M}; \tilde{E}, \tilde{F})$ are {\em (locally) equivalent} if there exists a (local) diffeomorphism $\phi : M \to \tilde{M}$ such that $d\phi(E) = \tilde{E}$ and $d\phi(F) = \tilde{F}$.  There is also a natural notion of duality of LC structures: the {\em dual} of $(M;E,F)$ is $(M;F,E)$.
 
 Since $E$ and $F$ are Legendrian, then $[E,E] \subset C$ and $[F,F] \subset C$.  The projections from $C$ onto $E$ and $F$ induce maps $\tau_E : \Gamma(E) \times \Gamma(E) \to \Gamma(F)$ and $\tau_F : \Gamma(F) \times \Gamma(F) \to \Gamma(E)$ that obstruct the integrability of $E$ and $F$.  The structure is {\em semi-integrable} or {\em integrable} according to whether one or both of $\tau_E,\tau_F$ are identically zero.  In the latter case, we call it an {\em ILC structure}.
 
 \begin{prop} \label{P:V-std}
 Given any contact manifold $(M,C)$ of dimension $2n+1$ and a rank $n$ integrable subdistribution $V \subset C$, we may choose local coordinates $(x^i,u,p_i)$ on $M$ such that contact form is $\sigma = du - p_i dx^i$ and $V = \tspan\{ \partial_{p_i} \}$.
 \end{prop}
 
 \begin{proof}
 Since $V$ is integrable and rank $n$, then by the Frobenius theorem there exist local coordinates $\{ x^i \}_{i=1}^{2n+1}$ such that $V = \ker\{ dx^1 = ... = dx^{n+1} \}$.  Hence, $C = \ker\{ \sigma \}$, where $\sigma = \lambda_1 dx^1 + ... + \lambda_{n+1} dx^{n+1}$.  The contact condition $(d\sigma)^n \wedge \sigma \neq 0$ implies that not all $\lambda_i$ can simultaneously vanish, so WLOG $\lambda_{n+1} \neq 0$ locally, and after rescaling we may assume $\lambda_{n+1} = 1$.  Now define 
 $u = x^{n+1}$ and $p_i = \lambda_i$.  The contact condition guarantees that $(x^i,u,p_i)$ is indeed a coordinate system.
 \end{proof}
 
 Suppose that $V:=F$ is integrable, i.e.\ the LC structure is semi-integrable.  By Proposition \ref{P:V-std}, there exist functions $f_{ij} = f_{ij}(x^k,u,p_\ell)$ with $f_{ij} = f_{ji}$ (since $E$ is Legendrian) such that
 \begin{align} \label{E:EV}
 E = \tspan\{ \cD_i := \partial_{x^i} + p_i \partial_u + f_{ij} \partial_{p_j} \}, \qquad V = \tspan\{ \partial_{p_i} \}.
 \end{align}
 Equivalently, we are studying the geometry of the system of scalar 2nd order PDE 
 \begin{align} \label{E:PDE}
 \frac{\partial^2 u}{\partial x^i \partial x^j} = f_{ij}(x,u,\partial u), \qquad 1 \leq i, j \leq n,
 \end{align}
 considered up to {\em point transformations}.  These are contact transformations that preserve the (vertical) bundle $V$.  All such transformations are precisely the prolongations of arbitrary diffeomorphisms in the $(x^i,u)$ variables.  The system \eqref{E:PDE} is overdetermined if $n > 1$.  If $n=1$, then \eqref{E:PDE} is a single 2nd order ODE, whose point geometry has been well-studied \cite{Tresse1896}.
 
 \begin{remark}  Consider the jet spaces $J^k = J^k(\bbC^n,\bbC)$ and projections $\pi^k_\ell : J^k \to J^\ell$.  On $J^2$, the contact system is $\{ du - p_i dx^i, dp_i - p_{ij} dx^j \}$, expressed in standard jet coordinates.  Pulling back to a submanifold $\cE$ defined by $p_{ij} = f_{ij}(x^k,u,p_\ell)$ yields the subbundle $E$ in \eqref{E:EV}.  The restriction $\pi^2_1|_\cE : \cE \to J^1$ is a local diffeomorphism.  The subbundle $V$ in \eqref{E:EV} is tangent to the fibers of $\pi^1_0 \circ \pi^2_1$.
 \end{remark}
 
 \begin{lemma}
 The PDE system \eqref{E:PDE} is compatible if and only if $E$ in \eqref{E:EV} is integrable.
 \end{lemma}
\begin{proof}
 It is easy to see that $[\cD_i,\cD_j]\in E$ if and only if $[\cD_i,\cD_j]=0$, which happens if and only if $\cD_j f_{ik}= \cD_i f_{jk}$ for $1\le i,j,k \le n$. This is exactly the compatibility condition of~\eqref{E:PDE}.
\end{proof}
 
 \subsection{Duality}
 If the dual LC structures $(M; E, F)$ and $(M; F, E)$ are equivalent, then we say that the structure is \emph{self-dual}.  For ILC structures, the notion of duality generalizes the classical duality for 2nd order ODE~\cite{Cartan1924}.  Namely, for the ILC structure $(M;E,V)$ given by \eqref{E:EV}, we can (by Proposition \ref{P:V-std}) find coordinates $(y^i,v,q_i)$ for the dual ILC structure $(M;V,E)$, i.e.
 \[
 V = \tspan\{ \partial_{y^i} + q_i \partial_v + \tilde{f}_{ij} \partial_{q_j} \}, \quad
  E = \tspan\{ \partial_{q_i} \}.
 \]
 Then $\frac{\partial^2 v}{\partial x^i \partial x^j} = \tilde{f}_{ij}$
 is the \emph{dual system} to \eqref{E:PDE} (and is well-defined only up to point transformations).
 
 \begin{example}
 The simplest example of an ILC structure is the flat model $u_{ij} = 0$.  The Legendre transformation $(y^i,v,q_k) = (p_i, u-p_j x^j, -x^k)$, is a contact (but non-point) transformation which swaps the $E$ and $V$ subbundles, so this structure is self-dual.
 \end{example}

 \begin{example}
 For ILC structures when $n=2$, we have the self-dual D.7 systems:
 \[
 \cS_\lambda : \quad u_{11} = p^2, \quad u_{12} = 0, \quad u_{22} = \lambda q^2, \qquad \lambda \in \bbC \backslash \{ -1 \},
 \]
 where $p = u_1$ and $q = u_2$.
 For fixed $\lambda$, a self-duality, i.e.\ a swap $(E,V) \mapsto (V,E)$, is exhibited by
 \[
 \Phi(x,y,u,p,q) = \left\{ \begin{array}{ll} \left(-\lambda (x+\frac{1}{p}),-(y+\frac{1}{\lambda q}),-u+\ln(-p)+\frac{1}{\lambda} \ln(-q),\frac{p}{\lambda},q \right), & \lambda \neq 0;\\
 (-(x+\frac{1}{p}),-q,-u+qy+\ln(-p), p,-y), & \lambda = 0\\ \end{array} \right.
 \]
 Moreover, $\cS_\lambda \cong \cS_{1/\lambda}$ when $\lambda \neq 0$ via the transformation $\Phi(x,y,u,p,q) = (y,x,\lambda u, \lambda q, \lambda p)$.
 \end{example}
 
As in the case of dual 2nd order ODEs, the dual ILC structures can be constructed in terms of the corresponding PDE models via swapping the space of independent and dependent variables with the space of integration constants parametrizing solutions of a given compatible PDE. In more detail, the general solution of any compatible system~\eqref{E:PDE} is parametrized by $n+1$ constants of integration and can be written as:
\begin{equation}\label{E:Dual}
F(x^i,u; a^j,b)=0,\qquad 1\le i,j \le n.
\end{equation}
We can consider this as an $(n+1)$-parameter family of hypersurfaces in $(x^i,u)$-space with parameter space $(a^j,b)$.  On the other hand, we can (locally) regard $b$ as a function of $a^j$, so that \eqref{E:Dual} can be interpreted as an $(n+1)$-parameter family of hypersurfaces in $(a^j,b)$-space with parameter space $(x^i,u)$.  This is the solution space of a well-defined compatible system of 2nd order PDE's on $b(a^j)$. 

Algorithmically, we construct the dual PDE system by differentiating \eqref{E:Dual} with respect to $a^j$ (regarding $x^i, u$ as constants and $b$ as a function of $a^j$), solving the obtained system of $n+1$ equations with respect to $x^i,u$ and substituting the solution into the second order derivatives of~\eqref{E:Dual} with respect to $a^j$. 
 \begin{example}
In the simplest example of the flat equation $u_{ij}=0$ the general solution is given by:
\[
u = a^1 x^1 + \dots a^n x^n + b.
\]
Treating $b$ as a function of $a^j$, differentiating this solution twice and excluding $x^i,u$ we get the same flat equation $b_{ij}=0$. This again demonstrates the self-duality of the flat model.
 \end{example}

 \begin{example}
 The III.6-1 system $u_{11} = \frac{p}{x-q}$, $u_{12} = u_{22} = 0$ has general solution
 \[
 u = -ay+c-b (x+a)^2, \qquad a,b,c \in \bbC.
 \]
 Regarding $c$ as a function of $a,b$ and treating $x,y,u$ as parameters, we have $c_a = y + 2b(x+a)$, $c_b = (x+a)^2$, and
 \[
 c_{aa} = 2b, \quad c_{ab} = 2(x+a) = \pm2\sqrt{c_b}, \quad c_{bb} = 0.
 \]
 WLOG, the $\pm$ ambiguity can be eliminated: the corresponding PDE systems are equivalent via the point transformation $(a,b,c) \mapsto (-a,b,c)$. Thus, the dual system to III.6-1 is 
 \[
 u_{11} = 2y, \quad u_{12} = 2 \sqrt{q}, \quad u_{22} = 0.
 \]
 Our classification indicates that III.6-1 is not self-dual (but a priori this is not at all obvious).
 \end{example}

 \subsection{LC structures as parabolic geometries}
 \label{S:LC-parabolic}

 There is an equivalence of categories between (holomorphic) LC structures $(M;E,F)$ and (regular, normal) parabolic geometries $(\cG\to M,\omega)$ of a fixed type $(G,P)$ \cite{CS2009}.  Here, $G = \operatorname{PGL}(n+2,\bbC)$ acts on the flag variety of pairs of incident lines and hyperplanes:
 \[
  G/P \cong \Flag_{1,n+1}(\bbC^{n+2}) = \{ (\ell,\pi) : \pi(\ell) = 0 \} \subset \bbC\bbP^{n+1} \times (\bbC\bbP^{n+1})^*,
 \]
 and $P \subset G$ is the parabolic subgroup which is the stabilizer of a chosen origin.  Since $A \in \GL(n+2,\bbC)$ and $\lambda A$ (for $\lambda \in \bbC^\times$) have the same action on $G/P$, we will instead use $G = \SL(n+2,\bbC)$.  The kernel of this action is isomorphic to the cyclic group $\bbZ_{n+2}$, generated by multiples of the identity matrix by $(n+2)$-th roots of unity.  In terms of Lie algebras, $P$ corresponds to the parabolic subalgebra $\fp\subset \fg=\fsl_{n+2}$ defined by the contact grading:
 \begin{align} \label{E:sl}
 \fsl_{n+2} = \l\{ \left[ \begin{array}{ccc} a & U & \gamma\\ X & A & W\\ \beta & Y & b \end{array} \right] : \begin{array}{l} b= -a -\tr(A), \\ a \in \bbC, \, A \in \fgl_n,\\ \mbox{etc.} \end{array} \r\} =  \fg_{-2} \op \fg_{-1} \op \overbrace{\fg_0 \op \underbrace{\fg_1 \op \fg_2}_{\fp_+}}^\fp.
 \end{align}
 The reductive part $G_0 \subset P$ has corresponding subalgebra $\fg_0 \cong \bbC^2 \times \fsl_n$ (corresponding to the diagonal blocks $(a,A,b)$) and there is a unique element $Z \in \cZ(\fg_0)$ that induces the grading.  We refer to the eigenvalues of $Z$ on a particular $\fg_0$-module as its {\em homogeneities}.

 At the origin $o \in G/P$, we have $T_o(G/P) \cong \fg/ \fp$.  Define the subspaces $E_o, F_o \subset T_o(G/P)$ (or subspaces in $\fg_{-1} / \fp$) corresponding to $X, Y$ in \eqref{E:sl} respectively.  The induced $G$-invariant structure $(G/P; E, F)$ is the {\em flat} LC structure, and $(G \to G/P, \omega_G)$ is the {\em flat} model, where $\omega_G$ is the Maurer--Cartan form on $G$.  The dimension of the Lie algebra of (infinitesimal) symmetries of the flat model is $\dim(G) = n^2 + 4n + 3$.
 
 A Cartan geometry $(\cG \to M, \omega)$ of type $(G,P)$ is a curved analogue of the flat model.  It consists of a principal $P$-bundle $\cG \to M$ equipped with a Cartan connection $\omega \in \Omega^1(\cG;\fg)$.  This means:
 \begin{enumerate}
 \item[(CC.1)] $\omega_u : T_u \cG \to \fg$ is a linear isomorphism for any $u \in \cG$;
 \item[(CC.2)] $R_p^* \omega = \Ad_{p^{-1}} \circ \omega$ for any $p \in P$;
 \item[(CC.3)] $\omega(\zeta_A) = A$ for any $A \in \fp$, where $\zeta_A(u) = \frac{d}{dt}|_{t=0} R_{\exp(tA)} (u)$, i.e. $\zeta_A$ is the fundamental vertical vector field corresponding to $A$.
 \end{enumerate}
 The curvature of $(\cG \to M,\omega)$ is the 2-form $K = d\omega + \frac{1}{2} [\omega, \omega] \in \Omega^2(\cG;\fg)$.  Using the framing of $T \cG$ provided by $\omega$ yields a $P$-equivariant function $\kappa : \cG \to \bigwedge^2 \fg^* \ot \fg$ which descends to $\kappa : \cG \to \bigwedge^2 (\fg/\fp)^* \ot \fg$ since $K$ is horizontal.  For parabolic geometries, the Killing form on $\fg$ yields a $P$-module isomorphism $(\fg / \fp)^* \cong \fp_+$, so we obtain a function $\kappa : \cG \to \bigwedge^2 \fp_+ \ot \fg$.  The geometry is 
 \begin{itemize}
 \item {\em regular} if $\kappa$ is valued in the subspace of $\bigwedge^2 \fp_+ \ot \fg$ consisting of positive homogeneities;
 \item {\em normal} if $\partial^* \kappa = 0$, where $\partial^*$ is the Lie algebra homology differential.
 \end{itemize}

 \subsection{Harmonic curvature}
 
 For regular, normal parabolic geometries, since $(\partial^*)^2 = 0$, we may 
 quotient $\kappa$ by $\im(\partial^*)$ to obtain $\kappa_H : \cG \to 
 \frac{\ker(\partial^*)}{\im(\partial^*)}$.  This fundamental curvature 
 quantity is called {\em harmonic curvature} and is a complete obstruction to 
 flatness of the geometry.  The $P$-module 
 $\frac{\ker(\partial^*)}{\im(\partial^*)}$ is completely reducible, so $\fp_+$ 
 acts trivially.  By a result of Kostant \cite{Kos1961, CS2009}, the 2-cochains 
 $C^2(\fg_-,\fg)$ admit the (orthogonal) $\fg_0$-module decomposition
 \begin{align} \label{E:Hodge}
 C^2(\fg_-,\fg) = \lefteqn{\overbrace{\phantom{\im(\partial^*) \op \ker(\Box)}}^{\ker(\partial^*)}}\im(\partial^*) \op \underbrace{\ker(\Box) \op  \im(\partial)}_{\ker(\partial)}, 
  \end{align}
 where $\partial$ is the Lie algebra differential, and $\Box = \partial 
 \partial^* + \partial^* \partial$ is the Kostant Laplacian. Thus,
 \[
 \frac{\ker(\partial^*)}{\im(\partial^*)} \cong \ker(\Box) \cong \frac{\ker(\partial)}{\im(\partial)} =: H^2(\fg_-,\fg).
 \]
 The $\fg_0$-module structure of the Lie algebra cohomology group $H^2(\fg_-,\fg)$ is completely described by Kostant's Bott--Borel--Weil theorem \cite{Kos1961, BE1989, CS2009}.  For LC structures with $n \geq 2$, $H^2(\fg_-,\fg)$ decomposes into three $\fg_0$-irreps
 \[
 H^2(\fg_-,\fg) = \bbW \op \bbT_1 \op \bbT_2%= \Axox{-3,4,-3} \op \Axox{-2,0,1} \op \Axox{1,0,-2} 
 \]
 having homogeneities $+2,+1,+1$ respectively.  The $\bbT_1$ and $\bbT_2$ components of $\kappa_H$ are precisely the torsions $\tau_E$ and $\tau_F$ (see Section \ref{S:gen}), and these vanish in the ILC case.  Results from twistor theory (see \cite{Cap2005}) indicate that the LC structures with trivial $\bbW$ and $\bbT_2$ components for $\kappa_H$ correspond to projective structures.  This is the case that was studied by Takeuchi \cite{Tak1994}.
 
 \subsection{Parametric computations of harmonic curvature}
 \label{S:KH-compute}

 Consider a semi-integrable LC structure  $(M;E,V)$ given by \eqref{E:EV}.  We will give an explicit formula for the $\bbW$-component of $\kappa_H$.

 We use the following co-frame for computations on the manifold $M$:
 \begin{align*}
 \theta^i =dx^i, \qquad
 \pi_i =dp_i-f_{ij}\,dx^j,\qquad
 \sigma =du-p_i\,dx^i, \qquad 1 \leq i,j \leq n,
 \end{align*}
 so that
 \[
 E = \ker\{ \sigma, \pi_i \}, \qquad V = \ker\{ \sigma, \theta^i \}.
 \]
 The differential of an arbitrary function $F$ is defined by the formula:
\[ d F= \frac{d F}{d x^i} \theta^i+\frac{\p F}{\p p_i} \pi_i+ \frac{\p F}{\p u} \sigma,\]
where $\frac{d}{d x^i} := \cD_i$ (see \eqref{E:EV}) is the total derivative with respect to $x^i$.

 Let $(\cG,\omega)$ be any {\em regular} Cartan geometry of type $(G,P)$ with underlying structure $(M;E,V)$ and curvature $K$.  Let $E_a{}^b \in \fgl_{n+2}$ denote the element with 1 in the $a$-th row and $b$-column and $0$ otherwise.  Here, we let $0 \leq a,b \leq n+1$.  If $s : M \to \cG$ is any (local) section, write
 \[
 s^*\omega = \omega^a{}_b \, E_a{}^b, \quad s^*K = K^a{}_b \, E_a{}^b,
 \]
 where $K^a{}_b = d\omega^a{}_b + \omega^a{}_c \wedge \omega^c{}_b$.

\begin{lemma}\label{lem1}
There exists a section $s\colon M\to\cG$ such that $s^* \omega$ satisfies
\[ 
 \omega^{n+1}{}_0 = \sigma,
\quad
\omega^i{}_0 = \theta^i,
\quad
\omega^{n+1}{}_i = \pi_i, 
\quad
 \omega^0{}_0 \equiv 0 \,\,\mod \{  \theta^i, \pi_i \}
\]
\end{lemma}

\begin{proof}
Consider a section $s\colon M\to \cG$. Since $\omega$ is regular, the negative 
part of $s^*\omega$ is an adapted coframe, i.e.
\begin{align*}
\omega^{n+1}{}_{0} = e \sigma, \quad
\omega^{n+1}{}_i = g^j{}_i \pi_j + g_i \sigma, \quad
\omega^i{}_0 = h^i{}_j \theta^j + h^i \sigma.
\end{align*}
An arbitrary section $\tilde{s}$ is given in terms of a function $h\colon M\to P$ such that $\tilde{s}=s\cdot 
h$. This satisfies:
\[ \tilde{s}^*\omega = h^{-1} \left(s^*\omega\right) h+ h^{-1} dh.\]
Since $h^{-1} dh$ term is $\fp$-valued, the negative part of 
$s^*\omega$ 
transforms via the adjoint action. 

Using the $G_0$-action, we can normalize $e=1$ and $g^j{}_i = \delta^j{}_i$. 
Since
 \[
 K^{n+1}{}_0 = \,d \omega^{n+1}{}_{0} + \omega^{n+1}{}_a \wedge \omega^a{}_0 
\equiv \,d \sigma  + \omega^{n+1}{}_i \wedge \omega^i{}_0
\equiv (-\pi_i + h^j{}_i \pi_j) \wedge \theta^i
\,\, \mod{\sigma},
 \]
 and regularity implies $K^{n+1}{}_0 \equiv 0 \,\,\mod{\sigma}$, then $h^i{}_j 
 =\delta^i{}_j$.
 Using the action of subgroup of $P$ corresponding to $\fg_1$, we can normalize $g_i = 0$, $h^i = 0$.
 Similarly, using the subgroup of $P$ corresponding to $\fg_2$, 
 we can normalize $\omega^0{}_0 \equiv 0 \,\, \mod \{ \theta^i, \pi_j \}$.
\end{proof}

 With respect to such a section, write
 \[
 \omega^a{}_b = r^a{}_{bi} \, \theta^i + s^a{}_b{}^i \, \pi_i + t^a{}_b \, \sigma.
 \]

To obtain the harmonic part of the normal curvature, it is sufficient to 
compute normalization conditions only in homogeneities 1 and 2. For any 
regular, normal parabolic geometry, the lowest homogeneity curvature component 
is harmonic \cite{CS2009}.  Thus, all curvature components in homogeneity $1$ 
must vanish except the coefficients of $\theta^j \wedge \theta^k$ in 
$K^{n+1}{}_i$, and this corresponds to the torsion of our semi-integrable 
structure.  Recalling that $\omega^{n+1}{}_{n+1} = -\omega^0{}_0 - \omega^i{}_i$ since 
$\omega$ is $\fsl_{n+2}$-valued, we have:
\begin{align*}
 K^{n+1}{}_{0} &= \,d \omega^{n+1}{}_{0} + \omega^{n+1}{}_a 
 \wedge \omega^a{}_0 \, = \omega^{n+1}{}_0 \wedge \omega^0{}_0 + 
 \omega^{n+1}{}_{n+1} \wedge \omega^{n+1}{}_0 = \sigma \wedge ( 2\omega^0{}_0 + 
 \omega^i{}_i)
\\
& = (2r^0{}_{0j}+ r^i{}_{ij}) \sigma \wedge \theta^j + (2s^0{}_0{}^j + 
s^i{}_i{}^j ) \sigma \wedge \pi_j
\\
 K^i{}_0
&=\,d \omega^i{}_0  + \omega^i{}_a \wedge \omega^a{}_0  \,\equiv \omega^i{}_0 
\wedge 
\omega^0{}_0 + \omega^i{}_j \wedge \omega^j{}_0 \quad\mod \sigma \\
%&{\color{red} = \theta^i \wedge (r^0{}_{0j} \theta^j + s^0{}_0{}^j \pi_j) + (r^i{}_{jk} \theta^k + s^i{}_j{}^k \pi_k) \wedge \theta^j \quad\mod \sigma}\\
& \equiv (r^i{}_{[jk]} + r^0{}_{0[j} \delta^i{}_{k]} ) \theta^k 
\wedge \theta^j +(s^i{}_j{}^k - s^0{}_0{}^k \delta^i{}_j)
\pi_k \wedge \theta^j  \quad\mod \sigma\\
K^{n+1}{}_i
&=\,d \omega^{n+1}{}_i + \omega^{n+1}{}_a \wedge \omega^a{}_i  \equiv d\pi_i + 
\pi_j 
\wedge \omega^j{}_i + \omega^{n+1}{}_{n+1} \wedge \pi_i \quad \mod \sigma\\
%&\equiv {\color{red} - \frac{d f_{ij}}{d x^k}\theta^k\wedge \theta^j - \frac{\p f_{ij}}{\p p_k}\pi_k\wedge \theta^j + r^j{}_{ik} \pi_j \wedge \theta^k + s^j{}_i{}^k \pi_j \wedge \pi_k - (r^0{}_{0k} + r^j{}_{jk})\theta^k \wedge \pi_i - (s^0{}_0{}^k + s^j{}_j{}^k)\pi_k \wedge \pi_i}  \quad \mod \sigma\\
&\equiv \frac{d f_{ij}}{d x^k}\theta^j\wedge \theta^k + \left(r^j{}_{ik} + 
(r^0{}_{0k} + r^l{}_{lk}) \delta^j{}_i - \frac{\p f_{ik}}{\p p_j} \right) \pi_j 
\wedge \theta^k\\
& \qquad + \left(s^j{}_i{}^k + (s^0{}_0{}^k + s^l{}_l{}^k) \delta^j{}_i\right) 
\pi_j \wedge \pi_k  \quad\mod \sigma
\end{align*}
 We confirm that the coefficient of $\theta^j \wedge \theta^k$ in $K^{n+1}{}_i$ 
 is indeed the obstruction $\cD_k f_{ij} - \cD_j f_{ik}$ to integrability of 
 $E$. All remaining terms above are zero, so we get:
\begin{equation} \label{peq1} 
 s^j{}_i{}^k=0, \quad s^0{}_0{}^i=0,\quad 
 r^0{}_{0i}=-\frac{1}{n+2}\frac{\p f_{ij}}{\p p_j},\quad
 r^i{}_{jk}=\frac{\p f_{jk}}{\p  p_i}- \delta^i{}_j\frac{1}{n+2}\frac{\p 
 f_{lk}}{\p p_l} .
\end{equation}
%Additionally, we make certain that the first order curvature part is indeed integrability condition for $E$:
%\[ K^{n+1}{}_i[\theta^k\wedge \theta^j]=\frac{d f_{ij}}{d x^k} - \frac{d f_{ik}}{d {x^j}}.\]

 \newcommand\rem[1]{\underbrace{{#1}}_{{\color{red} \mbox{remove}}}}

Proceed now to homogeneity 2.  Using \eqref{peq1}, we compute:
 \begin{align} \label{peq5b}
 K^i{}_0 &=  s^i{}_{n+1}{}^j \pi_j\wedge \sigma +  
 \left(r^i{}_{n+1,j}-t^i{}_j\right)\theta^j\wedge \sigma 
 \\
 K^{n+1}{}_i &= \frac{d f_{ij}}{d x^k}\theta^j\wedge \theta^k + \left(\frac{\p 
 f_{ij}}{\p u} - r^0{}_{ij}\right)\theta^j\wedge\sigma 
 +  \left(t^j{}_i-s^0{}_i{}^j+\delta^j_i t^k{}_k\right)\pi_j\wedge\sigma
 \\
 K^0{}_0 &\equiv \left(\frac{d r^0{}_{0i}}{d {x^j}} + 
 r^0{}_{ij}\right)\theta^j\wedge\theta^i +  \left(\frac{\p r^0{}_{0i} }{\p p_j} 
 + s^0{}_i{}^j\right)\pi_j\wedge\theta^i \quad\mod \sigma
 \\
 K^i{}_j &\equiv \left(\frac{d r^i{}_{jl}}{d x^k} +\delta^i{}_k r^0{}_{jl} 
 +r^i{}_{pk} r^p{}_{jl} \right) \theta^k\wedge\theta^l
+ \left( \frac{\p r^i{}_{jl}}{\p p_k} -\delta^k{}_l t^i{}_j-\delta^k{}_j 
r^i{}_{n+1,l} -\delta^i{}_l s^0{}_j{}^k \right) \pi_k \wedge \theta^l 
\label{peq5}
\\
&\qquad+ s^i{}_{n+1}{}^k\pi_k\wedge\pi_j \quad\mod \sigma 
\end{align}

 To obtain the pullback $s^*\kappa : M \to \bigwedge^2 \fp_+ \otimes \fg$ of 
 the curvature function $\kappa : \cG \to \bigwedge^2 \fp_+ \otimes \fg$, we
 note that the framing provided by $\omega$ together with $P$-equivariancy of 
 $\kappa$ allows us to identify $\sigma = \omega^{n+1}{}_0,$ $\theta^i = 
 \omega^i{}_0$ and $\pi_i = \omega^{n+1}{}_i$ with $ (E_j{}^0)^* $, 
 $(E_{n+1}{}^{j})^*$ and $(E_0{}^{n+1})^*$ respectively. 
A form $B$ on $\fgl_{n+2}$ which is defined by $B(X,Y)=\tr(XY)$ and is 
proportional to the Killing form on $\fgl_{n+2}$ induces a $P$-module 
isomorphism $(\fg / \fp)^* \cong \fp_+$. This allows us to make the replacements
 \[
 \theta^j \leftrightarrow E_0{}^j, \quad
 \pi_j \leftrightarrow E_j{}^{n+1}, \quad
 \sigma \leftrightarrow E_0{}^{n+1}
 \]
 in the curvature 2-form $K$.  The homology differential $\partial^* : 
 \bigwedge^2 \fp_+ \otimes \fg \to \fp_+ \otimes \fg$ is defined on 
 decomposable elements as
 \[
 \partial^*(X \wedge Y \otimes v) = -Y \otimes [X,v] + X \otimes [Y,v] - [X,Y] \otimes v.
 \]
 
 We introduce a bi-grading on $\bigwedge^\bullet \fp_+ \otimes \fg$. Let $\fh 
 \subset \fgl_{n+2}$ be Cartan subalgebra for the standard upper-triangular 
 Borel subalgebra. Let also $Z_i\in \fh, 1 \le i \le n+1$ be a dual basis to 
 the  simple roots basis $\alpha_i\in \fh^*, 1 \le i \le n+1$. Then the pair 
 $(Z_1,Z_{n+1})$ induces bi-grading $X\to (a_1,a_{n+1})$ where $[Z_i,X]=a_i 
 X$ for $i=1,n+1.$ Homogeneity of an element $X$ is equal to $a_1+a_{n+1}$ 
 since $Z=Z_1+Z_{n+1}$ where $Z$ is a grading element. Moreover, since 
 $\partial^*$ is $P$-equivariant map it respects bi-grading.
% Let us use a subscript ``2'' to denote a homogeneity +2 component.
% \begin{align*}
% \partial^*(K^i{}_0 \otimes E_i{}^0) &= \partial^*( s^i{}_{n+1}{}^j E_j{}^{n+1} \wedge E_0{}^{n+1} \otimes E_i{}^0 + \left(r^i{}_{n+1,j}-t^i{}_j\right) E_0{}^j\wedge E_0{}^{n+1} \otimes E_i{}^0)\\
% &= -s^i{}_{n+1}{}^j E_j{}^{n+1} \otimes E_i{}^{n+1} - \left(r^i{}_{n+1,j}-t^i{}_j\right) E_0{}^{n+1} \otimes (\delta_i{}^j E_0{}^0 - E_i{}^j) - \left(r^i{}_{n+1,j}-t^i{}_j\right) E_0{}^j \otimes E_i{}^{n+1} 
% \end{align*}

% $P$-EQUIVARIANCY OF $\partial^*$; in particular, bi-grading is respected.
 In order to compute harmonic curvature it is sufficient to use only 
 $\partial^* \kappa_{(1,1)}=0$ and $\partial^* \kappa_{(0,2)}=0$ normality 
 conditions. Using \eqref{peq5b}-\eqref{peq5} and $ K^{n+1}{}_{n+1} = -K^0{}_0 
 -K^i{}_i$ we compute:
 
 \begin{align*}
% 0 &=  \partial^*_{(2,0)} \kappa = \left(\frac{\p f_{ij}}{\p u} - 
%r^0{}_{ij}\right) E_0{}^j \otimes E_0{}^i + \left(\frac{d r^0{}_{0i}}{d {x^j}} 
%+ r^0{}_{ij}\right) (E_0{}^i \otimes E_0{}^j - E_0{}^j \otimes E_0{}^i) + \\
% &\qquad\qquad\qquad
% + \left(\frac{d r^i{}_{jl}}{d x^k} +\delta^i{}_k r^0{}_{jl} +r^i{}_{pk} 
%r^p{}_{jl} \right) (- E_0{}^l \otimes \delta^k{}_i E_0{}^j  + E_0{}^k \otimes 
%\delta^l{}_i E_0{}^j )
% \\
% &= \left(\frac{\p f_{ji}}{\p u} + r^0{}_{ij} - (n+1) r^0{}_{ji} - \frac{d 
%r^0{}_{0j}}{d {x^i}} + \frac{d r^0{}_{0i}}{d {x^j}} - \frac{d r^k{}_{ji}}{d 
%x^k} + \frac{d r^l{}_{jl}}{d x^i} - r^k{}_{pk} r^p{}_{ji} + r^l{}_{pi} 
%r^p{}_{jl} \right) E_0{}^i \otimes E_0{}^j
% \\
 0 = \partial^* \kappa_{(1,1)} =& 
 \left(r^i{}_{n+1,j}-t^i{}_j\right)
    \left( E_0{}^{n+1} \otimes (E_i{}^j - \delta_i{}^j E_0{}^0) -  E_0{}^j 
    \otimes E_i{}^{n+1} \right) 
    \\ 
    &+ \left(t^j{}_i-s^0{}_i{}^j+\delta^j{}_i t^k{}_k\right) \left( 
    -E_0{}^{n+1} \otimes (E_j{}^i - \delta_j{}^i E_{n+1}{}^{n+1}) + E_j{}^{n+1} 
    \otimes 
    E_0{}^i \right) 
    \\
    &+ \left(\frac{\p r^0{}_{0i} }{\p p_j} + 
    s^0{}_i{}^j\right)\left(-E_j{}^{n+1} 
    \otimes E_0{}^i + \delta^i{}_j E_0{}^{n+1} \otimes E_0{}^0 \right)
    \\
    &+\left( \frac{\p r^i{}_{jl}}{\p p_k} -\delta^k{}_l t^i{}_j-\delta^k{}_j 
    r^i{}_{n+1,l} -\delta^i{}_l s^0{}_j{}^k \right)
    \left(E_0{}^l \otimes \delta^j{}_k E_i{}^{n+1} + E_k{}^{n+1} \otimes 
    \delta^l{}_i E_0{}^j + \delta^l{}_k E_0{}^{n+1} \otimes E_i{}^j\right)
    \\
    &+\left( \frac{\p r^0{}_{0l}}{\p p_k}+\frac{\p r^i{}_{il}}{\p p_k}-\delta 
    ^k{}_l t^i{}_i-r^k_{n+1,l}\right) \left( E_0{}^l\otimes 
    E_k{}^{n+1}-\delta_k{}^l 
    E_0{}^{n+1}\otimes E_{n+1}{}^{n+1}\right)
    \\
    =& \left( \frac{\p r^j{}_{ki}}{\p p_k}+\frac{\p r^0{}_{0i}}{\p p_j}+
    \frac{\p r^k{}_{ki}}{\p p_j} - (n+2) r^j{}_{n+1,i} - \delta^j{}_i 
    (s^0{}_k{}^k+t^k{}_k)\right) E_0{}^i\otimes E_j{}^{n+1}
    \\
    &+ \left( \frac{\p r^i{}_{jk}}{\p p_k} - (n+2) t^i{}_j-\delta^i{}_j 
    t^k{}_k \right) E_0{}^{n+1}\otimes (E_i{}^j - \delta_i{}^jE_{n+1}{}^{n+1})
    \\
    &+ \left( \frac{\p r^0{}_{0i}}{\p p_i}+ t^i{}_i+s^0{}_i{}^i - r^i{}_{n+1,i}
    \right) E_0{}^{n+1} \otimes (E_0{}^0-E_{n+1}{}^{n+1})
    \\
    &+ \left( \frac{\p r^k{}_{ik}}{\p p_j}-\frac{\p r^0{}_{0i}}{\p p_j} 
    - (n+2) s^0{}_i{}^j+\delta^j{}_i(t^k{}_k-r^k{}_{n+1,k})\right) 
    E_j{}^{n+1}\otimes E_0{}^i,
 \\
 0 = \partial^* \kappa_{(0,2)} =& -s^i{}_{n+1}{}^j E_j{}^{n+1} \otimes 
 E_i{}^{n+1} + s^i{}_{n+1}{}^k (1-n) E_k{}^{n+1} \otimes E_i{}^{n+1} 
 -s^i{}_{n+1}{}^j E_j{}^{n+1} \otimes 
 E_i{}^{n+1}
 \\
 &+ s^i{}_{n+1}{}^j E_i{}^{n+1} \otimes 
 E_j{}^{n+1} = \left( s^j{}_{n+1}{}^i  -(n+1)s^i{}_{n+1}{}^j \right) 
 E_j{}^{n+1} \otimes 
 E_i{}^{n+1}.
 \end{align*}
 
Substituting  \eqref{peq1} we obtain linear system of equations on coefficients 
of normal regular Cartan connection:
\begin{align*}
0 &= \frac{\p^2 f_{ik}}{\p p_j\p p_k} - (n+2) r^j{}_{n+1,i} - \delta^j{}_i 
(s^0{}_k{}^k+t^k{}_k),
\\
0 &= \frac{\p^2 f_{jk}}{\p p_i\p p_k} - \delta^i{}_j\frac1{n+2}\frac{\p^2 
f_{lk}}{\p p_l \p p_k}  - (n+2) t^i{}_j-\delta^i{}_j, 
t^k{}_k 
\\
0 &= -\frac1{n+2}\frac{\p^2 f_{ij}}{\p p_i\p p_j} + t^i{}_i+s^0{}_i{}^i - 
r^i{}_{n+1,i},
\\
0 &= \frac{\p^2 f_{ik}}{\p p_j\p p_k} - (n+2) 
s^0{}_i{}^j+\delta^j{}_i(t^k{}_k-r^k{}_{n+1,k}).
\end{align*}

Solving the linear system of equations above we get the homogeneity 2 
coefficients of normal Cartan connection needed for the computation of 
$\kappa_H$:
\begin{align}
t^i{}_j&=\frac1{n+2}\frac{\p^2 f_{jk}}{\p p_i \p 
p_k}-\delta^i_j\frac1{(n+2)(n+1)}\frac{\p^2 f_{lk}}{\p p_l \p p_k}, \label{peq2}
\\
r^i{}_{n+1,j}&=t^i{}_j,  \label{peq3}
\\
s^0{}_i{}^j&=\frac1{n+2}\frac{\p^2 f_{ik}}{\p p_j \p p_k},  \label{peq4}
\\
s^i{}_{n+1}{}^j&=0.
%\\
%r^0{}_{ij}&=\frac1n\left(\frac{\p f_{ij}}{\p u}+\frac{\p f_{jl}}{\p 
%p_k}\frac{\p f_{ik}}{\p p_l}-\frac{\p f_{kl}}{\p p_k}\frac{\p f_{ij}}{\p 
%p_l}-\frac{d}{d x^k}\frac{\p f_{ij}}{\p p_k}+\frac{n+1}{n+2}\frac{d}{d 
%{x^j}}\frac{\p f_{ik}}{\p p_k}+\frac{1}{n+2}\frac{d}{d x^i}\frac{\p f_{jk}}{\p 
%p_k} \right).
\end{align}
From Kostant's theorem we know that $\bbW$ has the lowest weight vector $ 
\phi_0=E_0{}^1 \wedge E_n{}^{n+1} \otimes E_n{}^1$. The element $\phi_0$ 
belongs to the module $\bbV$ generated by:

 \[ w_l{}^k{}_i{}^j = E_0{}^k \wedge E_l{}^{n+1} \ot (E_i{}^j-\delta_i{}^j 
 E_{n+1}{}^{n+1}). \]
The module $\bbW$ is the submodule of $\bbV$ consisting of tensors that are trace-free in 
$(i,j)$, symmetric in $(l,i)$, and symmetric in $(k,j)$.
 We denote coefficients of $\kappa$ which corresponds to $w_l{}^k{}_i{}^j$ as 
$W^l{}_k{}^i{}_j$ and assume that 
\[ T^k{}_i{}^l{}_j= -\frac{\p^2 f_{ij}}{\p p_k \p p_l}.\]
 and $T^l{}_j=T^i{}_i{}^l{}_j$, $T=T^i{}_i{}^j{}_j$. According to \eqref{peq5} 
 \[ W^l{}_k{}^i{}_j = -\left( \frac{\p r^i{}_{jk}}{\p p_l} -\delta^l{}_k 
 t^i{}_j-\delta^l{}_j 
 r^i{}_{n+1,k} 
 -\delta^i{}_k r^l{}_{j,-1} \right). \]
 Using \eqref{peq1} and \eqref{peq2}-\eqref{peq4} we obtain that $W^l{}_k{}^i{}_j$ is equal to 
 trace-free part of tensor $T^l{}_k{}^i{}_j$:
\[ W^l{}_k{}^i{}_j=T^l{}_k{}^i{}_j-\frac1{n+2}\left( \delta^i{}_j T^l{}_k + 
\delta^l{}_k 
T^i{}_j + \delta^l{}_j T^i{}_k + \delta^i{}_k T^l{}_j 
\right)+\frac1{(n+2)(n+1)}\left( 
\delta^l{}_k\delta^i{}_j+\delta^l{}_j\delta^i{}_k\right)T.\]
Coefficients  $W^l{}_k{}^i{}_j$ are symmetric in $(i,l)$, $(j,k)$, trace-free in 
$(i,j)$ and form  the $\bbW$-component of $\kappa_H$. 

We summarize computations of this sub-section in the following theorem.
\begin{theorem}
	With respect to the section $s$ defined by Lemma \ref{lem1}, the $\bbW$ 
	component of the harmonic curvature of the regular, normal connection for 
	semi-integrable LCS given by \ref{E:EV} is
	\begin{equation}
	W^l{}_k{}^i{}_j E_0{}^k \wedge E_l{}^{n+1} \ot (E_i{}^j-\delta_i{}^j 
	E_{n+1}{}^{n+1}),
	\end{equation}
	where $W^l{}_k{}^i{}_j$ is the trace-free part of the tensor
	\[ T^l{}_k{}^i{}_j = -\frac{\p^2 f_{kj}}{\p p_l \p p_i}. \]
\end{theorem}

%The case of particular interest is $n=2$. We have just 5 components of $W$ in this case:
%\begin{align*}
%W_{22}^{11}&=\frac{\p^2 f_{22}}{\p p_1^2 },
%\\
%W_{12}^{22}=-W_{11}^{12}&=\frac12 \frac{\p^2 f_{22}}{\p p_1 \p p_2 }-\frac12 \frac{\p^2 f_{12}}{\p p_1^2},
%\\
%W_{11}^{11}=-W_{12}^{12}=W_{22}^{22}&=\frac16 \frac{\p^2 f_{22}}{\p p_2^2 }-\frac23 \frac{\p^2 f_{12}}{\p p_1 \p p_2} + \frac16 \frac{\p^2 f_{11}}{\p p_1^2 },
%\\
%W_{12}^{11}=-W_{22}^{12}&=\frac12 \frac{\p^2 f_{11}}{\p p_1 \p p_2 }-\frac12 \frac{\p^2 f_{12}}{\p p_2^2},
%\\
%W_{11}^{22}&=\frac{\p^2 f_{11}}{\p p_2^2 }.
%\end{align*}

 \section{ILC structures in dimension five}\label{S:five}
 
 Henceforth, we specialize to the $n=2$ ILC case, which corresponds to {\em compatible} PDE systems
 \[
 u_{11} = F, \quad u_{12} = G, \quad u_{22} = H,
 \]
 where $F,G,H$ are functions of $(x,y,u,p,q)$ with $p=u_1$ and $q=u_2$.  Equivalently, $E$ and $V$ as in \eqref{E:EV} (with $f_{11} = F$, $f_{12} = G$, $f_{22} = H$) are both integrable.  

 Let us fix notation for $\fp$.  Take the standard (upper triangular) Borel subalgebra, diagonal Cartan subalgebra $\fh \subset \fsl_4$, and simple roots $\alpha_i = \epsilon_i - \epsilon_{i+1} \in \fh^*$ for $i=1,2,3$.  The dual basis $Z_1,Z_2,Z_3 \in \fh$ to the simple roots is given by
 \begin{align*}
 Z_1 = \diag\left(\frac{3}{4}, -\frac{1}{4}, -\frac{1}{4}, -\frac{1}{4}\right), \quad 
 Z_2 = \diag\left(\frac{1}{2}, \frac{1}{2}, -\frac{1}{2}, -\frac{1}{2}\right), \quad
 Z_3 = \diag\left(\frac{1}{4}, \frac{1}{4}, \frac{1}{4}, -\frac{3}{4}\right).
 \end{align*}
 The grading element adapted to $P$ is $Z:= Z_1 + Z_3$.
 Use linear coordinates on $\fp$:
 \begin{align}
 \bmat{ \frac{3z_1 + z_2}{4} & t_1 & t_2 & t_5\\ 0 & v_1 + \frac{z_2 - z_1}{4} & v_2 & t_3\\ 0 & v_3 & -v_1 + \frac{z_2 - z_1}{4} & t_4\\ 0 & 0 & 0 & -\frac{z_1 + 3z_2}{4}} \in \fp.
 \label{E:p-coords}
 \end{align}
 We have $\fg_0 = \cZ(\fg_0) \times (\fg_0)_{ss} \cong \bbC^2 \times \fsl_2$, where $\bbC^2 = \tspan\{ Z_1,Z_3 \}$.  In terms of the standard basis $\{ E_a{}^b \}_{0 \leq a,b \leq 3}$ of $\fgl_4$, a standard $\fsl_2$-triple spanning the semisimple part $(\fg_0)_{ss} \subset \fp$ is given by:
 \begin{align} \label{E:HXY}
 	\sfH := E_1{}^1 - E_2{}^2, \qquad % \leftrightarrow \sfr\partial_\sfr - \sfs\partial_\sfs, \qquad
	\sfX := E_1{}^2, \qquad% \leftrightarrow \sfr\partial_\sfs, \qquad
	\sfY := E_2{}^1. % \leftrightarrow \sfs\partial_\sfr.
 \end{align}

 For ILC structures, $\kappa_H$ takes values in\footnote{In terms of $\fsl_4$ weights $\{ \lambda_i \}$, $\bbW$ has {\em lowest} weight $3\lambda_1 - 4\lambda_2 + 3\lambda_3 = \alpha_1 - \alpha_2 + \alpha_3$ by the ``minus lowest weight'' convention \cite{BE1989}.} the module $\bbW = \Axox{-3,4,-3}$ (in the notation of \cite{BE1989}).
 With respect to $(Z_1,Z_3)$, $\bbW$ has bi-grading $(+1,+1)$ so that its homogeneity is $+2$. As $\fsl_2$-modules, $\bbW \cong \bigodot^4(\bbC^2)$, i.e. the space of binary quartics in $\sfr,\sfs$, say. 
Hence, $\kappa_H$ (up to sign) is:
 \begin{align} \label{E:KH2}
 \kappa_H = F_{qq} \sfr^4 + 2 (F_{pq} - G_{qq}) \sfr^3 \sfs + (F_{pp} - 4 G_{pq} + H_{qq}) \sfr^2 \sfs^2 + 2( H_{pq} - G_{pp}) \sfr\sfs^3 + H_{pp} \sfs^4.
 \end{align}
 Strictly speaking, this is the pullback of $\kappa_H$ by a (local) section $s \colon M \to \cG$.  Since $P_+$ acts vertically trivially on $\ker(\partial^*) / \im(\partial^*)$, \eqref{E:KH2} is canonically defined only up to a $G_0$-transformation.
 
 \subsection{Petrov classification}
\label{SS:Petrov} As in the Petrov classification of the Weyl tensor in 4-dimensional Lorentzian (conformal) geometry and the classification of $(2,3,5)$-distributions \cite{Cartan1910}, ILC structures can be classified based on the (pointwise) root type of the binary quartic field \eqref{E:KH2}.  We use the same notation for types as in the Petrov classification, e.g. type N and D indicate a single quadruple root and a pair of double roots respectively.
 
Any ILC structure admits at most a 15-dimensional symmetry algebra and 15 is realized only on (an open subset of) the flat model (up to local isomorphism).  Among (regular, normal) parabolic geometries $(\cG \to M, \omega)$ of a given type $(G,P)$, Kruglikov and The~\cite{KT2013} gave a general method for finding the submaximal symmetry dimension, i.e.\ the symmetry dimension for any non-flat structure, and for ILC structures this dimension is eight. These techniques can also be used to determine the maximal symmetry dimension for ILC structures with constant root type.  We briefly outline their method.  A non-trivial root type corresponds to a $G_0$-orbit $\{ 0 \} \neq \cO \subset \bbW$ (or in type I, a collection of $G_0$-orbits).  Defining $\fa^\phi = \fg_- \op \fann(\phi)$ for non-flat ILC structures, we have:
 \begin{align} \label{E:KT-bound}
 \dim(\finf(\cG,\omega)) \leq \max\{ \dim(\fa^\phi) : \phi \in \cO \} = 5 + \max\{ \dim(\fann(\phi)) : \phi \in \cO \}.
 \end{align}
 Since $\dim(\fann(\phi))$ is constant along $G_0$-orbits, it suffices to evaluate it on a cross-section.
  
 \begin{theorem} \label{T:Petrov-sym}
 Among ILC structures with constant root type, we have:\\
 \begin{center}
 \begin{tabular}{|c|c|c|c|c|c|c|} \hline
 Root type & O & N & D & III & II & I\\ \hline\hline
 Max. sym. dim. & 15 & 8 & 7 & 6 & 5 & 5\\ \hline
 Sharp? & \checkmark & \checkmark & \checkmark & \checkmark & \checkmark & \checkmark \\ \hline
 \end{tabular}
 \end{center}
 \end{theorem}
 
 \begin{proof} See Table \ref{F:ILC-classify} for type N, D, III models with the stated symmetry dimensions.  

A Type I model with 5-dimensional symmetry is given by:
\[
u_{11} = 6S^{5/3}-6uS^{4/3}+2(u^2-q)S-2pq,\quad u_{12}=3S^{4/3} -2uS-q^2,\quad u_{22}=2S,
\]
where $S=p+uq$. Its harmonic curvature is given by the quartic:
\[
\kappa_H = -\frac{4}{3} \sfr(u\sfr+\sfs)(\sfr-(u\sfr+\sfs)S^{-1/3})(3\sfr - 2(u\sfr+\sfs)S^{-1/3}),
\]
which has four distinct roots on the open set $\{S\ne 0\}$. The equation is invariant with respect to the action of $\mathfrak{sl}_2\ltimes \bbC^2$ generated by:
\begin{multline*}
\partial_x,\quad \partial_y,\quad x\partial_y+\partial_u-q\partial_p, 
\quad 2x\partial_x + y\partial_y - u\partial_u-3p\partial_p-2q\partial_q, \\
\quad x^2\partial_x+xy\partial_y+(y-xu)\partial_u-(u+3xp+yq)\partial_p+(1-2xq)\partial_q. 
\end{multline*}

Next, consider
\[
u_{11} = p^{\lambda+2}q^{\mu}, \quad
u_{12} = p^{\lambda+1}q^{\mu+1}, \quad
u_{22} = p^{\lambda}q^{\mu+2},
\]
which is type II when $\lambda,\mu \ne 0$, $\lambda+\mu\ne 0, 1$ according to
\[
\kappa_H = p^{\lambda-2}q^{\mu-2}(p\sfr-q\sfs)^2(\mu(\mu-1)p^2\,\sfr^2+2\lambda\mu pq\,\sfr\sfs +\lambda(\lambda-1)q^2\,\sfs^2).
\] 
The symmetry algebra is generated by the 5 vector fields:
\[
\partial_x, \quad \partial_y, \quad \partial_u,\quad  
-\mu (x\partial_x - p\partial_p)+ \lambda (y\partial_y - q\partial_q),\quad
(1+\lambda+\mu)(x\partial_x+y\partial_y)+(\lambda+\mu)u\partial_u-p\partial_p-q\partial_q.
\]

 Now we establish upper bounds. Up to scale, representative elements in the $G_0$-orbits are
 \[
 \mbox{N:} \,\, \sfs^4; \quad
 \mbox{D:} \,\, \sfr^2 \sfs^2;\quad
 \mbox{III:} \,\, \sfr\sfs^3;\quad
 \mbox{II:} \,\, \sfr^2 \sfs (\sfr-\sfs); \quad
 \mbox{I:} \,\, \sfr\sfs(\sfr-\sfs)(\sfr-c\sfs), \quad c\in\bbC \backslash \{ 0,1 \}.
 \]
 The annihilators of the above elements, cf. \eqref{E:HXY}, are spanned by:
 \[
 \mbox{N:} \,\, Z_1 - Z_3, \, Y, \, H + 4 Z_1; \quad
 \mbox{D:} \,\, Z_1 - Z_3,\, H;\quad
 \mbox{III:} \,\, Z_1 - Z_3,\, H + 2 Z_1;\quad
 \mbox{II, I:} \,\, Z_1 - Z_3.
 \]
 By \eqref{E:KT-bound}, the result is proved for N and D, while for III, II, I the upper bound is one more than in the stated result.  For the latter, we show that the upper bound is never realizable.
 
 Consider the type III orbit and assume there is a model with 7-dimensional symmetry algebra $\fs$. According to~\cite[Cor.3.4.8]{KT2013} (in particular, ILC structures are ``prolongation-rigid''), $\fs$ admits a natural filtration $\fs=\fs^{(-2)}\supset \fs^{(-1)}\supset \fs^{(0)}$ with associated-graded Lie algebra isomorphic to $\fg_{-2} \oplus \fg_{-1} \oplus \fa_0$, where $\fa_0$ is the above annihilator of the type III orbit and $\fg_{-1}, \fg_{-2}$ are graded subspaces of $\fg=\fsl_4$. In other words, $\fs$ is a \emph{filtered deformation} of the above 7-dimensional graded Lie algebra. 

Any such deformation is necessarily invariant with respect to $\fs^{(0)}=\fa_0$. Fix a basis $T_1=Z_1-Z_3, T_2 =H+2Z_1$ in $\fa_0$. Its action on $\fg_{-1}$ and $\fg_{-2}$ diagonalizes with pairs of eigenvalues $(-1,-1)$, $(-1,-3)$, $(1,-1)$, $(1,1)$ and $(0,-2)$ respectively. Denote by 
$E_1 = E_1{}^0$, $E_2=E_2{}^0$, $F_1=E_3{}^1$, $F_2=E_3{}^2$, $U=-E_3{}^0$ 
the corresponding eigenvectors of this action. Then all possible deformations of $\fa_0\oplus \fg_{-1}\oplus \fg_{-2}$ preserving the filtration and the action of $\fa_0$ have the form:
\begin{align*}
& [T_1, E_1] = -E_1, \quad [T_1, E_2] = -E_2, \quad [T_1, F_1]=F_1, \quad [T_1,F_2]=F_2, \\
& [T_2, E_1] = -E_1, \quad [T_2, E_2] = -3E_2, \quad [T_2, F_1]=-F_1, \quad [T_2,F_2]=F_2, \quad [T_2,U]=-2U,\\
& [E_1, F_1] = U, \quad [E_1,F_2]=aT_1+bT_2, \quad [E_2,F_2]=U, \quad [F_2, U]=cF_1.
\end{align*}
However, due to Jacobi identity we get $a=b=c=0$. Thus, there are no non-trivial deformations in Type III case, and dimension $7$ of symmetry algebra is not realized.

Similarly, for types I and II we have the one-dimensional annihilator $\fa_0$ spanned by $T=Z_1-Z_3$. Using the same argument, we get a 4-parameter family of non-trivial deformations $\fs$ given by:
\begin{align*}
& [T, E_1] = -E_1, \quad [T, E_2] = -E_2, \quad [T, F_1]=F_1, \quad [T, F_2]=F_2,\\
& [E_1,F_1] = U + a_{11}T, \quad [E_1,F_2]=a_{12}T, \quad [E_2,F_1]=a_{21}T, \quad [E_2,F_2] = U+a_{22}T,\\
& [E_1, U] = -a_{22}E_1+a_{12}E_2, \quad [E_2, U] = a_{21}E_1-a_{11}E_2, \\
& [F_1, U] = a_{22}F_1 - a_{21} F_2,\quad [F_2, U] = -a_{12}F_1+a_{11}F_2.
\end{align*}
Replacing $U$ by $U + \lambda T$, we may assume that $a_{22} = -a_{11}$.  Each of these deformations $\fs$ defines an $S$-invariant ILC structure on the homogeneous space $S/S_0$, where $S$ is the corresponding Lie group and $S_0$ is the subgroup corresponding to the 1-dimensional subalgebra spanned by $T$. 
The linear map  $\alpha : \fs \to \fsl_4$ given by
\begin{align*}
 E_1 &\mapsto E_1{}^0 - \frac{1}{2} a_{11} E_1{}^3 - \frac{1}{2} a_{12} E_2{}^3, \qquad
 E_2 \mapsto E_2{}^0 - \frac{1}{2} a_{21} E_1{}^3 + \frac{1}{2} a_{11} E_2{}^3,\\
 F_1 &\mapsto E_0{}^1 - \frac{1}{2} a_{11} E_0{}^1 - \frac{1}{2} a_{21} E_0{}^2, \qquad
 F_2 \mapsto E_0{}^2 - \frac{1}{2} a_{12} E_0{}^1 + \frac{1}{2} a_{11} E_0{}^2,\\
 U &\mapsto -E_3{}^0 - \frac{1}{2} (a_{11} E_1{}^1 + a_{21} E_1{}^2 + a_{12} E_2{}^1 - a_{11} E_2{}^2) - \frac{1}{4} (a_{11}^2 + a_{12} a_{21}) E_0{}^3
\end{align*}
is in fact a Lie algebra homomorphism.  Hence, all these deformations are in fact trivial and yield the flat ILC structure \cite[Sec.1.5.15-16]{CS2009}.  This contradicts the type I or II assumption.
\end{proof}

 We exclude types II and I from further consideration, since no multiply transitive models exist.
 
 \subsection{Curvature module}
 
 Since ILC structures are torsion-free geometries, a result of \v{C}ap \cite[Sec.\ 3.2 corollary]{Cap2005} implies that the curvature function $\kappa$ takes values in the $P$-module $\bbK \subset \bigwedge^2 \fp_+ \ot \fp$ generated by $\bbW$.  We refer to $\bbK$ as the {\em curvature module}.
 
 From Kostant's theorem, $\bbW$ has lowest weight vector $\phi_0 = E_0{}^1 \wedge E_2{}^3 \otimes E_2{}^1 \leftrightarrow \sfs^4$, and we generate all of $\bbW$ by applying the raising operator $E_1{}^2 \leftrightarrow \sfr\partial_\sfs$. The result of applying the raising operators  $E_0{}^1, E_0{}^2, E_1{}^3, E_2{}^3, E_0{}^3 \in \fp_+$ to $\bbW$ is given in Table \ref{F:LC-curvature}.  Introduce coordinates on $\bbK$ (26-dimensional):
 \begin{align*}
 \begin{array}{|c|l@{\,}l|} \hline
 \mbox{Quartic}
 & A &= A_1 \sfr^4 + 4 A_2 \sfr^3 \sfs + 6 A_3 \sfr^2 \sfs^2 + 4 A_4 \sfr\sfs^3 + A_5 \sfs^4\\ \hline
 \mbox{Cubic}
 & B &= 4(B_1 \sfr^3 + 3 B_2 \sfr^2 \sfs + 3 B_3 \sfr\sfs^2 + B_4 \sfs^3)\sfw_1\\
 & B' &= 4(B_5 \sfr^3 + 3 B_6 \sfr^2 \sfs + 3 B_7 \sfr\sfs^2 + B_8 \sfs^3)\sfw_2\\ \hline
 \mbox{Quadratic}
 & C &= 6(C_1 \sfr^2 + 2C_2 \sfr\sfs + C_3 \sfs^2)\sfw_1^2\\
 & \widehat{C} &= 12(C_4 \sfr^2 + 2 C_5 \sfr\sfs + C_6 \sfs^2)\sfw_1\sfw_2\\
 & C' &= 6(C_7 \sfr^2 + 2C_8 \sfr\sfs + C_9 \sfs^2)\sfw_2^2\\ \hline
 \mbox{Linear}
 & D &= 12(D_1 \sfr + D_2 \sfs)\sfw_1^2 \sfw_2\\
 & D' &= 12(D_3 \sfr + D_4 \sfs)\sfw_1 \sfw_2^2\\ \hline
 \end{array}
 \end{align*}

 \begin{footnotesize}
 \begin{table}[h]
 \[
 \begin{array}{|c|c|l|} \hline
 \mbox{$(Z_1,Z_3)$-grade} & \mbox{Label} & \mbox{2-chain}\\ \hline\hline
 (+1,+1)
 & \sfr^4 & E_0{}^2 \wedge E_1{}^3 \otimes E_1{}^2 \\ 
 & 4\sfr^3 \sfs & -\Omega \otimes E_1{}^2 - E_0{}^2 \wedge E_1{}^3 \otimes \sfH\\ 
 & 6\sfr^2 \sfs^2 & 
 -E_0{}^1 \wedge E_2{}^3 \otimes E_1{}^2 - E_0{}^2 \wedge E_1{}^3 \otimes E_2{}^1 + \Omega \otimes \sfH\\
 & 4\sfr\sfs^3 & 
 \Omega \otimes E_2{}^1 + E_0{}^1 \wedge E_2{}^3 \otimes \sfH\\
 & \sfs^4 & E_0{}^1 \wedge E_2{}^3 \otimes E_2{}^1\\  \hline
 %%%%%%%%%%%
 (+2,+1) & 4\sfr^3 \sfw_1 & E_0{}^2 \wedge E_0{}^3 \otimes E_1{}^2 + E_0{}^2 \wedge E_1{}^3 \otimes E_0{}^2\\
 & 12 \sfr^2 \sfs \sfw_1 & -E_0{}^1 \wedge E_0{}^3 \otimes E_1{}^2 - E_0{}^2 \wedge E_1{}^3 \otimes E_0{}^1 - \Omega \otimes E_0{}^2 - E_0{}^2 \wedge E_0{}^3 \otimes \sfH \\
 & 12 \sfr\sfs^2 \sfw_1 & 
 -E_0{}^2 \wedge E_0{}^3 \otimes E_2{}^1 - E_0{}^1 \wedge E_2{}^3 \otimes E_0{}^2 + \Omega \otimes E_0{}^1 + E_0{}^1 \wedge E_0{}^3 \otimes \sfH\\
 & 4\sfs^3 \sfw_1 & E_0{}^1 \wedge E_0{}^3 \otimes E_2{}^1 + E_0{}^1 \wedge E_2{}^3 \otimes E_0{}^1\\ \hline
 %%%%%%%%%%%
 (+1,+2) & 4\sfr^3 \sfw_2 & E_1{}^3 \wedge E_0{}^3 \otimes E_1{}^2 + E_1{}^3 \wedge E_0{}^2 \otimes E_1{}^3 \\
 & 12\sfr^2\sfs \sfw_2 & E_2{}^3 \wedge E_0{}^3 \otimes E_1{}^2 + E_1{}^3 \wedge E_0{}^2 \otimes E_2{}^3 + \Omega \otimes E_1{}^3 - E_1{}^3 \wedge E_0{}^3 \otimes \sfH \\
 & 12\sfr\sfs^2 \sfw_2 & -E_1{}^3 \wedge E_0{}^3 \otimes E_2{}^1 - E_2{}^3 \wedge E_0{}^1 \otimes E_1{}^3 + \Omega \otimes E_2{}^3 - E_2{}^3 \wedge E_0{}^3 \otimes \sfH\\
 & 4\sfs^3 \sfw_2 & -E_2{}^3 \wedge E_0{}^3 \otimes E_2{}^1 - E_2{}^3 \wedge E_0{}^1 \otimes E_2{}^3 \\ \hline
 %%%%%%%%%%%
 (+3,+1) 
 & 6\sfr^2 \sfw_1^2 & E_0{}^2 \wedge E_0{}^3 \otimes E_0{}^2\\
 & 12\sfr\sfs \sfw_1^2 & E_0{}^3 \wedge E_0{}^2 \otimes E_0{}^1 + E_0{}^3 \wedge E_0{}^1 \otimes E_0{}^2\\
 & 6\sfs^2 \sfw_1^2 & E_0{}^1 \wedge E_0{}^3 \otimes E_0{}^1\\ \hline
 %%%%%%%%%%%
 (+2,+2) 
 & 12 \sfr^2 \sfw_1 \sfw_2 & E_1{}^3 \wedge E_0{}^3 \otimes E_0{}^2 + E_0{}^3 \wedge E_0{}^2 \otimes E_1{}^3 + E_1{}^3 \wedge E_0{}^2 \otimes E_0{}^3\\
 & 24\sfr\sfs \sfw_1 \sfw_2 & E_0{}^3 \wedge E_0{}^2 \otimes E_2{}^3 - E_0{}^3 \wedge E_2{}^3 \otimes E_0{}^2 + \Omega \otimes E_0{}^3 \\
 &&\qquad + E_0{}^3 \wedge E_1{}^3 \otimes E_0{}^1 - E_0{}^3 \wedge E_0{}^1 \otimes E_1{}^3\\
 & 12\sfs^2 \sfw_1 \sfw_2 & 
 E_0{}^1 \wedge E_0{}^3 \otimes E_2{}^3 + E_0{}^3 \wedge E_2{}^3 \otimes E_0{}^1 + E_0{}^1 \wedge E_2{}^3 \otimes E_0{}^3\\ \hline
 %%%%%%%%%%%
 (+1,+3) 
 & 6\sfr^2 \sfw_2^2 & E_0{}^3 \wedge E_1{}^3 \otimes E_1{}^3\\
 & 12\sfr\sfs \sfw_2^2 & E_0{}^3 \wedge E_1{}^3 \otimes E_2{}^3 + E_0{}^3 \wedge E_2{}^3 \otimes E_1{}^3\\
 & 6\sfs^2 \sfw_2^2 & E_0{}^3 \wedge E_2{}^3 \otimes E_2{}^3 \\ \hline
 %%%%%%%%%%%
 (+3,+2) 
 & 12 \sfr \sfw_1^2 \sfw_2 & E_0{}^3 \wedge E_0{}^2 \otimes E_0{}^3\\
 & 12 \sfs \sfw_1^2 \sfw_2 & E_0{}^1 \wedge E_0{}^3 \otimes E_0{}^3\\ \hline
 %%%%%%%%%%%
 (+2,+3)
 & 12 \sfr \sfw_1 \sfw_2^2 & E_0{}^3 \wedge E_1{}^3 \otimes E_0{}^3 \\
 & 12 \sfs \sfw_1 \sfw_2^2 & E_0{}^3 \wedge E_2{}^3 \otimes E_0{}^3 \\ \hline\hline
 %%%%%%%%%%%
 \multicolumn{2}{|c}{\mbox{Notation:}} & \Omega := E_0{}^1 \wedge E_1{}^3 - E_0{}^2 \wedge E_2{}^3\\
 \multicolumn{2}{|l}{\mbox{$\fp$-module description:}} & \mbox{Degree 4 polynomials in $\sfr,\sfs,\sfw_1,\sfw_2$ modulo $\sfw_1^3,\sfw_2^3,\sfw_1^3 \sfw_2, \sfw_1^2 \sfw_2^2, \sfw_1 \sfw_2^3$}\\ 
  \multicolumn{2}{|c}{\mbox{$\fsl_2$-action:}} & \sfH = E_1{}^1 - E_2{}^2 \leftrightarrow \sfr\partial_\sfr - \sfs\partial_\sfs, \qquad
	\sfX = E_1{}^2 \leftrightarrow \sfr\partial_\sfs, \qquad
	\sfY = E_2{}^1 \leftrightarrow \sfs\partial_\sfr\\
 \multicolumn{2}{|c}{\mbox{$\fp_+$-action:}} &
 E_0{}^1 \leftrightarrow \sfw_1\partial_\sfr, \qquad 
 E_0{}^2 \leftrightarrow \sfw_1\partial_\sfs, \qquad
 E_1{}^3 \leftrightarrow -\sfw_2\partial_\sfs, \qquad 
 E_2{}^3 \leftrightarrow \sfw_2\partial_\sfr, \qquad 
 E_0{}^3 \leftrightarrow 0 \\ \hline
 \end{array}
 \]
 \caption{The curvature module for ILC structures}
 \label{F:LC-curvature}
 \end{table}

 \end{footnotesize}
  
 \subsection{Structure equations}

 Write the Cartan connection $\omega \in \Omega^1(\cG;\fg)$ as
 \begin{align*}
   \omega = [\omega^a{}_b] &= \bmat{ 
   \frac{3\zeta_1 + \zeta_2}{4} & \tau_1 & \tau_2 & \tau_5\\ 
   \varpi_1 & \nu_1 + \frac{\zeta_2 - \zeta_1}{4} & \nu_2 & \tau_3\\
   \varpi_2 & \nu_3 & -\nu_1 + \frac{\zeta_2 - \zeta_1}{4} & \tau_4\\
   \varpi_5 & \varpi_3 & \varpi_4 & -\frac{\zeta_1 + 3\zeta_2}{4}}.
 \end{align*}
 Decompose $K = K^a{}_b E_a{}^b$, where $K^a{}_b \in \Omega^2(\cG)$.
 By torsion-freeness, $K^1{}_0 = K^2{}_0 = K^3{}_0 = K^3{}_1 = K^3{}_2 = 0$, and for ILC structures $K^0{}_0 = K^3{}_3 = 0$. The structure equations are $d\omega^a{}_b = -\omega^a{}_c \wedge \omega^c{}_b + K^a{}_b$, i.e.
 \begin{footnotesize}
 \[
 \begin{array}{l@{\quad\quad}l}
 \begin{array}{r@{\,}l}
 d\tau_1 &= (\nu_1 -\zeta_1) \wedge \tau_1 + \nu_3 \wedge \tau_2 - \tau_5 \wedge \varpi_3 +  K^0{}_1\\
 d\tau_2 &= \nu_2 \wedge \tau_1 -(\zeta_1 + \nu_1) \wedge \tau_2 - \tau_5 \wedge \varpi_4 +  K^0{}_2\\
 d\tau_3 &= \tau_5 \wedge \varpi_1 - \left(\nu_1 + \zeta_2\right) \wedge \tau_3 - \nu_2 \wedge \tau_4 + K^1{}_3\\
 d\tau_4 &= \tau_5 \wedge \varpi_2 - \nu_3 \wedge \tau_3 + \left(\nu_1 - \zeta_2 \right) \wedge \tau_4 + K^2{}_3\\
 d\tau_5 &= - \tau_1 \wedge \tau_3 - \tau_2 \wedge \tau_4 -(\zeta_1 + \zeta_2) \wedge \tau_5 +  K^0{}_3\\
 \end{array} &
 \begin{array}{r@{\,}l}
 d\varpi_1 &=  \left(\zeta_1 - \nu_1 \right) \wedge \varpi_1 - \nu_2 \wedge \varpi_2 - \tau_3 \wedge \varpi_5 \\
 d\varpi_2 &= - \nu_3 \wedge \varpi_1 + \left(\nu_1 + \zeta_1 \right) \wedge \varpi_2 - \tau_4 \wedge \varpi_5 \\
 d\varpi_3 &= \left(\nu_1 + \zeta_2 \right) \wedge \varpi_3 + \nu_3 \wedge \varpi_4 + \tau_1 \wedge \varpi_5 \\
 d\varpi_4 &= \nu_2 \wedge \varpi_3  + \left( \zeta_2 - \nu_1 \right) \wedge \varpi_4 + \tau_2 \wedge \varpi_5 \\
 d\varpi_5 &= \varpi_1 \wedge \varpi_3 + \varpi_2 \wedge \varpi_4 + (\zeta_1 + \zeta_2) \wedge \varpi_5 
 \end{array}\\
 \multicolumn{2}{l}{
 \begin{array}{r@{\,}l}
 d\zeta_1 &= -\frac{3}{2} \tau_1 \wedge \varpi_1 - \frac{3}{2} \tau_2 \wedge \varpi_2 + \frac{1}{2} \tau_3 \wedge \varpi_3 + \frac{1}{2} \tau_4 \wedge \varpi_4 - \tau_5 \wedge \varpi_5 \\
 d\zeta_2 &= \frac{1}{2} \tau_1 \wedge \varpi_1 + \frac{1}{2} \tau_2 \wedge \varpi_2 - \frac{3}{2} \tau_3 \wedge \varpi_3 - \frac{3}{2} \tau_4 \wedge \varpi_4 - \tau_5 \wedge \varpi_5 \\
 d\nu_1 &= \frac{1}{2} \tau_1 \wedge \varpi_1 - \frac{1}{2} \tau_2 \wedge \varpi_2 - \frac{1}{2} \tau_3 \wedge \varpi_3 + \frac{1}{2} \tau_4 \wedge \varpi_4 - \nu_2 \wedge \nu_3  + \frac{1}{2} K^1{}_1 - \frac{1}{2} K^2{}_2\\
 d\nu_2 &= \tau_2 \wedge \varpi_1 - 2 \nu_1 \wedge \nu_2 - \tau_3 \wedge \varpi_4 + K^1{}_2\\
 d\nu_3 &= \tau_1 \wedge \varpi_2 + 2\nu_1 \wedge \nu_3 - \tau_4 \wedge \varpi_3 + K^2{}_1
 \end{array}}
 \end{array}
 \]
  \end{footnotesize}
 
 To convert from $\kappa : \cG \to \bigwedge^2 \fp_+ \ot \fg$ to $K \in \Omega^2(\cG;\fg)$, 
 the Killing form on $\fsl_4$ induces $(\fg/\fp)^* \cong \fp_+$:
 \[
 E_0{}^1 \leftrightarrow \omega^1{}_0 = \varpi_1, \quad
 E_0{}^2 \leftrightarrow \omega^2{}_0 = \varpi_2, \quad
 E_1{}^3 \leftrightarrow \omega^3{}_1 = \varpi_3, \quad
 E_2{}^3 \leftrightarrow \omega^3{}_2 = \varpi_4, \quad
 E_0{}^3 \leftrightarrow \omega^3{}_0 = \varpi_5.
 \]
 Writing $\varpi_{kl} := \varpi_k \wedge \varpi_l$ for $1 \leq k,l \leq 5$, we have
 \begin{footnotesize}
 \begin{align*}
 K^1{}_1 = -K^2{}_2 &= -A_2 \varpi_{23} + A_3 ( \varpi_{13} - \varpi_{24}) + A_4 \varpi_{14}
- B_2 \varpi_{25} + B_3 \varpi_{15} - B_6 \varpi_{35} - B_7 \varpi_{45}\\
 K^1{}_2 &=+ A_1 \varpi_{23} - A_2(\varpi_{13} - \varpi_{24}) - A_3 \varpi_{14} + B_1 \varpi_{25} - B_2 \varpi_{15} + B_5 \varpi_{35} + B_6 \varpi_{45}\\
 K^2{}_1 &= -A_3 \varpi_{23} + A_4(\varpi_{13} - \varpi_{24}) + A_5 \varpi_{14} - B_3 \varpi_{25} + B_4 \varpi_{15} - B_7 \varpi_{35} - B_8 \varpi_{45}\\
 K^0{}_1 &= -B_2 \varpi_{23} + B_3 (\varpi_{13} - \varpi_{24}) + B_4 \varpi_{14} - C_2 \varpi_{25} + C_3 \varpi_{15} - C_5 \varpi_{35} - C_6 \varpi_{45}\\
 K^0{}_2 &= +B_1 \varpi_{23} - B_2 (\varpi_{13} - \varpi_{24}) - B_3 \varpi_{14} + C_1 \varpi_{25} - C_2 \varpi_{15} + C_4 \varpi_{35} + C_5 \varpi_{45}\\
 K^1{}_3 &= -B_5 \varpi_{23} + B_6 (\varpi_{13} - \varpi_{24}) + B_7 \varpi_{14} - C_4 \varpi_{25} + C_5 \varpi_{15} - C_7 \varpi_{35} - C_8 \varpi_{45}\\
 K^2{}_3 &= -B_6 \varpi_{23} + B_7 (\varpi_{13} - \varpi_{24}) + B_8 \varpi_{14} - C_5 \varpi_{25} + C_6 \varpi_{15} - C_8 \varpi_{35} - C_9 \varpi_{45}\\
 K^0{}_3 &= -C_4 \varpi_{23} + C_5 (\varpi_{13} - \varpi_{24}) + C_6 \varpi_{14} - D_1 \varpi_{25} + D_2 \varpi_{15} - D_3 \varpi_{35} - D_4 \varpi_{45}
 \end{align*}
 \end{footnotesize}

  Recall from Section \ref{S:LC-parabolic} that $\kappa$ is $P$-equivariant.  Let $\lambda$ be the $P$-representation $\bbK$.  Then
  \[
  R^*_p \kappa = \lambda(p^{-1}) \cdot \kappa \qRa
  \left. \frac{d}{d\epsilon} \right|_{\epsilon = 0} R_{\exp(\epsilon X)}^* \kappa = -\lambda'(X) \cdot \kappa.
  \]
  We let $\delta$ refer to the infinitesimal $\fp$-action. Given $X \in \fp$ as in \eqref{E:p-coords}, we obtain Table \ref{F:K-vert}.
 \begin{footnotesize}
 \begin{table}[h]
 \[
 \begin{array}{rl}
 -\delta A_1 &= (z_1 + z_2 + 4 v_1) A_1 + 4 v_2 A_2\\
 -\delta A_2 &= v_3 A_1 + (z_1 + z_2 + 2 v_1) A_2 + 3 v_2 A_3\\
 -\delta A_3 &= 2 v_3 A_2 + (z_1 + z_2) A_3 + 2 v_2 A_4\\
 -\delta A_4 &= 3 v_3 A_3 + (z_1 + z_2 - 2 v_1) A_4 + v_2 A_5\\
 -\delta A_5 &= 4 v_3 A_4 + (z_1 + z_2 - 4 v_1) A_5\\ \hline
 -\delta B_1 &= (2 z_1 + z_2 + 3 v_1) B_1 + 3 v_2 B_2 + t_1 A_1 + t_2 A_2\\
 -\delta B_2 &= v_3 B_1 + (2 z_1 + z_2 + v_1) B_2 + 2 v_2 B_3 + t_1 A_2 + t_2 A_3\\
 -\delta B_3 &= 2 v_3 B_2 + (2 z_1 + z_2 - v_1) B_3 + v_2 B_4 + t_1 A_3 + t_2 A_4\\
 -\delta B_4 &= 3 v_3 B_3 + (2 z_1 + z_2 - 3 v_1) B_4 + t_1 A_4 + t_2 A_5\\ \hline
 -\delta B_5 &= (z_1 + 2 z_2 + 3 v_1) B_5 + 3 v_2 B_6 + t_4 A_1 - t_3 A_2\\
 -\delta B_6 &= v_3 B_5 + (z_1 + 2 z_2 + v_1) B_6 + 2 v_2 B_7 + t_4 A_2 - t_3 A_3\\
 -\delta B_7 &= 2 v_3 B_6 + (z_1 + 2 z_2 - v_1) B_7 + v_2 B_8 + t_4 A_3 - t_3 A_4\\
 -\delta B_8 &= 3 v_3 B_7 + (z_1 + 2 z_2 - 3 v_1) B_8 + t_4 A_4 - t_3 A_5\\ \hline
 -\delta C_1 &= (3 z_1 + z_2 + 2 v_1) C_1 + 2 v_2 C_2 + 2 t_1 B_1 + 2 t_2 B_2\\
 -\delta C_2 &= v_3 C_1 + (3 z_1 + z_2) C_2 + v_2 C_3 + 2 t_1 B_2 + 2 t_2 B_3\\
 -\delta C_3 &= 2 v_3 C_2 + (3 z_1 + z_2 - 2 v_1) C_3 + 2 t_1 B_3 + 2 t_2 B_4\\ \hline
 -\delta C_4 &= (2 z_1 + 2 z_2 + 2 v_1) C_4 + 2 v_2 C_5 + t_4 B_1 - t_3 B_2 + t_1 B_5 + t_2 B_6\\
 -\delta C_5 &= v_3 C_4 + (2 z_1 + 2 z_2) C_5 + v_2 C_6 + t_4 B_2 - t_3 B_3 + t_1 B_6 + t_2 B_7\\
 -\delta C_6 &= 2 v_3 C_5 + (2 z_1 + 2 z_2 - 2 v_1) C_6 + t_4 B_3 - t_3 B_4 + t_1 B_7 + t_2 B_8\\ \hline
 -\delta C_7 &= (z_1 + 3 z_2 + 2 v_1) C_7 + 2 v_2 C_8 + 2 t_4 B_5 - 2 t_3 B_6\\
 -\delta C_8 &= v_3 C_7 + (z_1 + 3 z_2) C_8 + v_2 C_9 + 2 t_4 B_6 - 2 t_3 B_7\\
 -\delta C_9 &= 2 v_3 C_8 + (z_1 + 3 z_2 - 2 v_1) C_9 + 2 t_4 B_7 - 2 t_3 B_8\\ \hline
 -\delta D_1 &= (3 z_1 + 2 z_2 + v_1) D_1 + v_2 D_2 + t_4 C_1 - t_3 C_2 + 2 t_1 C_4 + 2 t_2 C_5\\
 -\delta D_2 &= v_3 D_1 + (3 z_1 + 2 z_2 - v_1) D_2 + t_4 C_2 - t_3 C_3 + 2 t_1 C_5 + 2 t_2 C_6\\ \hline
 -\delta D_3 &= (2 z_1 + 3 z_2 + v_1) D_3 + v_2 D_4 + 2 t_4 C_4 - 2 t_3 C_5 + t_1 C_7 + t_2 C_8\\
 -\delta D_4 &= v_3 D_3 + (2 z_1 + 3 z_2 - v_1) D_4 + 2 t_4 C_5 - 2 t_3 C_6 + t_1 C_8 + t_2 C_9
 \end{array}
 \]
 \caption{Vertical derivatives of curvature coefficients}
 \label{F:K-vert}
 \end{table}
 \end{footnotesize}
 
 On $\cG$, the curvature coefficients $A,B,C,D$ will satisfy structure equations that also account for variation in the horizontal direction.  These are immediately deduced from Table \ref{F:K-vert}.  For example,
 \begin{align} \label{E:dA1}
 d(A_1) = \delta(A_1) + \alpha_1 = (\zeta_1 + \zeta_2 + 4 \nu_1) A_1 + 4 \nu_2 A_2 + \alpha_1.
 \end{align}
 Here, $\alpha_1$ is a semi-basic form, i.e. it is a linear combination (with coefficients that are functions on $\cG$) of the $\varpi_i$.  We have abused notation in \eqref{E:dA1} by taking a slightly different meaning for $\delta(A_1)$: we have taken the corresponding formula in Table \ref{F:K-vert} and replaced Lie algebra parameters by their corresponding forms in the Cartan connection.  This abuse is justified by axiom (CC.3) in the Cartan connection definition.  Similarly, we will write
 \begin{align*}
 d(A_i) = \delta(A_i) + \alpha_i, \quad
 d(B_i) = \delta(B_i) + \beta_i, \quad
 d(C_i) = \delta(C_i) + \gamma_i, \quad
 d(D_i) = \delta(D_i) + \delta_i.
 \end{align*}
 (The repetition of $\delta$ in the last formula is slightly unfortunate, but should not cause much confusion.)
 
 \subsection{Duality}
 The pullback of the subbundles $E,V \subset TM$ via the projection $\pi : \cG \to M$ are
 \[
 \pi^{-1}(E) = \{ \varpi_3 = \varpi_4 = \varpi_5 = 0 \}, \qquad
 \pi^{-1}(V) = \{ \varpi_1 = \varpi_2 = \varpi_5 = 0 \}.
 \]
 These are interchanged by the duality transformation, a representative of which is
 \begin{align*}
 \iota : &(\varpi_1, \varpi_2, \varpi_3, \varpi_4, \varpi_5, \zeta_1, \zeta_2, \nu_1, \nu_2, \nu_3, \tau_1, \tau_2, \tau_3, \tau_4, \tau_5)\\
 &\mapsto (\varpi_3, \varpi_4, \varpi_1, \varpi_2, -\varpi_5, \zeta_2, \zeta_1, -\nu_1, -\nu_3, -\nu_2, \tau_3, \tau_4, \tau_1, \tau_2, -\tau_5),
 \end{align*}
 which induces
 \begin{align*}
 (A_1,A_2,A_3,A_4,A_5) &\mapsto (A_5,-A_4,A_3,-A_2,A_1)\\
 (B_1,B_2,B_3,B_4,B_5,B_6,B_7,B_8) &\mapsto  (-B_8,B_7,-B_6,B_5,B_4,-B_3,B_2,-B_1)\\
 (C_1,C_2,C_3,C_4,C_5,C_6,C_7,C_8,C_9) &\mapsto  (C_9,-C_8,C_7,-C_6,C_5,-C_4,C_3,-C_2,C_1)\\
 (D_1,D_2,D_3,D_4) &\mapsto (D_4, -D_3, -D_2, D_1)
 \end{align*}
 In particular, the induced action on the quartic is realizable by a $G_0$-transformation, namely that induced by $\rho : (x,y) \mapsto (y,-x)$.  Since any $G_0$-transformation preserves root type, this proves:
 
 \begin{prop}
 The duality transformation preserves root type.
 \end{prop}
 
 However, the duality transformation differs from $(x,y) \mapsto (y,-x)$ on $B,C,D$ coefficients:
 \begin{align*}
 (B_1,B_2,B_3,B_4,B_5,B_6,B_7,B_8) &\mapsto  (-B_4,B_3,-B_2,B_1,-B_8,B_7,-B_6,B_5)\\
 (C_1,C_2,C_3,C_4,C_5,C_6,C_7,C_8,C_9) &\mapsto  (C_3,-C_2,C_1,C_6,-C_5,C_4,C_9,-C_8,C_7)\\
 (D_1,D_2,D_3,D_4) &\mapsto (-D_2, D_1, -D_4, D_3) 
 \end{align*}
 Note that the composition $\rho \circ \iota$ preserves $A$ and induces
 \begin{align*}
 (B_1,B_2,B_3,B_4,B_5,B_6,B_7,B_8) &\mapsto  (-B_5,-B_6,-B_7,-B_8,B_1,B_2,B_3,B_4)\\
 (C_1,C_2,C_3,C_4,C_5,C_6,C_7,C_8,C_9) &\mapsto  (C_7,C_8,C_9,-C_4,-C_5,-C_6,C_1,C_2,C_3)\\
 (D_1,D_2,D_3,D_4) &\mapsto (D_3, D_4, -D_1, -D_2) 
 \end{align*}

 \section{Cartan analysis} \label{S:Cartan}
 
 Starting with the (regular, normal) Cartan geometry $(\cG \to M, \omega)$ which is an equivalent description of any ILC structure, the goal is to classify all homogeneous sub-bundles of total dimension at least six that are obtained via natural reductions of the structure group $P$.  We give an outline of how this is achieved in the type N case.  The analysis for types D and III are similar, so we only provide a few details on how the analysis is begun in these cases.  Types II and I do not contain any multiply transitive structures.  The reader interested in the full details of the Cartan analysis is encouraged to examine the Maple files which accompany the arXiv submission of this paper.
 
 \subsection{Type N reduction}
 
 Using the $P$-action ($G_0$-action), we can always normalize $A = y^4$, i.e. $A_5 = 1, \,A_4 = A_3 = A_2 = A_1 = 0$.\footnote{This normalization is always possible working over $\bbC$, but over $\bbR$ we would have two possibilities: $A = \pm y^4$.}  Now $0 = d(A_i)$ are equivalent to:
 \begin{align} \label{E:N-rel0}
 \alpha_1 = \alpha_2 = \alpha_3 = 0, \quad \nu_2 = \alpha_4, \quad \nu_1 = \frac{1}{4} ( \zeta_1 + \zeta_2 - \alpha_5 ).
 \end{align}
 Differentiating the $\nu_1$-relation in \eqref{E:N-rel0} yields the vertical action on coefficients in $\alpha_5 = a_{5j} \varpi_j$.   (More precisely, we calculate $0 = d(\nu_1 - \frac{1}{4} (\zeta_1 + \zeta_2 - \alpha_5)) \wedge \varpi_i \wedge \varpi_j \wedge \varpi_k \wedge \varpi_\ell$, where $1 \leq i,j,k,l \leq 5$.)
 \begin{align*}
 \delta a_{51} &= \left( \frac{z_2-3z_1}{4} \right) a_{51} + (a_{52} - 4a_{41}) v_3 - 3 t_1, \qquad
 \delta a_{52} = - \left( \frac{5z_1 + z_2}{4} \right) a_{52} - 4 a_{42} v_3 + t_2, \\
 \delta a_{54} &= \left( \frac{z_1 - 3 z_2}{4} \right) a_{54} -(a_{53} + 4 a_{44}) v_3 - 3 t_4, \qquad
  \delta a_{53} = - \left( \frac{z_1 + 5 z_2}{4} \right) a_{53} - 4 a_{43} v_3 + t_3,  \\
 \delta a_{55} &= - a_{55} (z_1 + z_2) - 4a_{45} v_3 - a_{53} t_1 - a_{54} t_2 + a_{51} t_3 + a_{52} t_4 -2t_5.
 \end{align*}
 
 The  $t_j$ induce translations on $a_{5j}$, so we can always normalize $\alpha_5 = 0$.  This forces $t_i = \lambda_i v_3$, where $(\lambda_1,\lambda_2,\lambda_3,\lambda_4) =  (-\frac{4}{3} a_{41}, \, 4 a_{42}, \, 4 a_{43},\, -\frac{4}{3} a_{44},\, -2 a_{45})$.  Hence, there exists functions $T_{ij}$ such that
 \begin{align} \label{E:N-rel1}
 \tau_i = \lambda_i \nu_3 + \sum_j T_{ij} \varpi_j.
 \end{align}
 We have reduced to a 3-dim structure algebra (with parameters $v_3,z_1,z_2$).  We will show that:

  \begin{theorem}
 Any multiply transitive type N structure with the normalizations $A = y^4$ and $\alpha_5 = 0$ satisfies $B = C = D = 0$.
 \end{theorem}

% \[
% \diag\left( \frac{3z_1 + z_2}{4}, \bmat{\frac{z_2}{2} & 0\\ v_3 & -\frac{z_1}{2}}, -\frac{z_1 + 3z_2}{4}\right).
% \]
% Differentiating the $\nu_2$-relation in \eqref{E:N-rel0}, we obtain:
% \begin{align*}
% \delta a_{41} &= - \left( \frac{5z_1 + z_2}{4} \right) a_{41} + 5 a_{42} v_3, \qquad
% \delta a_{42} = -\left( \frac{7z_1+3z_2}{4} \right) a_{42} \\
% \delta a_{44} &= -\left( \frac{z_1+5z_2}{4} \right) a_{44} - 5 a_{43} v_3, \qquad
%  \delta a_{43} = -\left( \frac{3z_1+7z_2}{4} \right) a_{43}, \\
% \delta a_{45} &= - \left( \frac{3z_1 + 3 z_2}{2} \right) a_{45}+ \frac{16}{3} (a_{41} a_{43} - a_{42} a_{44} )v_3
% \end{align*}
 The integrability conditions $0 = d^2 \nu_i \wedge \varpi_5$ $(i=1,2,3)$ force 
 \begin{align*}
 & B_1 = B_2 = B_5 = B_6 = 0, \quad a_{41} = -2 B_4, \quad a_{42}= -2 B_3, \quad a_{43}= +2 B_7, \quad a_{44} = -2 B_8,
% & C_1 = -2 B_3 a_{42}, \quad C_4 = B_3 a_{43} = -B_7 a_{42}, \quad 
 \end{align*}
 and $0 = d(B_1) = d(B_2) = d(B_5) = d(B_6)$ are equivalent to:
 \[
 \beta_1 = \beta_5 = 0, \quad \beta_2 = 2 B_3 \alpha_4, \quad \beta_6 = 2 B_7 \alpha_4.
 \]
 Moreover, $B_3$ and $B_7$ are relative invariants:
 \begin{align*}
 \delta B_3 &= -\left(\frac{7z_1+3z_2}{4}\right) B_3, &  \delta B_4 &= 5 B_3 v_3 -  \left(\frac{5z_1+z_2}{4}\right) B_4, \\
 \delta B_7 &= -\left(\frac{3z_1+7z_2}{4}\right) B_7, &  \delta B_8 &= 5 B_7 v_3 -  \left(\frac{z_1+5z_2}{4}\right) B_8.
% \delta a_{45} &= -\frac{3}{2} (z_1 + z_2) a_{45} - \frac{64}{3} (B_3 B_8 + B_4 B_7)
 \end{align*}
 If $B_3 B_7$ is nowhere vanishing, we can normalize $(B_3,B_4,B_7)=(1,0,1)$.  This trivializes the structure algebra and so such structures admit at most five symmetries (henceforth excluded since these are not multiply transitive).  We have the following trichotomy\footnote{Implicitly, this trichotomy depends on $B_3$ and $B_7$ have locally constant type, i.e. the stated invariant conditions are true locally.  For (multiply) transitive structures, this is always true.}:
 \[
 \begin{array}{ccc}
 \mbox{Condition} & \mbox{Bound on symmetry dimension} \\ \hline
  B_3 = B_7 = 0 & 8\\ 
 (B_3,B_7) \neq (0,0),\, B_3 B_7 = 0 & 6\\
 B_3 B_7 \neq 0 & 5
 \end{array}
 \]
 
 \begin{lemma} No structures with 6 symmetries exist when $(B_3,B_7) \neq (0,0),\, B_3 B_7 = 0$.
 \end{lemma}
 
 \begin{proof} By duality, take $B_3 \neq 0$ and $B_7 = 0$.
 Normalizing $(B_3, B_4) = (1,0)$ forces $z_2 = -\frac{7}{3}z_1$, $v_3 = 0$. The structure algebra is reduced to $\diag(\frac{z_1}{6}, -\frac{7z_1}{6}, -\frac{z_1}{2}, \frac{3z_1}{2})$, and this acts trivially vertically under the 6 symmetry assumption.  Hence, $B_8=0$.  From $0 = d(B_3) = d(B_4) = d(B_7) = d(B_8)$,
 \[
 \zeta_2 = -\frac{7}{3} \zeta_1 + \frac{4}{3} \beta_3, \qquad  
 \nu_3 = \frac{1}{5} (\tau_2 - \beta_4), \qquad \beta_7 = 0, \qquad  \beta_8 = -\tau_3.
 \]
 All coefficients with nonzero (vertical) scaling weight with respect to $z_1$ must vanish.  Differentiating the relations on $\zeta_2, \nu_1, \nu_2, \nu_3, \tau_1, \tau_2, \tau_3, \tau_4, \tau_5$, we conclude from these weights that
 \begin{align*}
 \beta_3 = \beta_4 = \beta_8 = \tau_2 =\tau_3 = \tau_4 = \nu_3 = 0, \quad a_{45} = 0, \quad
  T_{ij} = 0,\quad (i,j) \not\in \{ (1,5), (5,1) \}
 \end{align*}
 But differentiating $\nu_3 = 0$ then yields the contradiction $0 = \varpi_1 \wedge \varpi_4 - (T_{15} + 1) \varpi_2 \wedge \varpi_5$.
 \end{proof}

 Thus, $B_3 = B_7 = 0$ for multiply transitive structures.  Now $0 = d(B_3) = d(B_7)$ implies:
 \[
 \beta_3 = B_4 \alpha_4, \quad \beta_7 = B_8 \alpha_4.
 \]
 Moreover, $B_4, B_8$ are relative invariants.
  \[
 \begin{array}{ccc}
 \mbox{Case} & \mbox{Condition} & \mbox{Bound on symmetry dimension} \\ \hline
 \mbox{(a)} & B_4 = B_8 = 0 & 8\\ 
 \mbox{(b)} & (B_4,B_8) \neq (0,0),\, B_4 B_8 = 0 & 7\\
 \mbox{(c)} & B_4 B_8 \neq 0 & 6
 \end{array}
 \]
 
 \begin{lemma}
 Any multiply transitive type N structure with normalization $A = y^4$ satisfies $B=0$.
 \end{lemma}
 
 \begin{proof}  Suppose $B_4 B_8 \neq 0$.  Normalizing $B_4 = B_8 = 1$ forces $z_1 = z_2 = 0$.  Hence, $\zeta_i = Z_{ij} \varpi_j$.  For multiply transitive structures, $Z_{ij}$ are constant, $\varpi_1,..., \varpi_5, \nu_3$ linearly independent, and $v_3$ must act vertically trivially.  This forces $C_i = 0$ ($i \neq 3,6,9$), $D_1 = D_3 = 0$, 
 \[
 T_{23} = T_{32} = -\frac{16}{9}, \quad
 T_{22} = T_{33} = -\frac{16}{3}, \quad
 Z_{12} = -Z_{23} = 4, \quad
 Z_{13} = -Z_{22} = \frac{4}{3}, \quad
 a_{45} = -\frac{32}{9},
 \]
 and several more linear relations between $T_{ij}$ and $Z_{ij}$.  Since all $B,C,D$ coefficients must be constant, apply $d$ to get further relations.  Imposing Bianchi identities yields $C_3 = C_9 = \frac{8}{3}$ and $C_6 = \frac{4}{9}$.  A contradiction is then obtained from $0 = d^2 \tau_1 \wedge \varpi_3 \wedge \varpi_4 \neq 0$.
 
 The case $B_4 \neq 0$, $B_8=0$ is more involved, but similarly yields a contradiction.
 
 % For at least 6 symmetries, the only possibly nonzero coefficients are $B_4$ and $C_3$ for curvature values.
 
%    From $0 = d(B_4) = d(B_8)$, we have
% \[
% \beta_4 = \frac{5}{4} \zeta_1 + \frac{1}{4} \zeta_2 + \tau_2, \quad
% \beta_8 = \frac{1}{4} \zeta_1 + \frac{5}{4} \zeta_2 - \tau_3.
% \]
% 
% From $0 = d^2 \tau_i \wedge \varpi_5$ for $i=2,3$, we obtain $C_1 = C_2 = C_4 = C_5 = C_7 = C_8 = 0$.
% 
% From $0 = d^2 \nu_i$ ($i=1,2,3$) and $0 = d^2 \tau_i \wedge \varpi_5$ ($i=1,4$), we obtain
% \[
% a_{45} = C_6 - 4, \quad
% b_{43} = 2 C_6 - 2 = -b_{82}, \quad
% b_{42} = 2 - C_3, \quad
% b_{83} = C_9 - 2.
% \]
% Let $\Phi_i := \zeta_i - Z_{ij} \varpi_j = 0$.  Wedging $0 = d\Phi_i$ successively with $\varpi_{2345}, \varpi_{1235}$, and $\varpi_{1234}$ yields
% \[
% Z_{12} = 4, \quad Z_{13} = \frac{4}{3}, \quad
% Z_{22} = -\frac{4}{3}, \quad Z_{23} = -4, \quad
% a_{45} = -\frac{32}{9}.
% \]

 \end{proof}
 
 Given $B=0$, the conditions $0 = d^2 \nu_i$ imply $C_j = 0$ for $j \neq 3,6,9$, and
 \[
 a_{45}=C_6, \quad b_{42}=-C_3, \quad b_{43}=2 C_6, \quad
 b_{82}=-2 C_6, \quad b_{83}=C_9, \quad b_{81} = -b_{44}.
 \] 
 Now imposing $0 = d(B_4) = d(B_8)$ and $0 = d(\nu_2 - \alpha_4)$, we obtain $C_3 = C_6 = C_9 = 0$, so $C=0$, and
 \[
 T_{2i} = 0\,\, (i \neq 1), \quad
 T_{3j} = 0\,\, (j \neq 4), \quad
 T_{21} = b_{41}, \quad T_{34} = -b_{84},  \quad
 b_{44} = b_{45} = b_{85} = 0.
 \]
 For $i = 1,..., 9$, $0 = d(C_i)$ implies $\gamma_i = 0$.  Then $0 = d^2 \tau_1 = d^2\tau_2$ implies $D = 0$.  For $i = 1,...,4$, $0 = d(D_i)$ implies $\delta_i = 0$. Now $0 = d(\nu_1 - \frac{1}{4} (\zeta_1 + \zeta_2))$ implies relations among the $T_{ij}$.  We obtain:
% \[
% T_{13} = T_{42} = T_{52} = T_{53} = 0, \quad
% T_{12} = -\frac{1}{3} T_{21}, \quad
% T_{14} = T_{41}, \quad
% T_{15} = \frac{2}{3} T_{51}, \quad
% T_{34} = -3 T_{43}, \quad
% T_{45} = \frac{2}{3} T_{54}.
% \]
 \begin{align*}
  &\tau_2 = T_{21} \varpi_1, \quad \tau_3 = -3 T_{43} \varpi_4, \quad
 \tau_5 = T_{51} \varpi_1 + T_{54} \varpi_4 + T_{55} \varpi_5,\\ 
 &\tau_1 = T_{11} \varpi_1 - \frac{T_{21}}{3} \varpi_2 + T_{41} \varpi_4 + \frac{2}{3} T_{51} \varpi_5, \quad
 \tau_4 = T_{41} \varpi_1 + T_{43} \varpi_3 + T_{44} \varpi_4 + \frac{2}{3} T_{54} \varpi_5,
 \end{align*}
 and all other $T_{ij}$ not appearing above are zero.  Differentiating the $\tau_i$-relations $(i=1,...,4)$, we obtain the vertical action:
 \begin{align*}
 \delta T_{21} &= -2 T_{21} z_1, \quad \delta T_{11} = \left( \frac{-3z_1+z_2}{2} \right) T_{11} + \frac{2}{3} T_{21} v_3\\
 \delta T_{43} &= -2 T_{43} z_2, \quad \delta T_{44} = \left( \frac{z_1 - 3 z_2}{2} \right) T_{44} + 2 T_{43} v_3
 \end{align*}
 
 \begin{lemma} With normalizations as above, we must have $T_{21} = T_{43} = T_{51} = T_{54} = T_{55} = 0$.
 \end{lemma}
 
 \begin{proof}
 If $T_{21} T_{43} \neq 0$, then there are at most 5 symmetries.  If $T_{21} \neq 0$ and $T_{43} = 0$, normalize $T_{21} = 1$ and $T_{11} = 0$, and write $\zeta_1 = Z_{1j} \varpi_j$, $\nu_3 = V_{3j} \varpi_j$.  We have at most 6 symmetries, and for 6 the residual structure algebra (generated by $z_2$) must act vertically trivially.  This forces $\zeta_1 = \nu_3 = 0$ and $T_{41} = T_{44} = T_{51} = T_{54} = T_{55} = 0$.  But $0 = d\zeta_1 = -2 \varpi_1 \wedge \varpi_2 \neq 0$ yields a contradiction.  The $T_{43} \neq 0$, $T_{21} = 0$ case similarly yields a contradiction.  Thus, we conclude that $T_{21} = T_{43} = 0$ and hence $\tau_2 = \tau_3 = 0$.  From $0 = d\tau_2 = d\tau_3$, we obtain $T_{51} = T_{54} = T_{55} = 0$.
 \end{proof}
 
 SUMMARY: For multiply transitive type $N$ structures, we have reduced to an 8-dimensional subbundle of the original Cartan bundle (given the normalizations $A = y^4$ and $\alpha_5 = 0$).
 \begin{itemize}
 \item Curvature coefficients: All $A,B,C,D$ are zero, except $A_5 = 1$.
 \item Coframe: $\varpi_1,...,\varpi_5, \zeta_1, \zeta_2, \nu_3$.  Relations on other forms:
 \begin{align*}
 \nu_1 = \frac{1}{4} (\zeta_1 + \zeta_2), \quad \nu_2 = \tau_2 = \tau_3 = \tau_5 = 0, \quad
 \tau_1 = T_{11} \varpi_1 + T_{41} \varpi_4, \quad
 \tau_4 = T_{41} \varpi_1 + T_{44} \varpi_4.
 \end{align*}
 \item Among $\alpha_i,\beta_j,\gamma_k,\delta_\ell$, the only possibly nontrivial forms are $\beta_4 = \tau_2$ and $\beta_8 = -\tau_3$.
 \item All Bianchi identities are satisfied, e.g. $0 = d^2 \nu_1$, etc.
 \item Structure group: $\bmat{r_1 & 0 & 0 & 0\\ 0 & r_2 & 0 & 0\\ 0 & s & r_3 & 0 \\ 0 & 0 & 0 & r_4} $, where $r_1 r_2 r_3 r_4 = 1$, $r_1 r_3{}^2 = r_2{}^2 r_4$, i.e. $ r_1 = \frac{1}{r_2{}^4 r_4{}^3}$,\, $r_3 = r_2{}^3 r_4{}^2.$
 This induces  $\left\{ \begin{array}{l} 
 \tilde{T}_{11} = \frac{r_2{}^2}{r_1{}^2} T_{11} = r_2{}^{10} r_4{}^6 T_{11}\\
 \tilde{T}_{41} = \frac{r_2 r_4}{r_1 r_3} T_{41} = r_2{}^2 r_4{}^2 T_{41} \\
 \tilde{T}_{44} = \frac{r_4{}^2}{r_3{}^2} T_{44} = \frac{1}{r_2{}^6 r_4{}^2} T_{44}
  \end{array} \right. \qRa
 \left\{ \begin{array}{l} 
 \delta(T_{11}) = \left( \frac{z_2 - 3 z_1}{2} \right) T_{11} \\
 \delta(T_{41}) = -\left( \frac{z_1+z_2}{2} \right)T_{41} \\ 
 \delta(T_{44}) = \left( \frac{z_1 - 3 z_2}{2} \right) T_{44} 
 \end{array} \right.$.
 \item  Let $\Phi$ be the duality transformation:
 \begin{align*}
 &\varpi_1 \leftrightarrow \varpi_4, \quad \varpi_2 \leftrightarrow \varpi_3, \quad \varpi_5 \mapsto -\varpi_5, \quad \tau_1 \leftrightarrow \tau_4, \quad \tau_2 \leftrightarrow \tau_3, \quad \tau_5 \mapsto -\tau_5,\\
 & \zeta_1 \leftrightarrow \zeta_2, \quad \nu_1 \mbox{ fixed}, \quad
 \nu_2 \mapsto -\nu_2, \quad \nu_3 \mapsto -\nu_3.
 \end{align*}
 This preserves $A$ and induces $(\tau_1,\tau_4) \mapsto (\tau_4, \tau_1)$ and so $(T_{11}, T_{41}, T_{44}) \mapsto (T_{44}, T_{41}, T_{11})$.
 \end{itemize}
 
 The case analysis based on the relative invariants $T_{11}, T_{41}, T_{44}$ is straightforward.  Table \ref{T:N-classify} summarizes this classification and Table~\ref{T:N-streq} contains the structure equations obtained.
 
 \begin{table}[ht]
 \[
 \begin{array}{c||c|c|c}
 & T_{11} = T_{44} = 0 & T_{11} \neq 0, T_{44} = 0 & T_{11} T_{44} \neq 0\\ \hline\hline
 T_{41} = 0 & \mbox{N.8} & \mbox{N.7-1} & \mbox{N.6-2} \\
 T_{41} \neq 0 & \mbox{N.7-2} & \mbox{none} & \mbox{N.6-1}
 \end{array}
 \]
 \caption{Classification of multiply transitive type N structures}
 \label{T:N-classify}
 \end{table}
 
 Some care is required to deduce any redundancy of parameters appearing in the structure equations.  Consider the case $T_{11} \neq 0$.  Normalize $T_{11} = 1$, so $r_4{}^6= \frac{1}{r_2{}^{10}}$, and $\zeta_2 = 3 \zeta_1 + Z_{21} \varpi_1 + Z_{24} \varpi_4$ is forced.  Quotienting the structure group by $\diag(r_2,r_2,r_2,r_2)$ (since these act trivially), we may WLOG take the diagonal to be $\diag\left(\frac{1}{r_2{}^5 r_4{}^3},1, r_2{}^2 r_4{}^2, \frac{r_4}{r_2}\right)$.  Let $Q$ be the residual group below.
 
 \begin{enumerate}
 \item \framebox{N.6-1}: $T_{41} = 1$ and $\zeta_1 = Z_{11} \varpi_1 + Z_{14} \varpi_4$.  Then $Q \cong \bbZ_2$, generated by $E = \diag(-1,1,1,-1)$.
 %Normalize $T_{41}=1$, so $r_2{}^4 = 1$ and $r_4{}^6 =  \frac{1}{r_2{}^2}$.  Since $r_2{}^2 r_4{}^2 = 1$, then $\left( \frac{r_4}{r_2}\right)^2 = 1$.  The residual structure group is $\bbZ_2$ given by $\diag\left( \frac{r_2}{r_4}, 1, 1, \frac{r_4}{r_2} \right)$, i.e. it is generated by $E = \diag(-1,1,1,-1)$.
 \[
 \tilde{T}_{44} = T_{44}, \quad \tilde{Z}_{j1} = - Z_{j1}, \quad \tilde{Z}_{j4} = - Z_{j4} \qquad (j=1,2).
 \]
 Let $Z_{24} = 4a$, so $\pm a$ yield the same structure.  We must $a(a^2+1) \neq 0$ and
 \[
 Z_{11} = \frac{1-2a^2}{a}, \quad Z_{14} = \frac{-a(2a^2+ 1)}{a^2 + 1}, \quad
 T_{44} = \frac{a^4}{(a^2 + 1)^2}, \quad Z_{21} = \frac{4(a^2 - 1)}{a}.
 \]
 Thus, $a^2 \in \bbC \backslash \{ 0, -1 \}$ is the essential parameter.
% These structures are self-dual via the duality (acting on the dual framing):
% \[
% \tilde{e}_1 = \frac{a^2+1}{a^2} e_4, \quad
% \tilde{e}_2 = \frac{a^2+1}{a^2} e_3, \quad
% \tilde{e}_3 = \frac{a^2}{a^2+1} e_2, \quad
% \tilde{e}_4 = \frac{a^2}{a^2+1} e_1, \quad
% \tilde{e}_5 = -e_5, \quad
% \tilde{e}_6 = -e_6.
% \]
 \item \framebox{N.6-2}: $T_{44} = 1$ and $\zeta_1 = -\frac{3}{8} Z_{21} \varpi_1 + Z_{14} \varpi_4$.  Write $Z_{21} = -4a$ and $Z_{14} = \frac{b}{2}$.  Here, $Q \cong \bbZ_2 \times \bbZ_2$ generated by $M_1 = \diag(-1,1,1,-1)$ and $M_2 = \diag(1,1,-1,1)$.
%  so $(r_2 r_4)^4 = 1$, $r_2{}^2 = r_4{}^2$.  Thus, $r_4{}^8 = 1 = r_2{}^8$.  The residual structure group is the Klein 4-group $\bbZ_2 \times \bbZ_2$ given by $\diag\left(\frac{r_4}{r_2},1,r_4{}^4,\frac{r_4}{r_2} \right)$, i.e. it is generated by $M_1 = \diag(-1,1,1,-1)$ and $M_2 = \diag(1,1,-1,1)$.
 Then:
 \[
 M_1: \quad (\tilde{a},\tilde{b}) = (-a,-b); \qquad
 M_2: \quad (\tilde{a},\tilde{b}) = (a,-b).
 \]
 Thus, $(a^2,b^2) \in \bbC \times \bbC$ is the essential parameter.
 
% \framebox{Double check: Self-dual $\iff$ $a^2 = b^2$}
 
 \item \framebox{N.7-1}: $T_{41} = T_{44} = 0$, $\zeta_2 = 4a \varpi_1 + 3 \zeta_1$.  Here, $Q$ has diagonal $\diag\left(\frac{1}{r_2{}^5 r_4{}^3},1, r_2{}^2 r_4{}^2, \frac{r_4}{r_2}\right)$, with $\epsilon := r_2{}^5 r_4{}^3 = \pm 1$.  Induced action: $\tilde{a} = \epsilon a$.  Thus, $a^2 \in \bbC$ is the essential parameter.
\end{enumerate}
 All type N structure equations are given in Table~\ref{T:N-streq}.
  \begin{footnotesize}
 \begin{longtabu}[HT]{|c|>{$}c<{$}|>{$}l<{$}|>{$}l<{$}|}
 \caption{Multiply transitive type N structure equations}\label{T:N-streq}\\
  \hline
 Model & \mbox{SD} & \multicolumn{1}{c|}{\mbox{Structure equations}} & \multicolumn{1}{c|}{\mbox{Embedding into Cartan bundle}}\\ \hline
\endfirsthead
 \caption[]{Multiply transitive type N structure equations (continued)}\\
  \hline
 Model & \mbox{SD} & \multicolumn{1}{c|}{\mbox{Structure equations}} & \multicolumn{1}{c|}{\mbox{Embedding into Cartan bundle}}\\ \hline
 \endhead
 N.8 & \cmark &
  \begin{array}{l@{\,}l}
 d\varpi_1&=\frac{3}{4}\zeta_1 \wedge \varpi_1 - \frac{1}{4}\zeta_2 \wedge \varpi_1\\
 d\varpi_2&=\frac{5}{4}\zeta_1 \wedge \varpi_2 + \frac{1}{4}\zeta_2 \wedge \varpi_2 - \nu_3 \wedge \varpi_1\\
 d\varpi_3&=\frac{1}{4}\zeta_1 \wedge \varpi_3 + \frac{5}{4}\zeta_2 \wedge \varpi_3 + \nu_3 \wedge \varpi_4\\
 d\varpi_4&=-\frac{1}{4}\zeta_1 \wedge \varpi_4 + \frac{3}{4}\zeta_2 \wedge \varpi_4\\
d\varpi_5&=\zeta_1 \wedge \varpi_5 + \zeta_2 \wedge \varpi_5 + \varpi_1 \wedge \varpi_3 + \varpi_2 \wedge \varpi_4\\
 d\nu_3&=\frac{1}{2} \zeta_1 \wedge \nu_3 + \frac{1}{2}\zeta_2 \wedge \nu_3 + \varpi_1 \wedge \varpi_4\\
 d\zeta_1&=0\\
 d\zeta_2&=0\\
 \end{array} &
 \begin{array}{l@{\,\,}l}
 \nu_1 &= \frac{1}{4} (\zeta_1 + \zeta_2)\\
 \nu_2 &= 0\\
 \tau_1 &= \tau_2 = \tau_3 = \tau_4 = \tau_5 = 0
 \end{array}
 \\ \hline
%%%%%%%%%%%%%%%%%%%%%%%%%%%%%%
 N.7-1 & \xmark &
  \begin{array}{l@{\,}l}
 d\varpi_1 &= 0\\
 d\varpi_2 &= 2\zeta_1 \wedge \varpi_2 - \nu_3 \wedge \varpi_1 + a\,\varpi_1 \wedge \varpi_2\\
 d\varpi_3 &= 4\zeta_1 \wedge \varpi_3 + \nu_3 \wedge \varpi_4 + \varpi_1 \wedge \varpi_5 + 5 a\,\varpi_1 \wedge \varpi_3 \\
 d\varpi_4 &= 2\zeta_1 \wedge \varpi_4 + 3 a\,\varpi_1 \wedge \varpi_4\\
d\varpi_5 &= 4\zeta_1 \wedge \varpi_5 + \varpi_1 \wedge \varpi_3 + \varpi_2 \wedge \varpi_4 + 4a\,\varpi_1 \wedge \varpi_5\\
 d\nu_3 &= 2\zeta_1 \wedge \nu_3 - 2a\,\nu_3 \wedge \varpi_1 + \varpi_1 \wedge \varpi_2  + \varpi_1 \wedge \varpi_4\\
 d\zeta_1 &= 0\\
 \end{array} &
 \begin{array}{l@{\,\,}l}
 \zeta_2 &= 3\zeta_1 + 4 a \varpi_1\\
 \nu_1 &= \zeta_1 + a \varpi_1\\
 \tau_1 &= \varpi_1\\
 \nu_2 &= \tau_2 = \tau_3 = \tau_4 = \tau_5 = 0\\
 \multicolumn{2}{c}{\mbox{(Parameter: $a^2 \in \bbC$)}}
 \end{array}
 \\ \hline
%%%%%%%%%%%%%%%%%%%%%%%%%%%%%%
 N.7-2 & \cmark &
  \begin{array}{l@{\,}l}
 d\varpi_1 &= \zeta_1 \wedge \varpi_1\\
 d\varpi_2 &= \zeta_1 \wedge \varpi_2 - \nu_3 \wedge \varpi_1 - \varpi_1 \wedge \varpi_5\\
 d\varpi_3 &= -\zeta_1 \wedge \varpi_3 + \nu_3 \wedge \varpi_4 + \varpi_4 \wedge \varpi_5\\
 d\varpi_4 &= -\zeta_1 \wedge \varpi_4\\
d\varpi_5 &= \varpi_1 \wedge \varpi_3 + \varpi_2 \wedge \varpi_4\\
 d\nu_3 &= -\varpi_2 \wedge \varpi_4 - \varpi_1 \wedge \varpi_3 + \varpi_1 \wedge \varpi_4\\
 d\zeta_1 &= 2 \varpi_1 \wedge \varpi_4\\
 \end{array} &
 \begin{array}{l@{\,\,}l}
 \zeta_2 &= -\zeta_1\\
 \tau_1 &= \varpi_4 \\
 \tau_4 &= \varpi_1\\
 \nu_1 &= \nu_2 = \tau_2 = \tau_3 = \tau_5 = 0
 \end{array}
 \\ \hline
%%%%%%%%%%%%%%%%%%%%%%%%%%%%%%
 N.6-1 & \cmark &
  \begin{array}{l@{\,}l}
d\varpi_1 &= a \varpi_1 \wedge \varpi_4\\
d\varpi_2 &= -\nu_3 \wedge \varpi_1-\frac{(3a^2-1)}{a}\varpi_1 \wedge \varpi_2 - \varpi_1 \wedge \varpi_5 \\
&\qquad + \frac{a(3a^2+1)}{(a^2+1)}\varpi_2 \wedge \varpi_4 - \frac{a^4}{(a^2+1)^2}\varpi_4 \wedge \varpi_5\\
d\varpi_3 &= \nu_3 \wedge \varpi_4 - \frac{(3a^2+1)}{a}\varpi_1 \wedge \varpi_3 + \varpi_1 \wedge \varpi_5 \\
&\qquad + \frac{a(3a^2-1)}{(a^2+1)}\varpi_3 \wedge \varpi_4 +\varpi_4 \wedge \varpi_5\\
d\varpi_4 &= -\frac{(a^2+1)}{a}\varpi_1 \wedge \varpi_4\\
d\varpi_5 &= \varpi_1 \wedge \varpi_3 + \varpi_2 \wedge \varpi_4 - 4a\left( \varpi_1 + \frac{a^2}{(a^2+1)}\varpi_4 \right) \wedge \varpi_5\\
d\nu_3 &= 2a \nu_3 \wedge \left( \varpi_1 + \frac{a^2}{(a^2+1)} \varpi_4 \right) + \varpi_1 \wedge \varpi_2 - \varpi_1 \wedge \varpi_3 \\
&\qquad + \varpi_1 \wedge \varpi_4 - \varpi_2 \wedge \varpi_4 +\frac{a^4}{(a^2+1)^2}\varpi_3 \wedge \varpi_4
 \end{array} &
 \begin{array}{l@{\,\,}l}
 \zeta_1 &= -\frac{(2 a^2-1)}{a}\varpi_1 - \frac{a(2 a^2+1)}{(a^2+1)} \varpi_4 \\
 \zeta_2 &= -\frac{(2a^2+1)}{a}\varpi_1 - \frac{a(2 a^2-1)}{(a^2+1)} \varpi_4\\
 \nu_1 &= -a\left( \varpi_1 + \frac{a^2}{(a^2+1)} \varpi_4 \right)\\
 \tau_1 &= \varpi_1+\varpi_4\\
 \tau_4 &= \varpi_1 + \frac{a^4}{(a^2+1)^2} \varpi_4 \\
 \nu_2 &= \tau_2 = \tau_3 = \tau_5 = 0\\ 
 \multicolumn{2}{c}{\mbox{(Parameter: $a^2 \in \bbC \backslash \{ 0, -1 \}$)}}
 \end{array}
 \\ \hline
%%%%%%%%%%%%%%%%%%%%%%%%%%%%%%
N.6-2 & \begin{array}{l} \cmark : \\
a^2 = b^2\\ \xmark : \\
a^2 \neq b^2 \end{array} & 
  \begin{array}{l@{\,}l}
 d\varpi_1 &= 0\\
 d\varpi_2 &= -\nu_3 \wedge \varpi_1 + 2a\varpi_1 \wedge \varpi_2 - b\varpi_2 \wedge \varpi_4 - \varpi_4 \wedge \varpi_5 \\
 d\varpi_3 &= +\nu_3 \wedge \varpi_4 + a\varpi_1 \wedge \varpi_3 - 2b\varpi_3 \wedge \varpi_4 + \varpi_1 \wedge \varpi_5 \\
 d\varpi_4 &= 0\\
 d\varpi_5 &= \varpi_1 \wedge \varpi_3 + \varpi_2 \wedge \varpi_4 + 2a\varpi_1 \wedge \varpi_5 + 2b\varpi_4 \wedge \varpi_5 \\
 d\nu_3 &= -\nu_3 \wedge (a\varpi_1 + b\varpi_4) + \varpi_1 \wedge \varpi_2 \\
 &\qquad+ \varpi_3 \wedge \varpi_4 + \varpi_1 \wedge \varpi_4
 \end{array} &
 \begin{array}{l@{\,\,}l}
 \zeta_1 &= \frac{3a}{2} \varpi_1 + \frac{b}{2} \varpi_4\\
 \zeta_2 &= \frac{a}{2} \varpi_1 + \frac{3b}{2} \varpi_4\\
 \nu_1 &= \frac{a}{2} \varpi_1 + \frac{b}{2} \varpi_4\\
 \tau_1 &= \varpi_1, \, \tau_4 = \varpi_4\\
 \nu_2 &= \tau_2 = \tau_3 = \tau_5 = 0\\
 \multicolumn{2}{c}{\mbox{(Parameter: $(a^2,b^2) \in \bbC \times \bbC$)}}
 \end{array}
 \\ \hline
  \end{longtabu}
 \end{footnotesize}

 \subsection{Type $D$ reduction}

 Normalize $A = 6 x^2 y^2$, i.e. $A_3 = 1, A_5 = A_4 = A_2 = A_1 = 0$.  Now $0 = d(A_i)$ implies:
 \begin{align} \label{E:D-rel0}
 \nu_2 = \frac{1}{3} \alpha_2, \quad \nu_3 = \frac{1}{3} \alpha_4, \quad \zeta_2 = -\zeta_1 + \alpha_3, \quad \alpha_1 = \alpha_5 = 0.
 \end{align}
 Differentiating the $\zeta_2$-relation above yields the vertical action on coefficients in $\alpha_3 = a_{3j} \varpi_j$:
 \begin{align*}
 \delta a_{31} &= a_{31} (v_1 - z_1) - t_1, \qquad
 \delta a_{32} = -a_{32} (v_1 + z_1) - t_2, \qquad
 \delta a_{33} = a_{33} (-v_1 + z_1) - t_3, \\
 \delta a_{34} &= a_{34} (v_1 + z_1) - t_4, \qquad
 \delta a_{35} = a_{32} t_4 + a_{31} t_3 - a_{34} t_2 - a_{33} t_1 - 2 t_5.
 \end{align*}
 Normalize $\alpha_3 = 0$, so $\tau_i = T_{ij} \varpi_j$.  We have reduced to the 2-dimensional structure algebra $\diag\left(  \frac{z_1}{2}, v_1 - \frac{z_1}{2}, -v_1 - \frac{z_1}{2}, \frac{z_1}{2} \right)$, so all type D structures admit at most seven symmetries.  Using duality and the $G_0$-map $(x,y) \mapsto (y,-x)$, we can assume that $B_1$ or $B_2$ is nonzero, or $B = 0$.

%\medskip
%\tikzstyle{level 1}=[level distance=2cm, sibling distance=0.5cm]
%\tikzstyle{level 2}=[level distance=2cm, sibling distance=1.5cm]
%\begin{figure}[h]
%\begin{tikzpicture}
%\begin{minipage}[c]{13cm}
%\Tree[.\framebox{\parbox{14mm}{Type D}}
%     \edge  node[auto=right] {\tiny $B = 0$};
%	[.  \framebox{}
%			 \edge  node[auto=right] {\tiny $C_2=C_8 = 0$};
%		{\framebox{D.7}}
%	     \edge  node[auto=left] {\tiny $C_2 C_8 \neq 0$};
%	    {\framebox{D.6-3}}
%	     \edge  node[auto=left] {\tiny $C_2 \neq 0, C_8=0$};
%		{\framebox{D.6-4}}
%	] 
%  \edge  node[auto=right] {\tiny $B_1\neq 0$};
%   [.\framebox{D.6-1} ] 
%     \edge  node[auto=left] {\tiny $B_2\neq 0$};
%   [.\framebox{D.6-2} ]
% ]
%\end{minipage}
%\end{tikzpicture}
%\caption{Decision tree for type D, multiply transitive ILC structures}
%\label{F:ILC-D}
%\end{figure}

 For the 7-symmetry case, the 2-dimensional structure algebra must act trivially.  This forces:
 \begin{itemize}
 \item only $C_5$ (necessarily constant) to survive among $B,C,D$ coefficients;
 \item $\alpha_2 = \alpha_4 = 0$ (so $\nu_2 = \nu_3 = 0$);
 \item all $T_{ij}$ to vanish except $T_{13}, T_{24}, T_{31}, T_{42}, T_{55}$ (necessarily constants).
 \end{itemize}
 From $0 = d\nu_2 = d(\zeta_1+\zeta_2) = d(\tau_1 - T_{13} \varpi_3) = d(\tau_2 - T_{24} \varpi_4)$, we obtain
 \[
 T_{31} = T_{13} = C_5 - \frac{1}{2}, \qquad T_{42} = T_{24} = -C_5 - \frac{1}{2}, \qquad T_{55} = (C_5)^2+\frac{1}{4}.
 \]
 Now, $0 = d(B_i) = d(C_j) = d(D_k)$ forces $\beta_i = 0$ ($i=1,4,5,8$), and $\gamma_j = 0$, and
 \begin{align*}
 (\beta_2,\beta_3,\beta_6,\beta_7) &= (\tau_2, \tau_1, -\tau_3, \tau_4), \qquad
 (\delta_1,\delta_2,\delta_3,\delta_4) = 2C_5 (\beta_2,\beta_3,\beta_6,\beta_7).
 \end{align*}
 This yields model D.7.  The 2-dimensional structure group is generated by $\diag(r_1,\frac{r_2}{r_1},\frac{1}{r_1 r_2}, r_1)$  $(r_1,r_2 \in \bbC^\times)$, along with $\diag(e^{i\pi/4},e^{i\pi/4},-e^{i\pi/4},e^{i\pi/4})$ and $\diag\left(1,\begin{bmatrix} 0 & 1\\ -1 & 0 \end{bmatrix},1\right)$.\footnote{The latter two correspond to $(x,y) \mapsto (x,-y)$ and $(x,y) \mapsto (-y,x)$.}  Only this last transformation acts non-trivially on $C_5$, i.e. $C_5 \mapsto -C_5$.  Thus, $(C_5)^2 \in \bbC$ is the essential parameter.
 
 The 6-symmetry case proceeds similarly, but is very tedious, particularly for the $B_2 \neq 0$ case that leads to model D.6-2.  All type D structure equations are given in Table~\ref{T:D-streq}.

 \begin{footnotesize}
 \begin{longtabu}{|>{$}c<{$}|>{$}c<{$}|>{$}l<{$}|>{$}l<{$}|}  
\caption{Multiply transitive type D structure equations}\label{T:D-streq}\\
 \hline
 \mbox{Model} & \mbox{SD} & \multicolumn{1}{c|}{\mbox{Structure equations}} & \multicolumn{1}{c|}{\mbox{Embedding into Cartan bundle}}\\ 
 \hline
 \endhead
 \begin{array}{@{}c@{}}\mbox{D.7} \\ (a^2 \in \bbC) \end{array}& \cmark &
 \begin{array}{l@{\,}l}
 d\varpi_1&= (\zeta_1 - \nu_1 + (a-\frac{1}{2})\varpi_5)  \wedge \varpi_1\\
 d\varpi_2&= (\zeta_1 + \nu_1 - (a+\frac{1}{2})\varpi_5) \wedge \varpi_2\\
 d\varpi_3&= (-\zeta_1 + \nu_1 - (a-\frac{1}{2})\varpi_5) \wedge \varpi_3\\
 d\varpi_4&= (-\zeta_1 - \nu_1 + (a+\frac{1}{2})\varpi_5) \wedge \varpi_4\\
 d\varpi_5&= \varpi_1 \wedge \varpi_3 + \varpi_2 \wedge \varpi_4\\
 d\zeta_1 &= (2 a-1)\varpi_1 \wedge \varpi_3 - (2 a+1)\varpi_2 \wedge \varpi_4\\
 d\nu_1 &= (-a+\frac{3}{2})\varpi_1 \wedge \varpi_3 - (a+\frac{3}{2})\varpi_2 \wedge \varpi_4
 \end{array} & \begin{array}{l@{\,\,}l}
 \zeta_2 &= -\zeta_1, \,
 \nu_2 = \nu_3 = 0 \\ 
 \tau_1 &= (a - \frac{1}{2}) \varpi_3 \\
 \tau_2 &= -(a+\frac{1}{2}) \varpi_4 \\
 \tau_3 &= (a - \frac{1}{2}) \varpi_1 \\
 \tau_4 &= -(a+\frac{1}{2}) \varpi_2 \\
 \tau_5 &= (a^2 + \frac{1}{4}) \varpi_5
 \end{array}\\ \hline
%  \begin{bmatrix} z & 0 & 0 & 0 \\ s_3 & 0 & 0 & 0\\ 0 & 0 & 0 & 0\\ 0 & 0 & t_3 & -z \end{bmatrix}\\ \hline  
\mbox{D.6-1} & \cmark &
 \begin{array}{l@{\,}l}
 d\varpi_1&= (3 \varpi_5 - 4 \nu_1) \wedge \varpi_1\\
 d\varpi_2&= (\frac{3}{2} \varpi_5 - 2 \nu_1) \wedge \varpi_2 + \sqrt{2}\, \varpi_1 \wedge \varpi_4\\
 d\varpi_3&= (4\nu_1 - 3 \varpi_5) \wedge \varpi_3\\
 d\varpi_4&= (2\nu_1 - \frac{3}{2} \varpi_5) \wedge \varpi_4 + \sqrt{2}\, \varpi_3 \wedge \varpi_2\\
 d\varpi_5&= \varpi_1 \wedge \varpi_3 + \varpi_2 \wedge \varpi_4\\
 d\nu_1 &= -\frac{1}{4} \varpi_1 \wedge \varpi_3 + \frac{3}{4} \varpi_2 \wedge \varpi_4
 \end{array} & \begin{array}{l@{\,\,}l} 
 \nu_2 &= -\sqrt{2} \varpi_2,\, \nu_3 = \sqrt{2} \varpi_4\\
 \zeta_2 &= -\zeta_1 = -\frac{7}{4} \varpi_5 + 3 \nu_1\\
 \tau_1 &= \frac{5}{4} \varpi_3,\, \tau_2 = -\frac{1}{4} \varpi_4\\
 \tau_3 &= \frac{5}{4} \varpi_1,\, \tau_4 = -\frac{1}{4} \varpi_2\\
 \tau_5 &= \frac{13}{16} \varpi_5
 \end{array}\\ \hline
 %%%%%%%%%%%%%%%%%%%%%%%%%%%%%%%%%%%%%%%
 \begin{array}{@{}c@{}} \mbox{D.6-2}\\ (a \neq 1, \frac{2}{3}) \end{array} & \cmark &
 \begin{array}{l@{\,}l}
 d\varpi_1&= -2\nu_1 \wedge \varpi_1 + (-a + \frac{2}{3}) \varpi_1 \wedge \varpi_2 \\
 &\qquad - \frac{9a}{(3a-2)(a-1)} \varpi_1 \wedge \varpi_4 
  - \frac{(3a+5)}{2(a-1)} \varpi_1 \wedge \varpi_5\\
 d\varpi_2&=  -\frac{6}{(3a-2)(a-1)} \varpi_1 \wedge \varpi_3 - \frac{9a}{(3a-2)(a-1)} \varpi_2 \wedge \varpi_4 \\
 &\qquad - \frac{2}{a-1} \varpi_2 \wedge \varpi_5 + \frac{18}{(3a-2)(a-1)^2} \varpi_4 \wedge \varpi_5 \\
 d\varpi_3&= 2 \nu_1 \wedge \varpi_3 - a \varpi_2 \wedge \varpi_3 + \frac{3}{a-1} \varpi_3 \wedge \varpi_4  \\
 &\qquad + \frac{3a+5}{2(a-1)} \varpi_3 \wedge \varpi_5\\
 d\varpi_4&= -\frac{2}{3} \varpi_1 \wedge \varpi_3 - a \varpi_2 \wedge \varpi_4 \\
 &\qquad + (\frac{4}{9} - \frac{2}{3} a) \varpi_2 \wedge \varpi_5 + \frac{2}{a-1} \varpi_4 \wedge \varpi_5\\
 d\varpi_5&= \varpi_1 \wedge \varpi_3 + \varpi_2 \wedge \varpi_4\\
 d\nu_1 &= \frac{9a^2-15a-2}{4(3a-2)(a-1)} \varpi_1 \wedge \varpi_3 - \frac{27a^2-21a+10}{4(3a-2)(a-1)} \varpi_2 \wedge \varpi_4 \\
 &\qquad - \frac{2(3a-1)}{3(a-1)} \varpi_2 \wedge \varpi_5 + \frac{6(3a-1)}{(a-1)^2 (3a-2)} \varpi_4 \wedge \varpi_5
 \end{array} & \begin{array}{l@{\,\,}l}
 \zeta_2 &= -\zeta_1 = \nu_1 - a \varpi_2 - \frac{9a}{(a-1)(3a-2)} \varpi_4 \\
 &\qquad\qquad - \frac{15 a^2+23 a-14}{4(a-1)(3a-2)} \varpi_5\\
 \nu_2 &= -\frac{2}{3} \varpi_1\\
 \nu_3 &= -\frac{6}{(a-1)(3a-2)} \varpi_3\\
 \tau_1 &= \frac{3a^2-5a-6}{4(a-1)(3a-2)} \varpi_3\\
 \tau_2 &= (\frac{4}{9} - \frac{2}{3} a) \varpi_2 - \frac{15 a^2 - a + 2}{4(a-1)(3a-2)} \varpi_4 \\
 &\qquad -\frac{2(3a+1)}{3(a-1)} \varpi_5\\
 \tau_3 &= \frac{3a^2-5a-6}{4(a-1)(3a-2)} \varpi_1\\
 \tau_4 &= -\frac{15a^2-a+2}{4(a-1)(3a-2)} \varpi_2 - \frac{18}{(3a-2)(a-1)^2} \varpi_4 \\
 &\qquad - \frac{6(3a+1)}{(3a-2)(a-1)^2} \varpi_5\\
 \tau_5 &= -\frac{3a+1}{3(a-1)}\varpi_2 - \frac{3(3a+1)}{(3a-2)(a-1)^2} \varpi_4 \\
 &\qquad + \frac{117 a^4-462 a^3-407 a^2+604 a-44}{16(a-1)^2 (3a-2)^2} \varpi_5
 \end{array}\\ \hline
 %%%%%%%%%%%%%%%%%%%%%%%%%%%%%%%%%%%%%%% 
 \begin{array}{@{}c@{}}\mbox{D.6-3} \\ (a \neq 0) \end{array}& \cmark &
 \begin{array}{l@{\,}l}
 d\varpi_1&= \frac{3}{2} \varpi_1 \wedge \varpi_5 - \nu_1 \wedge \varpi_1 - \frac{1}{2} a \varpi_4 \wedge \varpi_5\\
 d\varpi_2&= \frac{3}{2} \varpi_2 \wedge \varpi_5 + \nu_1 \wedge \varpi_2 - \frac{1}{2} a \varpi_3 \wedge \varpi_5\\
 d\varpi_3&= \frac{1}{2} a \varpi_2 \wedge \varpi_5 - \frac{3}{2} \varpi_3 \wedge \varpi_5 + \nu_1 \wedge \varpi_3\\
 d\varpi_4&= \frac{1}{2} a \varpi_1 \wedge \varpi_5 - \frac{3}{2} \varpi_4 \wedge \varpi_5 - \nu_1 \wedge \varpi_4\\
 d\varpi_5&= \varpi_1 \wedge \varpi_3 + \varpi_2 \wedge \varpi_4\\
 d\nu_1 &= -\frac{1}{2} a \varpi_1 \wedge \varpi_2 + \frac{3}{2} \varpi_1 \wedge \varpi_3 \\
 & \qquad - \frac{3}{2} \varpi_2 \wedge \varpi_4 + \frac{1}{2} a \varpi_3 \wedge \varpi_4
 \end{array} & \begin{array}{l@{\,\,}l} 
 \zeta_2 &= -\zeta_1 = \varpi_5,\,\,  \nu_2 = \nu_3 = 0\\
 \tau_1 &= \frac{1}{2} a \varpi_2- \frac{1}{2} \varpi_3\\
 \tau_2 &= \frac{1}{2} a \varpi_1 - \frac{1}{2} \varpi_4 \\
 \tau_3 &= -\frac{1}{2} \varpi_1 + \frac{1}{2} a \varpi_4 \\
 \tau_4 &= -\frac{1}{2} \varpi_2 + \frac{1}{2} a \varpi_3 \\
 \tau_5 &= \frac{1}{4}( 1 - a{}^2) \varpi_5
 \end{array}\\ \hline
 %%%%%%%%%%%%%%%%%%%%%%%%%%%%%%%%%%%%%%%
\mbox{D.6-4} & \xmark &
 \begin{array}{l@{\,}l}
 d\varpi_1&= \frac{3}{2} \varpi_1 \wedge \varpi_5 - \nu_1 \wedge \varpi_1\\
 d\varpi_2&= \frac{3}{2} \varpi_2 \wedge \varpi_5 + \nu_1 \wedge \varpi_2\\
 d\varpi_3&= -\frac{1}{2} \varpi_2 \wedge \varpi_5 - \frac{3}{2} \varpi_3 \wedge \varpi_5 + \nu_1 \wedge \varpi_3\\
 d\varpi_4&= -\frac{1}{2} \varpi_1 \wedge \varpi_5 - \frac{3}{2} \varpi_4 \wedge \varpi_5 - \nu_1 \wedge \varpi_4\\
 d\varpi_5&= \varpi_1 \wedge \varpi_3 + \varpi_2 \wedge \varpi_4\\
 d\nu_1 &= \frac{1}{2} \varpi_1 \wedge \varpi_2 + \frac{3}{2} \varpi_1 \wedge \varpi_3 - \frac{3}{2} \varpi_2 \wedge \varpi_4
 \end{array} & \begin{array}{l@{\,\,}l} 
 \zeta_2 &= -\zeta_1 = \varpi_5,\,\,  \nu_2 = \nu_3 = 0\\
 \tau_1 &= -\frac{1}{2} \varpi_2- \frac{1}{2} \varpi_3\\
 \tau_2 &= -\frac{1}{2} \varpi_1 - \frac{1}{2} \varpi_4 \\
 \tau_3 &= -\frac{1}{2} \varpi_1\\
 \tau_4 &= -\frac{1}{2} \varpi_2\\
 \tau_5 &= \frac{1}{4} \varpi_5
 \end{array}\\ \hline
 %%%%%%%%%%%%%%%%%%%%%%%%%%%%%%%%%%%%%%%
 \end{longtabu}
 \end{footnotesize}

 \subsection{Type III reduction}
 
 Normalize $A = xy^3$, i.e. $A_4 = \frac{1}{4}, \, A_5 = A_3 = A_2 = A_1 = 0$.  Then $0 = d(A_i)$ implies:
 \begin{align*}
 \alpha_1 = \alpha_2 = 0, \quad \nu_2 = 2 \alpha_3, \quad 
 \nu_1 = \frac{1}{2} (\zeta_1 + \zeta_2) - 2\alpha_4, \quad \nu_3 = \alpha_5.
 \end{align*}
 Differentiating the $\nu_i$-relations above yields
 \begin{align*}
 2\delta a_{31} &= -(3z_1 + z_2) a_{31} + t_2, \quad
 2\delta a_{34} = -(z_1 + 3z_2) a_{34} - t_3, \\
 2\delta a_{41} &= (z_2 - z_1) a_{41} - t_1, \quad
 2\delta a_{44} = (z_1 - z_2) a_{44} - t_4.
 \end{align*}
 Normalize $a_{31} = a_{34} = a_{41} = a_{44} = 0$ (so $t_1 = t_2 = t_3 = t_4 = 0$).  Then $2\delta a_{45} = -2 (z_1 + z_2) a_{45} - t_5$, so normalize $a_{45} = 0$, and let $\tau_i = T_{ij} \varpi_j$.  We have reduced to a 7-dimensional subbundle with:
 \begin{align*}
 \begin{array}{rl}
 \delta a_{42} &= -\left( \frac{3z_1 + z_2}{2} \right) a_{42}\\
 \delta a_{43} &= -\left( \frac{z_1 + 3z_2}{2} \right) a_{43}\\
 \delta a_{32} &= -\left( \frac{5z_1 + 3z_2}{2} \right) a_{32}\\
 \delta a_{33} &= -\left( \frac{3z_1 + 5z_2}{2} \right) a_{33}\\
 \delta a_{35} &= -2(z_1+z_2) a_{35}
 \end{array}, \quad
 \begin{array}{rl}
 \delta a_{51} &= \left( \frac{z_1 + 3z_2}{2} \right) a_{51}\\
 \delta a_{52} &= \left( \frac{z_2 - z_1}{2} \right) a_{52}\\ 
 \delta a_{53} &= \left( \frac{z_1 - z_2}{2} \right) a_{53}\\
 \delta a_{54} &= \left( \frac{3z_1 + z_2}{2} \right) a_{54}\\
 \delta a_{55} &= 0
 \end{array}
 \end{align*}
 However, as indicated in Theorem \ref{T:Petrov-sym}, there are no type III structures with 7 symmetries.
 
% \begin{lemma}
% There do not exist any type III structures with 7 symmetries.
% \end{lemma}
% 
% \begin{proof}
% If such a structure exists, it must be homogeneous and the 2-dimensional structure group must act trivially.  This forces all $B,C,D, T_{ij}, a_{ij}$ coefficients to vanish except for $T_{14}, T_{41}, a_{55}$, all of which must be constant.  Differentiate the relations on $\nu_1$ and $\nu_3$ to obtain the inconsistent system:
% \begin{align*}
% 0 = \left(T_{41} - T_{14} + \frac{1}{4}\right) \varpi_1 \wedge \varpi_4, \quad
% 0 = \left(-T_{41}+\frac{1}{4} - a_{55}\right) \varpi_1 \wedge \varpi_3 + \left(-T_{14}- \frac{1}{4} - a_{55}\right) \varpi_2 \wedge \varpi_4
% \end{align*}
%  \end{proof}
 
 Now, $0 = d^2 \nu_2 \wedge \varpi_{45} = d^2 \nu_2 \wedge \varpi_{15}$ implies $B_1 = B_5 = 0$, and $0 = d(B_1) = d(B_5)$ is equivalent to:
 \[
 \beta_1 = 6 B_2 \alpha_3, \quad
 \beta_5 = 6 B_6 \alpha_3,
 \]
 and further Bianchi identities imply
 \[
 a_{42} = -2B_3, \quad a_{43} = 2 B_7, \quad a_{52} = -2 B_4, \quad a_{32} = -2 B_2, \quad a_{33} = 2 B_6, \quad a_{53} = 2 B_8, \quad B_2 B_6 = 0.
 \]
 There is a duality inducing $(B_j,B_{j+4}) \mapsto (-B_{j+4},B_j)$, where $j=1,...,4$, so WLOG, we may assume that one of $B_2,B_3,B_4$ is nonzero, or $B=0$.  Similar calculations show that for multiply transitive structures, we must have $B_2 = B_3 = 0$ (hence, $B_6 = B_7 = 0$ also).  Up to duality, we only have: III.6-1 ($B_4 \neq 0$ branch), and III.6-2 ($B=0$ branch).  Structure equations are given in Table \ref{T:III-streq}.

 \vspace{1in}
 \begin{footnotesize}
 \begin{longtabu}{|>{$}c<{$}|>{$}c<{$}|>{$}l<{$}|>{$}l<{$}|}  
 \caption{Multiply transitive type III structure equations}
 \label{T:III-streq}\\
\hline
 \mbox{Model} & \mbox{SD} &  \multicolumn{1}{|c|}{\mbox{Structure equations}} & \mbox{Embedding into Cartan bundle}\\ \hline
 %%%%%%%%%%%%%%%%%%%%%%%%%%%%%%%%%%%%%%%
 \mbox{III.6-1} & \xmark &
 \begin{array}{l@{\,}l}
 d\varpi_1&= \varpi_1 \wedge \varpi_4\\
 d\varpi_2&= (-\frac{5}{4} \varpi_1 + \varpi_4 + 2\zeta_1) \wedge \varpi_2 - \frac{1}{2} \varpi_1 \wedge \varpi_5\\
 d\varpi_3&= -\frac{3}{2} \varpi_1 \wedge \varpi_3 + \frac{3}{16} \varpi_1 \wedge \varpi_5 - \frac{3}{4} \varpi_2 \wedge \varpi_4 \\
 &\qquad - 3 \varpi_3 \wedge \varpi_4 - 2 \varpi_3 \wedge \zeta_1 + \frac{3}{4} \varpi_4 \wedge \varpi_5\\
 d\varpi_4&= -\frac{1}{2} \varpi_1 \wedge \varpi_4\\
 d\varpi_5&= \varpi_1 \wedge \varpi_3 + \varpi_2 \wedge \varpi_4 - \varpi_1 \wedge \varpi_5 \\
 &\qquad + 2 \varpi_4 \wedge \varpi_5 - 2\varpi_5 \wedge \zeta_1\\
 d\zeta_1 &= \frac{9}{8} \varpi_1 \wedge \varpi_4
 \end{array} & \begin{array}{l@{\,\,}l} 
 \zeta_2 &= - \varpi_1 + 2 \varpi_4 + \zeta_1 \\
 \nu_1 &= \zeta_1 - \frac{1}{2} \varpi_1 + \varpi_4\\
 \nu_3 &= -\frac{3}{4} \varpi_2 - \frac{1}{8} \varpi_5\\
 \tau_1 &= \frac{3}{16} \varpi_1 + \frac{5}{8} \varpi_4\\
 \tau_4 &= \frac{3}{8} \varpi_1\\
 \nu_2 &= \tau_2 = \tau_3 = \tau_5 = 0 \\
 \end{array}\\ \hline
 %%%%%%%%%%%%%%%%%%%%%%%%%%%%%%%%%%%%%%%
  \mbox{III.6-2} & \xmark &
 \begin{array}{l@{\,}l}
 d\varpi_1&= (-2\zeta_2 + 2\varpi_3) \wedge \varpi_1 \\
 d\varpi_2&= (-4\zeta_2 + 6\varpi_3) \wedge \varpi_2 + \frac{1}{4} \varpi_1 \wedge \varpi_5\\
 d\varpi_3&= \frac{3}{8} \varpi_1 \wedge \varpi_4 \\
 d\varpi_4&= (2\zeta_2 - 2\varpi_3) \wedge \varpi_4, \\
 d\varpi_5&= (-2\zeta_2 + 4\varpi_3) \wedge \varpi_5 + \varpi_1 \wedge \varpi_3 + \varpi_2 \wedge \varpi_4 \\
 d\zeta_2 &= \frac{5}{8} \varpi_1 \wedge \varpi_4
 \end{array} & \begin{array}{l@{\,\,}l} 
 \zeta_1 &= 4 \varpi_3 - 3 \zeta_2\\
 \nu_1 &= 2 \varpi_3 - \zeta_2\\
 \nu_2 &= 0, \nu_3 = \frac{3}{8} \varpi_1 - \frac{1}{8} \varpi_5\\
 \tau_1 &= -\frac{1}{8} \varpi_4\\
 \tau_2 &= \tau_3 = \tau_5 = 0\\
 \tau_4 &= -\frac{3}{8} \varpi_1
 \end{array}\\ \hline
 \end{longtabu}
 \end{footnotesize}

\section{Integration of structure equations}

In this section, we outline the transition from structure equations found in the previous section to the corresponding systems of 2nd order PDEs. This is done in three steps:
\begin{enumerate}
\item Normalize the algebraic structure of the Lie algebra data defined by the structure equations. This step consists of identifying the type of the Lie algebra $\fg$, the isotropy subalgebra $\fk$ and the subspaces $E,V\subset\fg /\fk$ corresponding to the two Legendrian subbundles.  We note that both $E+\fk$ and $V+\fk$ are in fact subalgebras of $\fg$, as we deal only with integrable structures. We also try to find a good basis in $\fg$, adjusting it to the Levi decomposition and the nilradical.
\item Realize $\fg$ as a transitive Lie algebra of vector fields on $\bbC^3$ in such a way that its isotropy subalgebra at a certain point is equal exactly to $V+\fk$. This guarantees that the first prolongation of $\fg$ is transitive on $J^1(\bbC^2,\bbC)$ and has isotropy $\fk$ at a certain point. 
\item Finally, we compute all complete systems of 2nd order PDEs admitting $\fg$ as its symmetry and identify those which correspond to $E+\fk$. In fact, in all cases but one (D.6-3${}_\infty$, see Example~\ref{ex:D63} below) there is exactly one such system, and this identification is obtained automatically.
\end{enumerate}

\begin{example}
Consider the structure equations for the model D.7 as given in Table~\ref{T:D-streq}. Simple analysis shows the corresponding Lie algebra $\fg$ has radical of dimension 1 if $a\ne\pm \frac{3}{4}$, and of dimension 4 if $a=\pm \frac{3}{4}$.  (Note that $a$ and $-a$ yield equivalent models.)  Consider first the case $a\ne \frac{3}{4}$. It is clear that $\fg$ has a 6-dimensional Levi subalgebra, which is isomorphic to $\fsl_2(\bbC)\times\fsl_2(\bbC)$ (the only complex semisimple Lie algebra in this dimension). As any action of this Levi subalgebra on the 1-dimensional radical is trivial, $\fg$ is isomorphic to $\fsl_2(\bbC)\times\fsl_2(\bbC)\times\bbC$. The corresponding basis change from the Cartan reduced basis to the adapted Lie algebra basis in given in Table~\ref{Cartan2adapted}. 

Next, analyzing the Cartan basis, we see that the isotropy $\fk$ is 2-dimensional and abelian. Moreover, its projection to each $\fsl_2(\bbC)$-factor is one-dimensional and diagonalizable, while the intersection with each $\fsl_2(\bbC)$-factor is trivial. This implies that $\fk$ is conjugate to the following subalgebra in $\fg$:
\[
\fk \sim \langle H_1 - Z, H_2 - \lambda Z \rangle, \quad \lambda \in \bbC \backslash \{ 0 \},
\] 
where $H_1, H_2$ are parts of the standard $\fsl_2(\bbC)$-basis $\{ X_i, H_i, Y_i \}$ in each copy of $\fsl_2(\bbC)$, and $Z$ spans the center $\fz=\bbC$.  Also, $\lambda=\frac{3+4a}{3-4a}$, and the redundancy $a \mapsto -a$ induces the redundancy $\lambda\mapsto \frac{1}{\lambda}$. 

Further, it is easy to check that the projections of both $E+\fk$ and $V+\fk$ to each $\fsl_2(\bbC)$-factor is two-dimensional. Thus, we can assume that:
\begin{align*}
V+\fk &= \langle X_1, X_2, H_1-Z, H_2-\lambda Z\rangle,\\
E+\fk &= \langle Y_1, Y_2, H_1-Z, H_2-\lambda Z\rangle.
\end{align*}
Let us now realize $\fg$ as a Lie algebra of vector fields on $\bbC^3=J^0(\bbC^2,\bbC)$ with the isotropy subalgebra equal to $V+\fk$. Note that $\fh = V+\fk + \fz$ is a subalgebra of codimension 2 in $\fg$. However, it is not effective, and the maximal ideal of $\fg$ contained in $\fh$ is exactly $\fz$. So, $\fg/\fz$ can be realized as a Lie algebra of vector fields on $\bbC^2$ with the isotropy $\fh/\fz$. But $\fg/\fz$ is isomorphic to $\fsl_2(\bbC)\times\fsl_2(\bbC)$ with $\fh/\fz$ identified with the direct product of two subalgebras of upper-triangular matrices. It is easy to see that it integrates to the global action of $\PSL_2(\bbC)\times \PSL_2(\bbC)$ on $\bbP^1 \times \bbP^1$. Locally this leads to the following realization of $\fg/\fz$:
\[
\langle \partial_x,\, 2x\partial_x, \, x^2\partial_x,\, \partial_y, \, 2y\partial_y, \, y^2\partial_y\rangle.
\]
We can always assume that the realization of $\fg$ is adapted to it. In other words, it can be obtained from the above one by adding terms of the form $f(x,y,u)\partial_u$ to the above vector fields and realizing the center $Z$ as a vector field of the form $g(x,y,u)\partial_u$. Simple computation shows that we can always adapt the coordinates $(x,y,u)$ such that $Z$ becomes equal to $\partial_u$, and we get the following realization of $\fg$:
\[
\langle \partial_x, \, 2x\partial_x+\partial_u, \, x^2\partial_x+x\partial_u,\, \partial_y, \, 2y\partial_y+\tfrac{1}{\lambda}\partial_u, \, y^2\partial_y+\tfrac{1}{\lambda}y\partial_u,\, \partial_u\rangle.
\]

Prolonging this Lie algebra of vector fields to $J^1(\bbC^2,\bbC)$ and checking which complete systems of 2nd order PDEs are invariant with respect to it, we immediately get that the only such system has the form:
\[
u_{11} = p^2,\quad u_{12} = 0,\quad u_{22} = \lambda q^2. 
\]

Setting now the parameter $\lambda$ to $0$ and computing the symmetry algebra of the above system of PDEs, we obtain exactly the Lie algebra $\fg$, its subalgebra $\fk$ and subspaces $E,V\subset \fg/\fk$, that match the exceptional case $a=\pm \frac{3}{4}$ of the Cartan structure equations in case of D.7.
\end{example}

\begin{example}\label{ex:D63}
Consider now the case D.6-3. We note that in this case the Lie algebra $\fg$ defined by the structure equations is semisimple if $a\ne \pm 3$ and has a 3-dimensional abelian ideal otherwise. First, consider the generic case of $a\ne \pm 3$. Then $\fg$ is isomorphic to $\fsl_2(\bbC)\times\fsl_2(\bbC)$.  As above, denote by $\{ X_i, H_i, Y_i\}$, $i=1,2$, the standard bases of these two copies of $\fsl_2(\bbC)$. Direct inspection of the Cartan structure equations shows that $\fk = \langle H_1 - H_2 \rangle$ and both subalgebras $E+\fk$ and $V+\fk$ are three-dimensional semisimple. But any simple subalgebra of $\fg$ containing $\fk$ has the form:
\begin{equation}\label{subalgs}
\langle X_1 + \mu Y_2, X_2 + \mu Y_1, H_1 - H_2\rangle,\quad \mu\ne 0,
\end{equation}
and any two such subalgebras are conjugate to each other by means of inner automorphisms preserving $\fk$. Hence, we can assume that $V+\fk$ corresponds to $\mu = 1$, which is exactly the diagonal of the direct product of $\fsl_2(\bbC)\times \fsl_2(\bbC)$.  Under the classical isomorphism $\fso(4,\bbC)\simeq \fsl_2(\bbC)\times \fsl_2(\bbC)$ this subalgebra corresponds to the standard embedding of $\fso(3,\bbC)\subset \fso(4,\bbC)$. So, we can realize the Lie algebra $\fg$ as a Lie algebra of vector fields corresponding to the action of $SO(4,\bbC)$ on the three-dimensional complex sphere. In an appropriate coordinate system we get the following vector fields:
\begin{align*}
  X_1 &= \partial_x,  & H_1 =& -2x\partial_x-2u\partial_u,  & Y_1 =& -x^2\partial_x-u\partial_y-2xu\partial_u\\
  X_2 &= \partial_y,   & H_2 = & -2y\partial_y-2u\partial_u,  & Y_2 =& -y^2\partial_y-u\partial_x-2yu\partial_u.
\end{align*}

Again, prolonging this Lie algebra of vector fields to $J^1(\bbC^2,\bbC)$ and computing all invariant systems of 2nd order PDEs, we obtain the following family of systems:
\[
  u_{11} = \lambda p^2\frac{\sqrt{u-pq}}{u^{3/2}}, \quad
  u_{12} = 1+\lambda (pq-2u)\frac{\sqrt{u-pq}}{u^{3/2}}, \quad
  u_{22} = \lambda q^2\frac{\sqrt{u-pq}}{u^{3/2}}. 
\]
Each such system corresponds to the subalgebra~\eqref{subalgs} with $\mu = \frac{2\lambda-1}{2\lambda+1}$.

In the limiting case of $a=\pm3$ in the structure equations we get $\fg\simeq \fso(3,\bbC)\rightthreetimes \bbC^3$ and $\fk+V=\bbC^3$. This pair corresponds to the group of complex Euclidean transformations of $\bbC^3$, which preserves the following family of complete systems of 2nd order PDEs:
 \[
 u_{11}=\lambda p^2\sqrt{1-2pq}, \quad
 u_{12}=\lambda (pq-1)\sqrt{1-2pq}, \quad
 u_{22}= \lambda q^2\sqrt{1-2pq}.
 \]
If $\lambda=0$, this system is flat and has 15-dimensional symmetry algebra. If $\lambda\ne 0$, then we can normalize it to $\lambda=1$ by means of the transformation $(x,y,u)\mapsto (\lambda x, \lambda y, \lambda u)$. To distinguish this special case from the generic one, we denote it by D.6-3${}_\infty$. 
\end{example}
\begin{example}
Consider the case N.6-2, which involves two parameters. The Lie algebra $\fg$ is solvable in this case and has a 4-dimensional abelian nilradical $\fn$. Two basis elements complementary to $\fn$ act on $\fn$ by the following two commuting matrices:
\[
\begin{pmatrix}
2b & 0 & 0 & 1 \\
0 & b & 1 & 0 \\
0 & 1 & 2b & 0 \\
1 & 0 & 0 & b 
\end{pmatrix},\quad
\begin{pmatrix}
a & 0 & -1 & 0 \\
0 & 2a & 0 & -1 \\
-1 & 0 & 2a & 0 \\
0 & -1 & 0 & a 
\end{pmatrix}. 
\]
If parameters $a,b$ of the structure equations satisfy $a^2+4\ne0, b^2+4\ne 0$, then both matrices simultaneously diagonalize in a certain basis $\{N_1, N_2, N_3, N_4\}$ of $\fn$ to become:
\begin{align*}
& \tfrac{1}{2}\diag\left(3b-\sqrt{b^2+4}, 3b+\sqrt{b^2+4}, 3b+\sqrt{b^2+4}, 3b-\sqrt{b^2+4}\right),\\ 
& \tfrac{1}{2}\diag\left(3a-\sqrt{a^2+4}, 3a-\sqrt{a^2+4},  3a+\sqrt{a^2+4}, 3a+\sqrt{a^2+4}\right).
\end{align*}
After rescaling, we can bring them to the form:
\begin{align*}
& \diag(\mu-1,\mu,\mu,\mu-1), &\mu &= \frac{1}{2} + \frac{3b}{2\sqrt{b^4+4}},\\
& \diag(\kappa+1,\kappa+1,\kappa+2,\kappa+2), &\kappa &= \frac{3}{2} + \frac{3a}{2\sqrt{a^2+4}}.
\end{align*}
Denote by $S_1,S_2$ the corresponding elements in $\fg$, which span the complementary subspace to $\fn$. In general, this subspace is not a subalgebra, and $[S_1,S_2]\in\fn$.  But if any of these two matrices is invertible (meaning $\mu\ne 0,1$ or $\kappa\ne -1,-2$) then we can always adjust $S_1,S_2$ by adding elements from $\fn$ such that we get $[S_1,S_2]=0$. We note that there are elements $u_1,u_2\in\fn$ such that $S_1+u_1\in V+\fk$, $S_2+u_2\in E+\fk$.

It is easy to check that the intersection of $V+\fk$ with $\fn$ is two-dimensional and can be made equal to $\langle N_1-N_4, N_2-N_3\rangle$ after suitable rescaling to basis vectors $\{N_i\}$. Hence, in any realization of $\fg$ as a transitive Lie algebra of vector fields on $\bbC^3$ having $V+\fk$ as a stabilizer, $\fn$ will be a 4-dimensional abelian Lie algebra with 2-dimensional orbits. In particular, we can always choose a local coordinate system $(x,y,u)$ in such a way that $N_3=\partial_u$, $N_4=\partial_y$, and two other basis vectors $N_1,N_2$ will be of the form $f(x)\partial_y+g(x)\partial_u$.  As $S_1,S_2$ act by scalings on any of $N_i$, $i=1,\dots,4$, it is natural to assume that they are represented as linear combinations of vector fields $x\partial_x,y\partial_y,u\partial_u$. Using this ansatz, we immediately get the following representation of $\fg$:
\begin{align*}
S_1 &= -(\mu-1)y\partial_y-\mu u\partial_u,\\ 
S_2 &=-x\partial_x-(\kappa+2)y\partial_y-(\kappa+2)u\partial_u,\\
 N_1 &=x\partial_y, \quad N_2=x\partial_u, \quad N_3=\partial_u, \quad N_4=\partial_y.
\end{align*}
Prolonging this Lie algebra of vector fields to $J^1(\bbC^2,\bbC)$ and computing all invariant complete systems of 2nd order PDEs, we arrive at the following system:
\[
u_{11} = q^\mu x^\kappa, \quad u_{12}=0, \quad u_{22} = 0.
\]
The special values of parameters we omitted on the way can be treated in a similar way and lead to the following systems of PDEs:
\begin{itemize}
\item $a^2+4 = 0, b^2+4\ne 0$ (or equivalently, $a^2+4\ne 0, b^2+4=0$):
\[
u_{11} = e^q x^\kappa, \quad u_{12}=0, \quad u_{22} = 0.
\]
\item $a^2+4=b^2+4=0$:
\[
u_{11} = e^q e^x, \quad u_{12}=0, \quad u_{22} = 0.
\]
\item $\mu=0,1$ (or equivalently, $\kappa=-1,-2$):
 \[
u_{11} = \ln(q) x^\kappa, \quad u_{12}=0, \quad u_{22} = 0.
\]
\end{itemize}
More details on restrictions on parameters and realizations of $\fg$ in terms of vector fields for these special values of parameters are given in Table~\ref{Ncases}.
\end{example}
 
 \bibliographystyle{amsplain}

  \appendix
  \section{Classification Tables}
  
  % % Type N models was double checked,
  % % TODO: double check change of basises to abstract models.
  
   \begin{landscape} 
 \setlength{\tabcolsep}{2pt}
 \fontsize{6}{8} \selectfont
  \begin{longtabu}[HT]{ l  *5{>{$} l<{$}} }
  \caption{Classification of type N cases}\label{Ncases}\\
  \toprule
  & \mbox{\bf Model} & \mbox{\bf Parameters}&  \mbox{\bf Symmetries} & \mbox{\bf Lie algebra structure} & \mbox{\bf Abstract ILC structure}\\ 
  \midrule
  \endfirsthead
  \caption[]{Classification of type N cases (continued)}\\
  \toprule
  & \mbox{\bf Model} & \mbox{\bf Parameters}&  \mbox{\bf Symmetries} & \mbox{\bf Lie algebra structure} & \mbox{\bf Abstract ILC structure}\\ 
  \midrule

  \endhead
     N.8 & \begin{array}{l} 
    u_{11} = q^2 \\
    u_{12} = 0  \\ 
    u_{22} = 0 \end{array}
    &  & \begin{array}{r@{\,}l} 
     S_1 = & -y\p_y-2u\p_u-2p\p_p-q\p_q
    \\
     S_2 = & -x\p_x -2y\p_y -2u\p_u -p\p_p
    \\
   N_1 = & \frac{x^2}{2}\p_y-\frac{y}{2}\p_u-qx\p_p-\frac12\p_q
   \\
   N_2 =& \p_x
   \\
     N_3 = & -x \p_y+q\p_p
     \\
     N_4 = & -\frac{x}{2}\p_u-\frac12\p_p
     \\
   N_5 = & -\frac12\p_u
   \\
   N_6 = & -\p_y
   \\
    & (\mbox{General pt: } x=y=u=p=q=0)
    \end{array} &  \begin{array}{c|ccccccccc} 
    & S_1 & S_2 & N_1 & N_2 & N_3 & N_4  & N_5 & N_6\\ \hline
    S_1 &\cdot & \cdot & N_1 &\cdot & N_3 &  2 N_4 & 2N_5 & N_6 \\ 
    S_2 & & \cdot  & \cdot & N_2 & N_3 & N_4  & 2 N_5 & 2 N_6 \\
    N_1 & & & \cdot  & N_3 & N_4 & \cdot &\cdot  & N_5 \\ 
    N_2 &&&& \cdot & N_6 & N_5 & \cdot  &\cdot\\ 
    N_3 &&&&&\cdot &\cdot & \cdot  &\cdot\\ 
    N_4 &&&&&&\cdot &\cdot &\cdot \\
    N_5 &&&&&&&\cdot &\cdot \\
    N_6 &&&&&&& &\cdot\end{array} &
    \begin{array}{rl} 
    E/\fk : &N_2, N_6\\
    V/\fk : &N_1,N_4  \\ 
     \fk : & N_3, S_1, S_2
    \end{array}
    \\
       \midrule
  
 % \begin{tabular}{l}  
 %  N.7-1 \\
 %  (dual)  
 %  \end{tabular} 
 %  & \begin{array}{l} 
 % u_{11} = q^\mu \\
 % u_{12} = 0  \\ 
 % u_{22} = 0 \end{array}
 % & \begin{array}{l} 
 % \mu\neq -1,2,0,1,\infty\\
 % (a^2 \neq \infty,\frac12,-2)
 % \end{array} & \begin{array}{r@{\,}l} 
 %  S_1 = & -(\mu-1)y \p_y-\mu u\p_u -\mu p\p_p-q\p_q 
 % \\
 %  S_2 = & -x\p_x-2 y\p_y -2u\p_u -p\p_p
 % \\
 % N_1 = & \p_x
 % \\
 % N_2 = & x\p_y-q\p_p
 %\\
 %N_3 = & x\p_u + \p_p
 %\\
 %N_4 = & \p_u
 %\\
 %N_5 = & \p_y
 %\\
 % & (\mbox{General pt: } y=u=p=q=0, x=1)
 % \end{array} &  \begin{array}{c|cccccccc} 
 % & S_1 & S_2 & N_1 & N_2 & N_3 & N_4 & N_5\\ \hline
 % S_1 &\cdot & \cdot & \cdot & (\mu-1)N_2 &  \mu N_3 &  \mu N_4 & (\mu-1) N_5 \\ 
 % S_2 & & \cdot  & N_1 &  N_2 &  N_3 &  2 N_4  &  2 N_5  \\
 % N_1 & & & \cdot  & N_5 & N_4 & \cdot & \cdot  \\ 
 % N_2 &&&&\cdot &\cdot& \cdot & \cdot \\ 
 % N_3 &&&&&\cdot &\cdot & \cdot \\ 
 % N_4 &&&&&&\cdot &\cdot \\
 % N_5 &&&&&&&\cdot \end{array} &
 % \begin{array}{rl} 
 % E/\fk : &N_1+N_3,\\
 % & N_4+N_5
 % \\
 % V/\fk : & N_3 , 
 % \\
 %  &  -S_1
 %  \\ 
 %  \fk : & S_2\\
 %  &  N_2+N_3 \end{array}
 %    \\
 %   \cmidrule{2-6} 
    N.7-1 & \begin{array}{l} 
  u_{11} = q^2 x^\kappa \\
  u_{12} = 0  \\ 
  u_{22} = 0 \end{array}
  & \begin{array}{l} 
  \kappa\neq -1,-2,0,-3,\infty\\
  (a^2 \neq \frac12,-2)
  \end{array} & \begin{array}{r@{\,}l} 
   S_1 = & -y \p_y-2 u\p_u -2 p\p_p-q\p_q 
  \\
   S_2 = & -x\p_x-(\kappa+2)y\p_y -(\kappa+2)u\p_u
   \\ 
   &-(\kappa+1)p\p_p
  \\
    N_1 = & \frac{x^{\kappa+2}}{\kappa+2}\p_y -\frac{\kappa+1}2 y\p_u -x^{\kappa+1}q\p_p \\
    & -\frac{k+1}2 \p_q
    \\
 N_2 = & x\p_y-q\p_p
 \\
 N_3 = &\frac{\kappa+1}2\left( x\p_u + \p_p \right)
 \\
 N_4 = &\frac{\kappa+1}2 \p_u
 \\
 N_5 = & \p_y
 \\
  & (\mbox{General pt: } y=u=p=q=0, x=1)
  \end{array} &  \begin{array}{c|cccccccc} 
  & S_1 & S_2 & N_1 & N_2 & N_3 & N_4 & N_5\\ \hline
  S_1 &\cdot & \cdot & N_1 & N_2 & 2 N_3 &  2 N_4 & N_5 \\ 
  S_2 & & \cdot  & \cdot & (\kappa+1) N_2 & (\kappa+1) N_3 &  (\kappa+2) N_4  &  (\kappa+2) N_5  \\
  N_1 & & & \cdot  & N_3 & \cdot & \cdot & N_4 \\ 
  N_2 &&&&\cdot &\cdot& \cdot & \cdot \\ 
  N_3 &&&&&\cdot &\cdot & \cdot \\ 
  N_4 &&&&&&\cdot &\cdot \\
  N_5 &&&&&&&\cdot \end{array} &
  \begin{array}{rl} 
  E/\fk : &S_2, N_2
  \\
  V/\fk : &N_3-N_4, \\
   & -N_1+\frac1{\kappa+2}N_2
   \\ 
   \fk : & S_1,  N_2-N_5 
  \end{array}
  \\
     \cmidrule{2-6}
  & \begin{array}{l} 
  u_{11} = q^2 x^{-1} \\
  u_{12} = 0  \\ 
  u_{22} = 0 \end{array}
  & \begin{array}{l} 
  \kappa= -1\\
  (a = \frac1{\sqrt2})
  \end{array} & \begin{array}{r@{\,}l} 
   S_1 = & -y \p_y-2 u\p_u -2 p\p_p-q\p_q 
  \\
   S_2 = & -x\p_x-y\p_y -u\p_u
  \\
    N_1 = & 2x(\ln(x)-1)\p_y -y\p_u \\
    & -2\ln(x)q\p_p - \p_q
    \\
 N_2 = & x\p_y-q\p_p
 \\
 N_3 = & x\p_u + \p_p
 \\
 N_4 = & \p_u
 \\
 N_5 = & \p_y
 \\
  & (\mbox{General pt: } y=u=p=q=0, x=1)
  \end{array} &  \begin{array}{c|cccccccc} 
  & S_1 & S_2 & N_1 & N_2 & N_3 & N_4 & N_5\\ \hline
  S_1 &\cdot & \cdot & N_1 & N_2 & 2 N_3 &  2 N_4 & N_5 \\ 
  S_2 & & \cdot  & - 2 N_2  & \cdot &\cdot &  N_4  & N_5  \\
  N_1 & & & \cdot  & N_3 & \cdot & \cdot & N_4 \\ 
  N_2 &&&&\cdot &\cdot& \cdot & \cdot \\ 
  N_3 &&&&&\cdot &\cdot & \cdot \\ 
  N_4 &&&&&&\cdot &\cdot \\
  N_5 &&&&&&&\cdot \end{array} &
  \begin{array}{rl} 
  E/\fk : &S_2,N_2\\
  V/\fk : & N_3 - N_4 , \\
   & N_1+2N_2
   \\ 
   \fk : & S_1, N_2-N_5 \end{array}
   
    \\
       \cmidrule{2-6}
    & \begin{array}{l}
    u_{11} = q^2 e^x \\
    u_{12} = 0  \\ 
    u_{22} = 0 \end{array}
    & \begin{array}{l} 
    \kappa= \infty \\
    (a =\pm2i)
    \end{array} & \begin{array}{r@{\,}l} 
     S_1 = & -y \p_y-2 u\p_u -2 p\p_p-q\p_q 
    \\
     S_2 = & -\p_x-y\p_y -u\p_u - p\p_p
    \\
      N_1 = & 2 e^x\p_y -y\p_u \\
      & -2e^xq\p_p - \p_q
      \\
   N_2 = & x\p_y-q\p_p
   \\
   N_3 = & x\p_u + \p_p
   \\
   N_4 = & \p_u
   \\
   N_5 = & \p_y
   \\
    & (\mbox{General pt: } x=y=u=p=q=0,)
    \end{array} &  \begin{array}{c|cccccccc} 
    & S_1 & S_2 & N_1 & N_2 & N_3 & N_4 & N_5\\ \hline
    S_1 &\cdot & \cdot & N_1 & N_2 & 2 N_3 &  2 N_4 & N_5 \\ 
    S_2 & & \cdot  & \cdot  & N_2-N_5 &N_3-N_4 &  N_4  & N_5  \\
    N_1 & & & \cdot  & N_3 & \cdot & \cdot & N_4 \\ 
    N_2 &&&&\cdot &\cdot& \cdot & \cdot \\ 
    N_3 &&&&&\cdot &\cdot & \cdot \\ 
    N_4 &&&&&&\cdot &\cdot \\
    N_5 &&&&&&&\cdot \end{array} &
    \begin{array}{rl} 
    E/\fk : & S_2,N_5\\
    V/\fk : & N_3 , -N_1+2N_5
     \\ 
     \fk : & S_1, N_2 \end{array}
   
     \\
    \midrule
 
     N.7-2 & \begin{array}{l} 
   u_{11} = q^{-1} \\
   u_{12} = 1  \\ 
   u_{22} = 0 \end{array}
   &  & \begin{array}{r@{\,}l} 
    X = & -\p_x + \p_u 
   \\
    H = & -2x\p_x + \p_y -2u\p_u -2q\p_q
   \\
     Y = &x^2\p_x + u\p_y  +(2ux+x^2)\p_u \\
     & +(2x+2u-qp)\p_p+q(2x-q) \p_q
     \\
  N_1 = & x\p_y+x^2\p_u-(q-2x)\p_p
  \\
  N_2 = &-\p_y -2x\p_u -2 \p_p
  \\
  N_3 = & 2 \p_u
  \\
  N_4 = & \p_y
  \\
   & (\mbox{General pt: } y=u=p=q=0, x=1)
   \end{array} &  \begin{array}{c|cccccccc} 
   & X & H & Y & N_1 & N_2 & N_3 & N_4\\ \hline
   X &\cdot & -2 X & H & N_2 &  N_3 & \cdot & \cdot \\ 
   H & & \cdot  & -2 Y & -2 N_1 & \cdot &  2 N_3  & \cdot  \\
   Y & & & \cdot  & \cdot & 2N_1 & 2 N_2 &\cdot \\ 
   N_1 &&&&\cdot &\cdot& \cdot & \cdot \\ 
   N_2 &&&&&\cdot &\cdot & \cdot \\ 
   N_3 &&&&&&\cdot &\cdot \\
   N_4 &&&&&&&\cdot \end{array} &
   \begin{array}{rl} 
   E/\fk : &X+Y+N_1-\frac12 N_3, \\
   & -N_1+\frac12 N_3+N_4 
   \\
   V/\fk : & Y, N_1
    \\ 
    \fk : & H-2 Y - N_4 ,\\
    & -2N_1+N_2+N_4
   \end{array}
   \\
      \midrule
 
     N.6-1 & \begin{array}{l} 
     u_{11} = q^\mu \\
     u_{12} = 1  \\ 
     u_{22} = 0 \end{array}
     & \begin{array}{l} 
     \mu\neq-1,2,0,1,\infty\\
     (a^2 \neq 0,\frac12,2,-1)
     \end{array} & \begin{array}{r@{\,}l} 
     S = & -x\p_x -(\mu+1)y \p_y
     \\
     &-(\mu+2) u\p_u -(\mu+1) p\p_p-q\p_q 
     \\
      N_1 = & \p_x
     \\
    N_2 = & x\p_y + x^2\p_u + (2x-q)\p_p
    \\
    N_3 = &  \p_y + 2x\p_u + 2\p_p
    \\
    N_4 = & 2\p_u
    \\
    N_5 = & \p_y
    \\
     & (\mbox{General pt: } x=y=u=p=0, q=1)
     \end{array} &  \begin{array}{c|cccccccc} 
     & S & N_1 & N_2 & N_3 & N_4 & N_5\\ \hline
     S & \cdot & N_1 & \mu N_2 & (\mu+1) N_3 &  (\mu+2) N_4 & (\mu+1)N_5 \\ 
     N_1 & & \cdot & N_3 & N_4 & \cdot & \cdot \\
     N_2 & & & \cdot & \cdot & \cdot & \cdot \\ 
     N_3 &&&&\cdot& \cdot & \cdot \\ 
     N_4 &&&&&\cdot & \cdot \\ 
     N_5 &&&&&&\cdot \end{array} &
     \begin{array}{rl} 
     E/\fk : &-S+N_1-N_2,\\
     & N_2+N_3+\frac12 N_4 
     \\
     V/\fk : & S , N_2
      \\
      \fk : & 2N_2+N_3-N_5 \end{array}
        \\
       \cmidrule{2-6} 
       & \begin{array}{l} 
     u_{11} = \ln(q) \\
     u_{12} = 1  \\ 
     u_{22} = 0 \end{array}
     & \begin{array}{l} 
     \mu=0\, (a^2 =\frac12)
     \end{array} & \begin{array}{r@{\,}l} 
     S = & -x\p_x + (\frac{x}{2} -y) \p_y
     \\
     &-2u\p_u - (\frac{q}2 + p)\p_p-q\p_q 
     \\
      N_1 = & \p_x
     \\
    N_2 = & x\p_y + x^2\p_u + (2x-q)\p_p
    \\
    N_3 = &  \p_y + 2x\p_u + 2\p_p
    \\
    N_4 = & 2\p_u
    \\
    N_5 = & -\frac12\p_y
    \\
     & (\mbox{General pt: } x=y=u=p=0, q=1)
     \end{array} &  \begin{array}{c|cccccccc} 
     & S & N_1 & N_2 & N_3 & N_4 & N_5\\ \hline
     S & \cdot & N_1+ N_5 & \cdot  & N_3 &  2 N_4 & N_5 \\ 
     N_1 & & \cdot & N_3 & N_4 & \cdot & \cdot \\
     N_2 & & & \cdot & \cdot & \cdot & \cdot \\ 
     N_3 &&&&\cdot& \cdot & \cdot \\ 
     N_4 &&&&&\cdot & \cdot \\ 
     N_5 &&&&&&\cdot \end{array} &
     \begin{array}{rl} 
     E/\fk : &-S+N_1+\frac12 N_2,\\
     & N_2+N_3+\frac12 N_4 
     \\
     V/\fk : & N_2,  -S +\frac12 N_2
      \\
      \fk : & 2N_2+N_3+2 N_5 \end{array}
        \\
       \cmidrule{2-6} 
        & \begin{array}{l} 
     u_{11} = q\ln(q) \\
     u_{12} = 1  \\ 
     u_{22} = 0 \end{array}
     & \begin{array}{l} 
     \mu=1\, (a^2 = 2)
     \end{array} & \begin{array}{r@{\,}l} 
     S = & -x\p_x +(\frac{x^2}2 - 2y) \p_y
      -(3u-\frac{x^3}3)\p_u 
     \\
     &- (2p+qx-x^2)\p_p-q\p_q 
     \\
      N_1 = & \p_x
     \\
    N_2 = & x\p_y + x^2\p_u + (2x-q)\p_p
    \\
    N_3 = &  \p_y + 2x\p_u + 2\p_p
    \\
    N_4 = & 2\p_u
    \\
    N_5 = & \p_y
    \\
     & (\mbox{General pt: } x=y=u=p=0, q=1)
     \end{array} &  \begin{array}{c|cccccccc} 
     & S & N_1 & N_2 & N_3 & N_4 & N_5\\ \hline
     S & \cdot & N_1 - N_2 & N_2  & 2N_3 &  3 N_4 & 2 N_5 \\ 
     N_1 & & \cdot & N_3 & N_4 & \cdot & \cdot \\
     N_2 & & & \cdot & \cdot & \cdot & \cdot \\ 
     N_3 &&&&\cdot& \cdot & \cdot \\ 
     N_4 &&&&&\cdot & \cdot \\ 
     N_5 &&&&&&\cdot \end{array} &
     \begin{array}{rl} 
     E/\fk : &-S+N_1,\\
     & N_2+N_3+\frac12 N_4 
     \\
     V/\fk : & S, N_2
      \\
      \fk : & 2N_2+N_3-N_5 \end{array}
        \\
       \midrule
    
      N.6-2 & \begin{array}{l} 
      u_{11} = q^\mu x^\kappa \\
      u_{12} = 0 \\ 
      u_{22} = 0 \end{array}
      & \begin{array}{l} 
      \mu\neq-1,2,0,1,\infty\\
      (b\neq \pm\frac1{\sqrt{2}},\pm 2i), \\
      \kappa\neq 0,-3,\infty\\
      (a\neq \pm 2i) \\
      \end{array} & \begin{array}{r@{\,}l} 
      S_1 = & -(\mu-1)y \p_y-\mu u\p_u -\mu p\p_p-q\p_q 
      \\
       S_2 = & -x\p_x-(\kappa+2)y\p_y -(\kappa+2)u\p_u
       \\ 
       &-(\kappa+1)p\p_p
      \\
     N_1 = & x\p_y-q\p_p
     \\
     N_2 = & x\p_u + \p_p
     \\
     N_3 = & \p_u
     \\
     N_4 = & \p_y
     \\
      & (\mbox{General pt: } y=u=p=0, q=x=1)
      \end{array} &  \begin{array}{c|cccccccc} 
      & S_1 & S_2 & N_1 & N_2 & N_3 & N_4\\ \hline
      S_1 & \cdot & \cdot & (\mu-1)N_1 & \mu N_2 &  \mu N_3 & (\mu-1)N_4 \\ 
      S_2 & & \cdot & (\kappa+1)N_1 & (\kappa+1)N_2 &  (\kappa+2)N_3 & (\kappa+2)N_4 \\
      N_1 & & & \cdot & \cdot & \cdot & \cdot \\ 
      N_2 &&&&\cdot& \cdot & \cdot \\ 
      N_3 &&&&&\cdot & \cdot \\ 
      N_4 &&&&&&\cdot \end{array} &
      \begin{array}{rl} 
      E/\fk : &-S_2-N_1+N_4,\\
      & N_1+N_2 
      \\
      V/\fk : & -N_1+N_4, S_1 
       \\
       \fk : & N_1+N_2-N_3-N_4 \end{array}
         \\
        \cmidrule{2-6}
       & \begin{array}{l} 
      u_{11} = e^q x^\kappa \\
      u_{12} = 0 \\  
      u_{22} = 0 \end{array}
      & \begin{array}{l} 
      \mu= \infty\\
      (b= \pm 2i) \\
       \kappa\neq 0,-3,\infty\\
       (a\neq \pm 2i) \\
      \end{array} & \begin{array}{r@{\,}l} 
      S_1 = & -y \p_y-(u+y)\p_u-p\p_p-\p_q 
      \\
       S_2 = & -x\p_x-(\kappa+2)y\p_y -(\kappa+2)u\p_u
       \\ &-(\kappa+1)p\p_p
      \\
     N_1 = & x\p_y-q\p_p
     \\
     N_2 = & x\p_u + \p_p
     \\
     N_3 = & \p_u
     \\
     N_4 = & \p_y
     \\
      & (\mbox{General pt: } y=u=p=x=0, q=1)
      \end{array} &  \begin{array}{c|cccccccc} 
      & S_1 & S_2 & N_1 & N_2 & N_3 & N_4\\ \hline
      S_1 & \cdot & \cdot & N_1+ N_2 & N_2 & N_3  & N_3+N_4 \\ 
      S_2 & & \cdot & (\kappa+1)N_1 & (\kappa+1)N_2 &  (\kappa+2)N_3 & (\kappa+2)N_4 \\
      N_1 & & & \cdot & \cdot & \cdot & \cdot \\ 
      N_2 &&&&\cdot& \cdot & \cdot \\ 
      N_3 &&&&&\cdot & \cdot \\ 
      N_4 &&&&&&\cdot \end{array} &
      \begin{array}{rl} 
      E/\fk : &S_2-N_2+N_3,\\
      & N_1
      \\
      V/\fk : & N_2-N_3, S_1 
      \\
       \fk : & N_1-N_4 \end{array}
           \\
        \cmidrule{2-6}
       & \begin{array}{l} 
      u_{11} = e^q e^x \\
      u_{12} = 0 \\  
      u_{22} = 0 \end{array}
      & \begin{array}{l} 
      \mu= \infty\\
      (b= \pm 2i) \\
       \kappa=\infty\\
       (a=\pm 2i) \\
      \end{array} & \begin{array}{r@{\,}l} 
      S_1 = & -y \p_y-(u+y)\p_u-p\p_p-\p_q 
      \\
       S_2 = & -\p_x-y\p_y -u\p_u -p\p_p
      \\
     N_1 = & x\p_y-q\p_p
     \\
     N_2 = & x\p_u + \p_p
     \\
     N_3 = & \p_u
     \\
     N_4 = & \p_y
     \\
      & (\mbox{General pt: } y=u=p=x=q=0)
      \end{array} &  \begin{array}{c|cccccccc} 
      & S_1 & S_2 & N_1 & N_2 & N_3 & N_4\\ \hline
      S_1 & \cdot & \cdot & N_1+ N_2 & N_2 & N_3  & N_3+N_4 \\ 
      S_2 & & \cdot & N_1-N_4 & N_2-N_3 &  N_3 & N_4 \\
      N_1 & & & \cdot & \cdot & \cdot & \cdot \\ 
      N_2 &&&&\cdot& \cdot & \cdot \\ 
      N_3 &&&&&\cdot & \cdot \\ 
      N_4 &&&&&&\cdot \end{array} &
      \begin{array}{rl} 
      E/\fk : &S_2+N_1-N_2,\\
      & N_4,
      \\
      V/\fk : & N_2, S_1 
      \\
       \fk : & N_1 \end{array}
          \\
        \cmidrule{2-6}
       & \begin{array}{l} 
      u_{11} = \ln(q) x^\kappa \\
      u_{12} = 0 \\  
      u_{22} = 0 \end{array}
      & \begin{array}{l} 
      \mu= 0\\
      (b= \pm\frac1{\sqrt{2}}) \\
       \kappa\neq -1,-2,0,-3\\
       (a\neq\pm\frac1{\sqrt{2}}) \\
      \end{array} & \begin{array}{r@{\,}l} 
      S_1 = & y \p_y-\frac{x^{\kappa+2}}{(\kappa+1)(\kappa+2)}\p_u-\frac{x^{\kappa+1}}{\kappa+1}\p_p-q\p_q 
      \\
       S_2 = & -x\p_x-(\kappa+2)y\p_y -(\kappa+2)u\p_u
       \\ &-(\kappa+1)p\p_p
      \\
     N_1 = & x\p_y-q\p_p
     \\
     N_2 = & x\p_u + \p_p
     \\
     N_3 = & \p_u
     \\
     N_4 = & \p_y
     \\
      & (\mbox{General pt: } y=u=p=0, q=x=1)
      \end{array} &  \begin{array}{c|cccccccc} 
      & S_1 & S_2 & N_1 & N_2 & N_3 & N_4\\ \hline
      S_1 & \cdot & \cdot & -N_1 &\cdot & \cdot  & -N_4 \\ 
      S_2 & & \cdot & (\kappa+1)N_1 & (\kappa+1)N_2 &  (\kappa+2)N_3 & (\kappa+2)N_4 \\
      N_1 & & & \cdot & \cdot & \cdot & \cdot \\ 
      N_2 &&&&\cdot& \cdot & \cdot \\ 
      N_3 &&&&&\cdot & \cdot \\ 
      N_4 &&&&&&\cdot \end{array} &
      \begin{array}{rl} 
      E/\fk : & S_2,  N_1+N_2 
      \\
      V/\fk : & -N_1+N_4, \\
       &  S_1 +\frac1{\kappa+1}N_2-\frac1{\kappa+2}N_3
      \\
       \fk : & N_1+N_2-N_3-N_4 \end{array}
       \\
        \cmidrule{2-6}
       & \begin{array}{l} 
      u_{11} = \ln(q) x^{-2} \\
      u_{12} = 0 \\ 
      u_{22} = 0 \end{array}
      & \begin{array}{l} 
      \mu= 0\\(b= \pm\frac1{\sqrt{2}})\\
      \kappa=-2\\ (a= \pm\frac1{\sqrt{2}})
      \end{array} & \begin{array}{r@{\,}l} 
      S_1 = & y \p_y+\ln(x)\p_u+\frac{1}{x}\p_p-q\p_q 
      \\
       S_2 = & -x\p_x+p\p_p
      \\
     N_1 = & x\p_y-q\p_p
     \\
     N_2 = & x\p_u + \p_p
     \\
     N_3 = & \p_u
     \\
     N_4 = & \p_y
     \\
      & (\mbox{General pt: } y=u=p=0, q=x=1)
      \end{array} &  \begin{array}{c|cccccccc} 
      & S_1 & S_2 & N_1 & N_2 & N_3 & N_4\\ \hline
      S_1 & \cdot & N_3 & -N_1 & \cdot& \cdot  & -N_4 \\ 
      S_2 & & \cdot & -N_1 & -N_2 &  \cdot & \cdot \\
      N_1 & & & \cdot & \cdot & \cdot & \cdot \\ 
      N_2 &&&&\cdot& \cdot & \cdot \\ 
      N_3 &&&&&\cdot & \cdot \\ 
      N_4 &&&&&&\cdot \end{array} &
      \begin{array}{rl} 
      E/\fk : &S_2,  N_1+N_2 
      \\
      V/\fk : & N_1- N_4, \\
       &  S_1 -N_2+N_3
      \\
       \fk : & N_1+N_2-N_3-N_4 \end{array}
       \\
    
   \bottomrule
 \end{longtabu}

  \begin{longtabu}{ l  *5{>{$} l<{$}} }
  \caption{Classification of type D cases}\\
  \toprule
  & \mbox{\bf Model} &  \mbox{\bf Parameters} & \mbox{\bf Symmetries} & \mbox{\bf Lie algebra structure} & \mbox{\bf Abstract ILC structure}\\ 
  \midrule
  \endhead 
  D.7 & \begin{array}{l} 
  u_{11} = p^2 \\
  u_{12} = 0 \\ 
  u_{22} = \lambda q^2 \end{array} & 
 \begin{array}{l} 
  \lambda\neq 0,-1\\
  (a \ne \pm 3/4)
  \end{array} &
 \begin{array}{r@{\,}l} 
  X_1 = & \partial_x\\
  H_1 =& -2 x \partial_x +\partial_u + 2p\partial_p\\
  Y_1 =& -x^2 \partial_x +x\partial_u + (1+2xp) \partial_p \\
  X_2 =& \partial_y \\
  H_2 =& -2y\partial_y+\frac{1}{\lambda}\partial_u+2q\partial_q\\
  Y_2 =& -y^2 \partial_y + \frac{1}{\lambda}y\partial_u + (\frac{1}{\lambda}+2yq)\partial_q\\
 Z = & \partial_u\\
  & (\mbox{General pt: } x=y=u=p=q=0)
  \end{array} &  \begin{array}{c|ccccccccc}
  & X_1 & H_1 & Y_1 & X_2 & H_2 & Y_2 & Z\\ \hline
  X_1 & \cdot & -2X_1 & H_1 & \cdot & \cdot & \cdot & \cdot\\ 
  H_1 & & \cdot & -2Y_1 & \cdot & \cdot & \cdot & \cdot \\
  Y_1 & & & \cdot & \cdot & \cdot & \cdot  & \cdot\\ 
  X_2 &&&&\cdot& -2X_2 & H_2 & \cdot \\ 
  H_2 &&&&&\cdot & -2H_2 & \cdot \\ 
  Y_2 &&&&&&\cdot & \cdot \\
 Z     &&&&&&&\cdot\end{array} &
  \begin{array}{rl} 
  \fk : & H_1-Z, \lambda H_2-Z\\ 
  E/\fk : & X_1, X_2 \\
  V/\fk : & Y_1, Y_2 \end{array} 
  \\
 \cmidrule{2-6} 
    & \begin{array}{l} 
  u_{11} = p^2 \\
  u_{12} = 0 \\ 
  u_{22} = 0 \end{array} & 
 \begin{array}{l} 
  \lambda= 0\\
  (a = \pm 3/4)
  \end{array} &
 \begin{array}{r@{\,}l} 
  X_1 = & \partial_x\\
  H_1 =& -2 x \partial_x +\partial_u + 2p\partial_p\\
  Y_1 =& -x^2 \partial_x +x\partial_u + (1+2xp) \partial_p \\
  S =& -y\partial_y + q\partial_q\\
  X_2 =& \partial_y\\
  Y_2 =& y\partial_u + \partial_q\\
 Z = & \partial_u\\
  & (\mbox{General pt: } x=y=u=p=q=0)
  \end{array} &  \begin{array}{c|ccccccccc}
  & X_1 & H_1 & Y_1 & S & X_2 & Y_2 & Z\\ \hline
  X_1 & \cdot & -2X_1 & H_1 & \cdot & \cdot & \cdot & \cdot\\ 
  H_1 & & \cdot & -2Y_1 & \cdot & \cdot & \cdot & \cdot \\
  Y_1 & & & \cdot & \cdot & \cdot & \cdot  & \cdot\\ 
  S &&&&\cdot& X_2 & -Y_2 & \cdot \\ 
  X_2 &&&&&\cdot & Z & \cdot \\ 
  Y_2 &&&&&&\cdot & \cdot \\
 Z     &&&&&&&\cdot\end{array} &
  \begin{array}{rl} 
  \fk : & H_1-Z, S\\ 
  E/\fk : & X_1, X_2 \\
  V/\fk : & Y_1, Y_2 \end{array} 
  \\
 \midrule
  D.6-1 & \begin{array}{l} 
  u_{11} = p^2-q^4/4 \\
  u_{12} = q(p-q^2/2) \\ 
  u_{22} = p-q^2/2 \end{array} & 
 % no parameters
  &
 \begin{array}{r@{\,}l} 
  X_1 = & \partial_x\\
  H_1 =& -2 x \partial_x - y\partial_y+\partial_u + 2p\partial_p+q\partial_q\\
  Y_1 =& -x^2 \partial_x -xy\partial_y+(x+y^2/2)\partial_u \\
  &+ (1+2xp+yq) \partial_p + (xq+y)\partial_q\\
  X_2 =& \partial_y \\
  Y_2 =& x\partial_y-y\partial_u-q\partial_p-\partial_q\\
  Z = & -\partial_u\\
  & (\mbox{General pt: } x=y=u=p=q=0)
  \end{array} &  \begin{array}{c|cccccccc}
  & X_1 & H_1 & Y_1 & X_2 & Y_2 & Z\\ \hline
  X_1 & \cdot & -2X_1 & H_1 & \cdot & X_2 & \cdot \\ 
  H_1 & & \cdot & -2Y_1 & X_2 & -Y_2 & \cdot \\
  Y_1 & & & \cdot & Y_2 & \cdot & \cdot \\ 
  X_2 &&&&\cdot& Z & \cdot \\ 
  Y_2 &&&&&\cdot & \cdot \\
 Z     &&&&&&\cdot\end{array} &
  \begin{array}{rl} 
  \fk : & H_1+Z\\ 
  E/\fk : & X_1, X_2 \\
  V/\fk : & Y_1, Y_2 \end{array} 
  \\
 \midrule
 D.6-2 & \begin{array}{l} 
  u_{11} = p^\mu \\
  u_{12} = 0 \\ 
  u_{22} = 0 \end{array} & 
 \begin{array}{l} 
  \mu\neq 0, 1, 2, \infty\\
  (a \ne 1, 2/3, 4/3)
  \end{array} &
 \begin{array}{r@{\,}l} 
  S_1 = & -y\partial_y+q\partial_q\\
  S_2 =& -\frac{\mu-1}{\mu-2}x \partial_x - y\partial_y-u\partial_u + \frac{1}{\mu-2}p\partial_p\\
  X =& \partial_y \\
  Y =& y\partial_u+\partial_q\\
  Z = & \partial_u\\
  Z_1 =& \partial_x\\
  & (\mbox{General pt: } x=y=u=q=0,p=1)
  \end{array}
 &  \begin{array}{c|cccccccc}
  & S_1 & S_2 & X & Y & Z & Z_1\\ \hline
  S_1 & \cdot & \cdot & X & -Y & \cdot & \cdot \\ 
  S_2 & & \cdot &  X & \cdot & Z & \frac{\mu-1}{\mu-2}Z_1 \\
  X &&&\cdot& Z & \cdot & \cdot\\ 
  Y &&&&\cdot & \cdot & \cdot\\
  Z    &&&&&\cdot &\cdot\\
  Z_1 &&&&&& \cdot \end{array} &
  \begin{array}{rl} 
  \fk : & S_1\\ 
  E/\fk : & (\mu-2)S_2+Z+Z_1,X \\
  V/\fk : & S_2,Y \end{array} 
  \\
 \cmidrule{2-6} 
 & \begin{array}{l} 
  u_{11} = \exp(p) \\
  u_{12} = 0 \\ 
  u_{22} = 0 \end{array} & 
 \begin{array}{l} 
  \mu=\infty\\
  (a = 4/3)
  \end{array} &
 \begin{array}{r@{\,}l} 
  S_1 = & -y\partial_y+q\partial_q\\
  S_2 =& -x \partial_x - y\partial_y-(u-x)\partial_u + \partial_p\\
  X =& \partial_y \\
  Y =& y\partial_u+\partial_q\\
  Z = & \partial_u\\
  Z_1 =& \partial_x\\
  & (\mbox{General pt: } x=y=u=p=q=0)
  \end{array}
 &  \begin{array}{c|cccccccc}
  & S_1 & S_2 & X & Y & Z & Z_1\\ \hline
  S_1 & \cdot & \cdot & X & -Y & \cdot & \cdot \\ 
  S_2 & & \cdot &  X & \cdot & Z & -Z+Z_1 \\
  X &&&\cdot& Z & \cdot & \cdot\\ 
  Y &&&&\cdot & \cdot & \cdot\\
  Z    &&&&&\cdot &\cdot\\
  Z_1 &&&&&& \cdot \end{array} &
  \begin{array}{rl} 
  \fk : & S_1\\ 
  E/\fk : & S_2+Z_1,X \\
  V/\fk : & S_2,Y \end{array} 
  \\
 \midrule
 D.6-3 & \begin{array}{l} 
  u_{11} = \lambda p^2\frac{\sqrt{u-pq}}{u^{3/2}} \\
  u_{12} = 1+\lambda (pq-2u)\frac{\sqrt{u-pq}}{u^{3/2}}\\ 
  u_{22} = \lambda q^2\frac{\sqrt{u-pq}}{u^{3/2}} \end{array} & 
 \begin{array}{l} 
  \lambda\neq 0, \pm 1/2\\
  (a \ne 0, \pm 3)
  \end{array} &
 \begin{array}{r@{\,}l} 
  X_1 = & \partial_x\\
  H_1 =& -2x\partial_x-2u\partial_u-2q\partial_q\\
  Y_1 =& -x^2\partial_x-u\partial_y-2xu\partial_u+
  \\ & (pq-2u)\partial_p+(q^2-2xq)\partial_q \\
  X_2 =& \partial_y\\
  H_2 = & -2y\partial_y-2u\partial_u-2p\partial_p\\
  Y_2 =& -u\partial_x-y^2\partial_y-2yu\partial_u+
  \\ & (p^2-2py)\partial_p+(pq-2u)\partial_q\\
  & (\mbox{General pt: } x=y=p=q=0,u=1)
  \end{array}
 &  \begin{array}{c|cccccccc}
  & X_1 & H_1 & Y_1 & X_2 & H_2 & Y_2\\ \hline
  X_1 & \cdot & -2X_1 & H_1 & \cdot & \cdot & \cdot \\ 
  H_1 & & \cdot &  -2Y_1 & \cdot & \cdot & \cdot \\
  Y_1 &&&\cdot& \cdot & \cdot & \cdot\\ 
  X_2 &&&&\cdot & -2X_2 & H_2\\
  H_2    &&&&&\cdot &-2Y_2\\
  Y_2 &&&&&& \cdot \end{array} &
  \begin{array}{rl} 
  \fk : & H_1-H_2\\ 
  E/\fk : & X_1+\frac{2\lambda-1}{2\lambda+1}Y_2, \\[1mm] & X_2+\frac{2\lambda-1}{2\lambda+1}Y_1\\[2mm]
  V/\fk : & X_1+Y_2, X_2+Y_1 \end{array} 
  \\
 \cmidrule{2-6} 
  & \begin{array}{l} 
  u_{11} = p^2\sqrt{1-2pq} \\
  u_{12} = (pq-1)\sqrt{1-2pq}\\ 
  u_{22} = q^2\sqrt{1-2pq} \end{array} & 
 \begin{array}{l} 
  (a = \pm 3)
  \end{array} &
 \begin{array}{r@{\,}l} 
  X = & -u\partial_x-y\partial_u+p^2\partial_p+(pq-1)\partial_q\\
  H =& -x\partial_x+y\partial_y+p\partial_p-q\partial_q\\
  Y =& -u\partial_y-x\partial_u+(pq-1)\partial_p+q^2\partial_q \\
  E_1 =& \partial_x\\
  E_2 = & \partial_u\\
  E_3 =& \partial_y\\
  & (\mbox{General pt: } x=y=u=0, p=q=1)
  \end{array}
 &  \begin{array}{c|cccccccc}
  & X & H & Y & E_1 & E_2 & E_3\\ \hline
  X & \cdot & -X & H & \cdot & E_1 & E_2 \\ 
  H & & \cdot &  -Y & E_1 & \cdot & -E_3 \\
  Y &&&\cdot& E_2 & E_3 & \cdot\\ 
  E_1 &&&&\cdot & \cdot & \cdot \\
  E_2  &&&&&\cdot & \cdot \\
  E_3 &&&&&& \cdot \end{array} &
  \begin{array}{rl} 
  \fk : & H \\ 
  E/\fk : & X+E_1, Y+E_2\\
  V/\fk : & X, Y \end{array} 
  \\
 \midrule
 D.6-4 & \begin{array}{l} 
  u_{11} = 0 \\
  u_{12} = \frac{1+pq}{u}\\ 
  u_{22} = 0 \end{array} & 
 %no parameters
 &
 \begin{array}{r@{\,}l} 
  X_1 = & \partial_x\\
  H_1 =& -2x\partial_x-u\partial_u+p\partial_p-q\partial_q\\
  Y_1 =& -x^2\partial_x-xu\partial_u+(xp-u)\partial_p-xq\partial_q \\
  X_2 =& \partial_y\\
  H_2 = & -2y\partial_y-u\partial_u-p\partial_p+q\partial_q\\
  Y_2 =& -y^2\partial_y-yu\partial_u-yp\partial_p+(yq-u)\partial_q\\
  & (\mbox{General pt: } x=y=p=q=0,u=1)
  \end{array}
 &  \begin{array}{c|cccccccc}
  & X_1 & H_1 & Y_1 & X_2 & H_2 & Y_2\\ \hline
  X_1 & \cdot & -2X_1 & H_1 & \cdot & \cdot & \cdot \\ 
  H_1 & & \cdot &  -2Y_1 & \cdot & \cdot & \cdot \\
  Y_1 &&&\cdot& \cdot & \cdot & \cdot\\ 
  X_2 &&&&\cdot & -2X_2 & H_2\\
  H_2    &&&&&\cdot &-2Y_2\\
  Y_2 &&&&&& \cdot \end{array} &
  \begin{array}{rl} 
  \fk : & H_1-H_2\\ 
  E/\fk : & X_1-Y_2, X_2-Y_1\\[2mm]
  V/\fk : & Y_1, Y_2 \end{array} 
  \\
  \bottomrule
  \end{longtabu}
 
  \begin{longtabu}{ l  *4{>{$} l<{$}} }
  \caption{Classification of type III cases}\\
  \toprule
  & \mbox{\bf Model} & \mbox{\bf Symmetries} & \mbox{\bf Lie algebra structure} & \mbox{\bf Abstract ILC structure}\\ 
  \midrule
  \endhead
 III.6-1 & \begin{array}{l} 
  u_{11} = p/(x-q) \\
  u_{12} = 0 \\ 
  u_{22} = 0 \end{array} &
 \begin{array}{r@{\,}l} 
  S_1 = & -x\partial_x-u\partial_u-q\partial_q\\
  S_2 = & -y\partial_y-u\partial_u-p\partial_p\\
  N_1 =& \partial_x +y\partial_u + \partial_q\\
  N_2 =& x\partial_y + \frac{x^2}{2}\partial_u + (x-q)\partial_p\\
  N_3 =& \partial_y\\
  N_4 =& -\partial_u\\
  & (\mbox{General pt: } y=u=p=q=0, x=1)
  \end{array} & 
 \begin{array}{c|cccccccc}
  & S_1 & S_2 & N_1 & N_2 & N_3 & N_4\\ \hline
  S_1 & \cdot & \cdot & N_1 & -N_2 & \cdot & N_4 \\ 
  S_2 & & \cdot & \cdot & N_2 & N_3 & N_4 \\
  N_1 & & & \cdot & N_3 & N_4 & \cdot \\ 
  N_2 &&&&\cdot& \cdot & \cdot \\ 
  N_3 &&&&&\cdot & \cdot \\ 
  N_4 &&&&&&\cdot \end{array} & 
  \begin{array}{rl} 
  \fk : & S_2 \\ 
  E/\fk : & S_1, N_3 \\
  V/\fk : & S_1+N_1, N_2 - N_3+\frac{1}{2}N_4 \end{array} 
 \\
  \midrule
  III.6-2 & \begin{array}{l} 
  u_{11} = 2q(2p - q u) \\
  u_{12} = q^2 \\ 
  u_{22} = 0 \end{array} & \begin{array}{r@{\,}l} 
  X = & \partial_x\\
  H =& -2 x \partial_x - y \partial_y + u \partial_u + 3p\partial_p+2q\partial_q\\
  Y =& -x^2 \partial_x - xy\partial_y + (ux+y) \partial_u \\
  & + (3px+u+qy) \partial_p + (2qx+1) \partial_q\\
  S =& -y\partial_y - u\partial_u - p \partial_p\\
  N_1 =& \partial_y\\
  N_2 =& x \partial_y - \partial_u - q \partial_p\\
  & (\mbox{General pt: } x=y=u=q=0, p=1)
  \end{array} &  \begin{array}{c|cccccccc}
  & X & H & Y & S & N_1 & N_2\\ \hline
  X & \cdot & -2X & H & \cdot & \cdot & N_1\\ 
  H & & \cdot & -2Y & \cdot & N_1 & -N_2 \\
  Y & & & \cdot & \cdot & N_2 & \cdot \\ 
  S &&&&\cdot& N_1 & N_2 \\ 
  N_1 &&&&&\cdot & \cdot \\ 
  N_2 &&&&&&\cdot \end{array} &
  \begin{array}{rl} 
  \fk : & H + 3 S\\ 
  E/\fk : & X - N_2, N_1 \\
  V/\fk : & S, Y \end{array} 
  \\
  \bottomrule
  \end{longtabu}

  \begin{longtabu}{ l  *2{>{$} l<{$}} }
   \caption{Basis change from Cartan reduced basis to adapted Lie algebra basis}
   \label{Cartan2adapted}
\\
   \toprule
  & \mbox{\bf Parameters change} &\mbox{\bf Basis change}\\ 
  \midrule
   \endfirsthead
   \caption[]{Basis change from Cartan reduced basis to adapted Lie algebra basis (continued)}
  \\
   \toprule
  & \mbox{\bf Parameters change} &\mbox{\bf Basis change}\\ 
  \midrule
  \endhead 
       \mbox{N.8} & & 
       S_1=-\frac12 e_7-\frac32e_8,
       S_1=-\frac12 e_8-\frac32e_7,
        N_1 = \frac12e_4,
        N_2= -2 e_1, 
        N_3 = -e_6,
        N_4 = -\frac12 e_3,
        N_5 = -e_5,
        N_6 = -2 e_2
            \\
             \midrule
   \mbox{N.7-1} &
  \kappa=-\frac32 + \frac{3a}{2\sqrt{a^2+4}} 
   ,\,\,\kappa\neq -1,-2,\infty
  &
  S_1 = -\frac12 e_6 ,
  S_2= \frac{1}{\sqrt{a^2+4}} e_1 - \frac{3a}{2\sqrt{a^2+4}} e_6 ,
  N_1 = -\frac{1}{\sqrt{a^2+4}(3a+\sqrt{a^2+4})}\left( e_2 + (2a^2-1)e_4 + 2a e_7 \right),
  N_2 = -\frac{1}{2(2a^2-1)}\left(e_2 + \frac{a+\sqrt{a^2+4}}2 e_7\right),
  \\
  &  a=\frac{2\kappa+3}{\sqrt{-\kappa^2-3\kappa}},\,\,  a\neq\pm\frac1{\sqrt2},\pm 2 i 
  & 
  N_3 = \frac{3a-\sqrt{a^2+4}}{8(a^2+4)^{\frac32}}\left(\frac{a+\sqrt{a^2+4}}2 e_3 + e_5 \right),
  N_4 = \frac{3a-\sqrt{a^2+4}}{8(a^2+4)^{\frac32}}\left(\frac{a-\sqrt{a^2+4}}2 e_3 +e_5 \right),
  N_5 = -\frac{1}{2(2a^2-1)}\left( e_2 + \frac{a-\sqrt{a^2+4}}2 e_7 \right)
  \\
   \cmidrule{2-3}
   &
    \kappa=-1
    &
    S_1 = -\frac12 e_6 ,
    S_2= \frac{\sqrt2}3 e_1 - \frac12 e_6 ,
    N_1 = \frac29 e_2 - e_4 - \frac{\sqrt2}9 e_7,
    N_2 = -\frac19 e_2 - \frac{\sqrt2}{9} e_7,
    \\
    &  a=\frac{1}{\sqrt2}
    & 
    N_3 = \frac{\sqrt2}9 e_3 + \frac19 e_5,
    N_4 =-\frac{1}{9\sqrt2}e_3 + \frac19 e_5,
    N_5 = -\frac19 e_2 + \frac{1}{9\sqrt2}e_7
   \\
   \cmidrule{2-3}
     &
      \kappa=\infty
      &
      S_1 = -\frac12 e_6 ,
      S_2= \frac{i}3 e_1 -  e_6 ,
      N_1 = -\frac{i}3 e_2 +3i e_4 - \frac{4}3 e_7,
      N_2 = - \frac12 e_7,
      \\
      &  a=2i
      & 
      N_3 = -\frac{3i}{2} e_3,
      N_4 =-\frac{i}{2}e_3 + \frac{1}{2} e_5,
      N_5 = -\frac{i}6 e_2 - \frac{1}{6}e_7
     \\
   \midrule
   \mbox{N.7-2} & & X=- e_1 -\frac12e_2+e_4+\frac12e_5+2 e_6+\frac12 e_7,
    H=-2e_4-\frac12 e_5-2e_6 -\frac12e_7, Y=-e_4, \\
    & &
    N_1 = \frac12 e_3 , 
    N_2 = e_3 +\frac12 e_5 -\frac12 e_7,
    N_3 = -e_2+e_3+e_5-e_7 ,
    N_4 = -\frac12 e_5 -\frac12 e_7
  \\
  \midrule
   \mbox{N.6-1}&
  \mu=\frac{2a^2-1}{a^2+1}\neq 0,1 ,
  &
     S = \frac1a e_4,
     N_1 =  -\frac{a}{a^2+1}e_1-\frac{a^2+1}{a(2a^2-1)(a^2-2)}e_3 + \frac1a e_4 +\frac{2(a^2-1)}{(2a^2-1)(a^2-2)}e_6,
     N_2 = -\frac{1}{(2a^2-1)(a^2-2)}\left(\frac{a^2+1}{a}e_3 + a^2 e_6\right),
  \\
  & a^2=\frac{1+\mu}{2-\mu}\neq \frac12,2 & 
      N_3 = \frac{1}{(2a^2-1)(a^2-2)}\left(
      \frac{a^3}{a^2+1}e_2 + \frac{a^2+2}{a}e_3 -e_5+\frac{1}{a^2+1}e_6\right),
      N_4 = \frac{1}{(2a^2-1)(a^2-2)}\left(
      \frac{2a}{a^2+1}e_2 -\frac2a e_3 + 2 e_5 -\frac{2}{a^2+1}e_6\right),
   \\ & &   
      N_5 = \frac{1}{(2a^2-1)(a^2-2)}\left(
      \frac{a^3}{a^2+1}e_2 - a e_3 - e_5-\frac{2a^2+1}{a^2+1}e_6\right)
   \\
 \cmidrule{2-3}
 &
  \mu= 0 ,\,\, a^2= \frac12
  &
    S = \frac{\sqrt2}{3} e_3 + \sqrt2 e_4 - \frac19 e_6,
    N_1 = -\frac{\sqrt2}3 e_1 + \sqrt2 e_4 - \frac23e_6,
    N_2 = \frac{2\sqrt2}{3} e_3 - \frac29 e_6,
    N_3 = -\frac{2\sqrt2}{27} e_2 - \frac{10\sqrt2}9 e_3 + \frac49 e_5 - \frac8{27} e_6,
  \\
  & &  
   N_4 = -\frac{8\sqrt2}{27} e_2 + \frac{8\sqrt2}9 e_3 - \frac89 e_5 + \frac{16}{27} e_6,
     N_5 = \frac{\sqrt2}{27} e_2 -\frac{\sqrt2}{9} e_3  - \frac29 e_5 - \frac8{27} e_6
    \\
 \cmidrule{2-3}
 &
  \mu= 1 ,\,\, a^2= 2
  &
     S = \frac{1}{\sqrt2}e_4,
     N_1 = -\frac{\sqrt2}3 e_1 + \frac{1}{\sqrt2} e_4 +\frac13e_6,
     N_2 = -\frac{1}{3\sqrt2} e_3 + \frac29 e_6,
     N_3 = \frac{2\sqrt2}{27}e_2 + \frac{2\sqrt{2}}9 e_3 - \frac19e_5 + \frac1{27} e_6,
 
  \\
  & & 
     N_4 = \frac{2\sqrt{2}}{27} e_2 - \frac{\sqrt{2}}9 e_2 + \frac29 e_5 - \frac2{27} e_6,
     N_5 = \frac{2\sqrt{2}}{27} e_2 - \frac{\sqrt{2}}9 e_3 - \frac19 e_5 - \frac5{27} e_6
   \\  
  \midrule
  \mbox{N.6-2} &
  \mu=\frac12+\frac{3b}{2\sqrt{b^2+4}} ,\,\, b=\frac{2\mu-1}{\sqrt{-\mu^2+\mu+2}},
  & S_1=\frac{-1}{\sqrt{b^2+4}}e_4,
    S_2=\frac{-1}{\sqrt{a^2+4}} e_1 + \frac{1}{(1-2b^2)\sqrt{a^2+4}}e_3 + \frac{2b}{(1-2b^2)\sqrt{a^2+4}}e_6, \\
  & \kappa=-\frac32 + \frac{3a}{2\sqrt{a^2+4}} 
   ,\,\, a=\frac{2\kappa+3}{\sqrt{-\kappa^2-3\kappa}}, &
    N_1=\frac{1}{(2b^2-1)(a^2+4)}\left(\frac{-b-\sqrt{b^2+4}}2 e_2 + \frac{a+\sqrt{a^2+4}}2 e_3 - e_5 + \frac{(a+\sqrt{a^2+4})(\sqrt{b^2+4}+b)}4 e_6\right), 
  \\ & 
   b\neq\pm\frac{1}{\sqrt{2}},\pm 2 i,\mu\neq 0,1,\infty; &
    N_2=\frac{1}{(2b^2-1)(a^2+4)}\left(\frac{b-\sqrt{b^2+4}}2 e_2 - \frac{a+\sqrt{a^2+4}}2 e_3 + e_5 + \frac{(a+\sqrt{a^2+4})(\sqrt{b^2+4}-b)}4 e_6\right), 
    \\ 
  &  a\neq\pm 2 i,\kappa\neq \infty
   & N_3=\frac{1}{(2b^2-1)(a^2+4)}\left(\frac{b-\sqrt{b^2+4}}2 e_2 - \frac{a-\sqrt{a^2+4}}2 e_3 + e_5 + \frac{(a-\sqrt{a^2+4})(\sqrt{b^2+4}-b)}4 e_6\right), 
   \\ & &
      N_4=\frac{1}{(2b^2-1)(a^2+4)}\left(\frac{-b-\sqrt{b^2+4}}2 e_2 + \frac{a-\sqrt{a^2+4}}2 e_3 - e_5 + \frac{(a-\sqrt{a^2+4})(\sqrt{b^2+4}+b)}4 e_6\right), 
      \\
   \cmidrule{2-3}
  &
  \mu=\infty ,\,\, b=2i,\,\,  \kappa=-\frac32 + \frac{3a}{2\sqrt{a^2+4}} 
  & S_1 = \frac{i}{3}e_4, 
      S_2 = \frac{-1}{\sqrt{a^2+4}} \left(e_1-\frac19 e_3 - \frac{4i}{9} e_6\right),
      N_1 = -\frac{i}{3(a^2+4)}\left(e_2 -  \frac{a+\sqrt{a^2+4}}{2} e_6\right) ,
      N_2 = \frac{1}{9(a^2+4)}\left(i e_2  -\frac{a+\sqrt{a^2+4}}{2} e_3 +e_5 - \frac{i(a+\sqrt{a^2+4})}2 e_6\right) 
  \\
  & 
   a=\frac{2\kappa+3}{\sqrt{-\kappa^2-3\kappa}}, a\neq\pm 2 i,\kappa\neq \infty
   &
   N_3 = \frac{1}{9(a^2+4)}\left(ie_2 + \frac{\sqrt{a^2+4}-a }{2}e_3  + e_5 + \frac{i(\sqrt{a^2+4}-a)}2 e_6\right) ,
   N_4 = -\frac{i}{3(a^2+4)}\left(e_2 + \frac{\sqrt{a^2+4}-a}{2} e_6\right) 
   \\
  \cmidrule{2-3} 
  &
  \mu=\infty ,\,\, b=\pm 2i,
  & 
     S_1 = -\frac{i}3 e_4,
     S_2 = -\frac{i}{3} e_1 + \frac{i}{27}e_3 +  \frac{4}{27}e_6 , 
     N_1 = -\frac19 e_6,
     N_2 = \frac{i}{27}e_3 +\frac1{27} e_6,
     N_3 = -\frac{i}{81} e_2 + \frac{i}{81} e_3 + \frac{1}{81} e_5 +\frac{1}{81} e_6,
     N_4 = \frac{i}{27}e_2 - \frac{1}{27} e_6
  \\
  & \kappa=\infty 
   ,\,\, a=\pm 2i
     \\
   \cmidrule{2-3}
  &
  \mu=0 ,\,\, b=\frac{1}{\sqrt{2}}, \,\, \kappa=-\frac32 +\frac{3a}{2\sqrt{a^2+4}} 
  & 
     S_1 = \frac{\sqrt{2}}{3}e_4 + \frac{2}{9(2a^2-1)}\left( \sqrt{2}  e_2 -2a e_3 +e_5 - 2\sqrt{2}a e_6 \right),
     S_2 = -\frac{1}{\sqrt{a^2+4}}\left( e_1 + \frac{\sqrt{2}}3 e_6 \right),
     \\
  &
    a=\frac{2\kappa+3}{\sqrt{-\kappa^2-3\kappa}}, a\neq\pm 2 i,\kappa\neq \infty 
   &
   N_1 = -\frac{1}{9(a^2+4)}\left(\sqrt{2} e_2 - 2 e_5 + \left( a+ \sqrt{a^2+4} \right) \left(e_3 - \frac{1}{\sqrt2} e_6\right)\right),
   N_2 =   -\frac{1}{9(a^2+4)}\left(2\sqrt{2} e_2 + 2 e_5 + \left( \sqrt{a^2+4}+a \right) \left(e_3 + {\sqrt2} e_6\right)\right),
   \\
    & &
   N_3 =   -\frac{1}{9(a^2+4)}\left(2\sqrt{2} e_2 + 2 e_5 + \left( \sqrt{a^2+4}- a\right) \left(e_3 + {\sqrt2} e_6\right)\right),
   N_4 =   -\frac{1}{9(a^2+4)}\left(\sqrt{2} e_2 - 2 e_5 + \left( a- \sqrt{a^2+4} \right) \left(e_3 - \frac1{\sqrt2} e_6\right)\right)
   \\
  \cmidrule{2-3}
  &
  \mu=0 ,\,\, b=\frac{1}{\sqrt{2}},
  & 
     S_1 = \frac{2\sqrt2}{27} e_3 + \frac{\sqrt2}{3} e_4 + \frac{4}{27} e_6,
     S_2 = -\frac{\sqrt2}{3}e_1- \frac29 e_6,
     N_1= - \frac{2\sqrt2}{81}e_2-  \frac{2\sqrt2}{81} e_3 +  \frac{4}{81} e_5 +  \frac{2}{81} e_6 ,
     N_2 = - \frac{4\sqrt2}{81} e_2 +  \frac{2\sqrt2}{81} e_3 -  \frac{4}{81} e_5 + \frac{4}{81} e_6 ,
  \\
  & \kappa=-2, a=-\frac1{\sqrt2}
   &
     N_3 =-\frac{4}{81} \left(\sqrt2 e_2 + \sqrt2 e_3 + e_5 + 2 e_6\right) ,
     N_4 = -\frac{2\sqrt2}{81} \left( e_2 - 2 e_3 - \sqrt2 e_5 +\sqrt2 e_6\right) ,
   \\
  \midrule
  \mbox{D.7} &
  \lambda = \frac{3+4a}{3-4a}, a\ne\pm \frac{3}{4},
  &
 X_1=\frac{2}{4a-3}e_1,
 H_1 = -\frac{1}{4a-3}(2e_5+2(2a-1)e_6-(2a-3)e_7),
 Y_1 = e_3,
 \quad
 X_2 = -e_2,
 H_2 = \frac{1}{4a+3}(2e_5-2(2a+1)e_6-(2a+3)e_7),
 Y_2 = \frac{2}{4a+3}e_4,
 \\
 &
  a = \frac{3}{4}\frac{\lambda-1}{\lambda+1}, \lambda\ne 0, -1 & 
 Z = -\frac{1}{4a-3}(2e_5+e_6+2ae_7)
 \\
 \cmidrule{2-3}
 & a=\pm \frac{3}{4}, \lambda=0 & 
 X_1 = -\frac{1}{3}e_4, H_1 = -\frac{1}{3}e_5 +\frac{5}{6}e_6+\frac{3}{4}e_7, Y_1 = e_2, 
 \quad S = \frac{1}{2}(e_6-e_7), X_2 = -\frac{1}{3}e_3, Y_2 = e_1, 
 \quad Z = -\frac{1}{3}e_5-\frac{1}{6}e_6-\frac{1}{4}e_7
 \\
 \midrule
 \mbox{D.6-1} &
 &
 X_1 =\frac{1}{\sqrt{2}}e_1, H_1 = -\frac{1}{2}e_5+\frac{1}{8}e_6, Y_1 = \frac{1}{\sqrt{2}}e_3,\quad
 X_2 = \sqrt{2}e_2, Y_2 = -\sqrt{2}e_4, Z = \frac{1}{2}e_5 + \frac{3}{8}e_6
 \\
 \midrule
 \mbox{D.6-2} 
 & \mu = \frac{6(a-1)}{3a-4},  a \ne 1, 2/3, 4/3  &
 S_1 = \frac{1}{2}e_6, S_2 = -\frac{3a-2}{12}(2(a-1)e_4+3e_6), 
 X = e_1,
 Y = -\frac{(a-1)(3a-2)}{(3a-4)^2}e_3,
 Z = -\frac{1}{(3a-4)^2}\Big(6e_2+\frac{2(a-1)(3a-2)}{3}e_4-(a-1)(3a-2)e_5-\frac{9a^2-15a-2}{4}e_6\Big),
 \\
 & a = \frac{2}{3}\frac{2\mu-3}{\mu-2}, \mu\ne 0, 1, 2, \infty &
 Z_1 =  \frac{1}{(3a-4)^2}\Big(3(3a-2)e_2+\frac{(a-1)(3a-2)^2}{3}e_4-(a-1)(3a-2)e_5+\frac{9(a-1)(3a-2)}{4}e_6\Big)
 \\
 \cmidrule{2-3}
 & a=4/3, \mu=\infty & 
 S_1=-\frac{1}{2}e_6,
 S_2=-\frac{1}{9}e_4-\frac{1}{2}e_6,
 X=-\frac{1}{6}e_3,
 Y=e_1,
 Z=\frac{3}{2}e_2+\frac{1}{9}e_4-\frac{1}{6}e_5+\frac{3}{8}e_6,
 Z_1=\frac{3}{2}e_2+\frac{1}{9}e_4+\frac{3}{2}e_6
 \\
 \midrule
 \mbox{D.6-3} 
 & \lambda = \frac{3}{2\sqrt{9-a^2}}, a\ne 0, \pm3 &
 X_1 = \lambda e_1 +\frac{\lambda(2\lambda-1)}{\sqrt{4\lambda^2-1}}e_4,
 H_1 =\frac{4\lambda}{3}e_5+e_6,
 Y_1 =-\frac{2}{3}\sqrt{4\lambda^2-1}e_2-\frac{2}{3}(2\lambda+1)e_3,
 \\
 & a = \frac{3\sqrt{4\lambda^2-1}}{2\lambda}, \lambda\ne 0, \pm1/2 &
 X_2 =\frac{2}{3}\sqrt{4\lambda^2-1}e_2+\frac{2}{3}(2\lambda+1)e_3,
 H_2 = \frac{4\lambda}{3}e_5-e_6,
 Y_2 =  -\lambda e_1 -\frac{\lambda(2\lambda+1)}{\sqrt{4\lambda^2-1}}e_4
 \\
 \cmidrule{2-3}
  & a=\pm 3 &  
 X = -e_4, H=e_6, Y=-\frac{2}{3}e_3, \quad
 E_1 = e_1+e_4, E_2 = -\frac{2}{3}e_5, E_3=\frac{2}{3}(e_2+e_3)
 \\
 \midrule
 \mbox{D.6-4}
 & 
 %no parameters
 &
 X_1 = -2e_1+\frac{1}{3}e_4, H_1 = \frac{2}{3}e_5+e_6, Y_1 = \frac{1}{3}e_3, \quad
 X_2 = -2e_2+\frac{1}{3}e_3, H_2 = \frac{2}{3}e_5 - e_6, Y_2 = \frac{1}{3}e_4
 \\
 \midrule
  \mbox{III.6-1} & & S_1=2e_1+\frac{5}{4}e_6, S_2=-\frac{1}{2}e_6, N_1=-2e_1+e_4-\frac{9}{4}e_6, N_2=-\frac{1}{2}e_2+\frac{1}{8}e_3+\frac{1}{2}e_5, N_3=-e_2, N_4=-e_2+\frac{3}{4}e_3-e_5
  \\
  \midrule
  \mbox{III.6-2} & & X=2 e_1 + e_5, H=\frac{3}{2} e_3 + \frac{5}{2} e_6, Y=-2 e_4, S=-\frac{1}{2} (e_3+e_6),N_1=-\frac{1}{2} e_2,N_2=e_5
    \\
 \bottomrule
 \end{longtabu}

  \begin{table}[ht]
  \caption{Basis change which reflects redundancy in parameters}
 \begin{tabu}{ l  *2{>{$} l<{$}} }
     \toprule
  & \mbox{\bf Parameters change} &\mbox{\bf Basis change}\\ 
     \midrule
      \mbox{N.7-1} 
      & a\to -a,(\kappa \to -\kappa-3 ), 
     & (e_1,e_2,e_3,e_4,e_5,e_6)\to (-e_1,e_2,-e_3,e_4,e_5,-e_6)
      \\ 
   \midrule
    \mbox{N.6-1} 
    & a\to -a 
   & (e_1,e_2,e_3,e_4,e_5,e_6)\to (-e_1,-e_2,-e_3,-e_4,e_5,e_6)
    \\ 
   \midrule
   \mbox{N.6-2} 
    & a\to -a,(\kappa \to -\kappa-3 ),
   & (e_1,e_2,e_3,e_4,e_5,e_6)\to (-e_1,e_2,-e_3,e_4,e_5,-e_6) \\
   & b\to -b,(\mu \to -\mu+1 )
   &(e_1,e_2,e_3,e_4,e_5,e_6)\to (e_1,-e_2,e_3,-e_4,e_5,-e_6)
 \\
   \midrule
   \mbox{D.7} 
    & a\to -a, (\lambda \to 1/\lambda )
   & (e_1,e_2,e_3,e_4,e_5,e_6,e_7)\to (e_2,e_1,e_4,e_3,e_5,e_6,-e_7) 
   \\
   \midrule
   \mbox{D.6-3}
   & a\to -a, (\lambda\to -\lambda)
   & (e_1,e_2,e_3,e_4,e_5,e_6)\to (e_2,-e_1,e_4,-e_3, e_5, -e_6)\\
   \bottomrule
 \end{tabu}
 \end{table}
 
 \begin{table}[ht]
 \caption{Duality}\label{Ap:Dual}
 \begin{tabu}{ l  *2{>{$} l<{$}} }
   \toprule
  & \mbox{\bf Parameters change} &\mbox{\bf Basis change}\\ 
   \midrule
      \mbox{N.8} 
      & \mbox{(self-dual)}
     & (e_1,e_2,e_3,e_4,e_5,e_6,e_7,e_8)\to  (e_4,e_3,e_2,e_1,-e_5,-e_6,e_8,e_7) \\
     \midrule
     \mbox{N.7-2} 
     & \mbox{(self-dual)}
     & (e_1,e_2,e_3,e_4,e_5,e_6,e_7)\to  (e_4,e_3,e_2,e_1,-e_5,-e_6,-e_7) \\
     \midrule
    \mbox{N.6-1} 
    & \mbox{(self-dual)}
   & (e_1,e_2,e_3,e_4,e_5,e_6)\to (\frac{a^2+1}{a^2}e_4,\frac{a^2+1}{a^2}e_3,\frac{a^2}{a^2+1}e_2,\frac{a^2}{a^2+1}e_1,-e_5,-e_6) \\
   \midrule
   \mbox{N.6-2} 
    & (a,b)\to(b,a), (\mu,\kappa) \to (\kappa+2,\mu-2)\phantom{aaa}
   & (e_1,e_2,e_3,e_4,e_5,e_6)\to (e_4,e_3,e_2,e_1,-e_5,-e_6) \\
   \midrule
   \mbox{D.7}
   & \mbox{(self-dual)}
   & (e_1,e_2,e_3,e_4,e_5,e_6,e_7)\to (e_3,e_4,e_1,e_2,-e_5,-e_6,-e_7) \\
   \midrule
   \mbox{D.6-1}
   & \mbox{(self-dual)}
   & (e_1,e_2,e_3,e_4,e_5,e_6)\to (e_3,e_4,e_1,e_2,-e_5,-e_6) \\
   \midrule
   \mbox{D.6-2}
   &\mbox{(self-dual)}
   & (e_1,e_2,e_3,e_4,e_5,e_6)\to (e_3,-\frac{(a-1)(3a-2)}{9}e_4,e_1,-\frac{9}{(a-1)(3a-2)}e_2,-e_5,-e_6) \\
   \midrule
   \mbox{D.6-3}
   &\mbox{(self-dual)}
   &  (e_1,e_2,e_3,e_4,e_5,e_6)\to (e_3,e_4,e_1,e_2,-e_5,-e_6) \\
   \bottomrule
 \end{tabu}
 \end{table}

   \end{landscape}
\end{document}